\mag=\magstep1

\documentclass[10pt,a4paper,dviout]{amsart} 
\usepackage{amssymb,latexsym} 
\usepackage{color} 
\usepackage{graphicx}
\usepackage{comment}
\usepackage{datetime}
\usepackage{comment}

\input xy
\xyoption{all}

\usepackage{amsmath, amssymb, eucal, amscd, color, amsthm}

\def\R{\Bbb R}
\def\si{\sigma}

\def\cal{\mathcal}
\def\co{{\cal O}}
\def\Q{\Bbb Q}
\def\part{\partial}

\def\we{\wedge}
\def\e{\epsilon}
\def\dis{\displaystyle}

\def\P{\mathbb P}
\def\C{\mathbb C}
\def\s{{\square^n}}

\def\sd{\operatorname{sd}}
\def\id{\operatorname{id}}
\def\codim{{{\rm codim}\,}}
\def\ov{\overline}
\def\Max{{\rm Max}}
\def\p1{\prec}
\def\sing{{\operatorname{sing}}}
\def\<{\langle}
\def\>{\rangle}

\def\sd{\operatorname{sd}}

\def\id{\operatorname{id}}
\newcommand{\sign}{\operatorname{sign}}
\newcommand{\Alt}{\operatorname{Alt}}
\def\res{\operatorname{Res}}

\newtheorem{theorem}{Theorem}[section] 
\newtheorem{proposition}[theorem]{Proposition} 
 
\newtheorem{definition}[theorem]{Definition} 
 
\newtheorem{remark}[theorem]{Remark} 
\newtheorem{lemma}[theorem]{Lemma} 
\newtheorem{lemma-definition}[theorem]{Lemma-Definition} 
\newtheorem{proposition-definition}[theorem]{Proposition-Definition}

\newtheorem{notation}[theorem]{Notation}

\newcommand{\CC}{{\mathbb{C}}}
\newcommand{\PP}{{\mathbb{P}}}
\newcommand{\QQ}{{\mathbb{Q}}}
\newcommand{\RR}{{\mathbb{R}}}
\newcommand{\ZZ}{{\mathbb{Z}}}

\newcommand\Spec{\operatorname{Spec}}

\newcommand\sgn{\operatorname{sgn}}

\newcommand{\al}{\alpha}  \newcommand{\ga}{\gamma} 
 \newcommand{\eps}{\epsilon} \newcommand{\la}{\lambda}
\newcommand{\Ga}{\Gamma}

\newcommand{\BP}{{\mathbb{P}}}

\newcommand{\sq}{\square} 

\newcommand{\ot}{\otimes}

\newcommand{\vphi}{{\varphi}}

\newcommand{\injto}{\hookrightarrow}

\newcommand{\mapright}[1]{%
  \smash{\mathop{%
    \hbox to 1cm{\rightarrowfill}}\limits^{#1} } } 

\newcommand{\smapr}[1]{%
  \smash{\mathop{%
    \hbox to 0.5cm{\rightarrowfill}}\limits^{#1} } } 
\newcommand{\maprb}[1]{%
  \smash{\mathop{%
    \hbox to 1cm{\rightarrowfill}}\limits_{#1} } } 
\newcommand{\mapleft}[1]{%
  \smash{\mathop{%
    \hbox to 1cm{\leftarrowfill}}\limits^{#1} } }

\newcommand{\maplb}[1]{%
  \smash{\mathop{%
    \hbox to 1cm{\leftarrowfill}}\limits_{#1} } }

\def\alt{\operatorname{alt}}

\makeatletter
  
  \@addtoreset{equation}{section}
 \makeatother

\makeindex
\usepackage{version}	
\begin{document}

\title{
Integrals of logarithmic forms on semi-algebraic sets and a 
generalized Cauchy formula \\
Part II: Generalized Cauchy formula 
} 
\author{Masaki Hanamura, Kenichiro Kimura and Tomohide Terasoma}

\maketitle
\begin{abstract}
This paper is the continuation of the paper arXiv:1509.06950, which is Part I under the same title. In this paper, we prove a generalized Cauchy formula for the integrals of logarithmic forms on products of projective lines, and give an application to the construction of Hodge realization of mixed Tate motives. 
\end{abstract}

\vskip 0.3cm


\setcounter{tocdepth}{3}

\markboth
{Integrals of logarithmic forms}
{Hanamura, Kimura, Terasoma}

\thispagestyle{empty}
\setcounter{tocdepth}{1}
\tableofcontents

\section{Introduction}

This paper is the continuation 
of the paper \cite{part I}, which is  Part I under the same title.
In this paper, we prove a generalized Cauchy formula for the integrals
of logarithmic forms on products $P^n=(\PP^1_\CC)^n$ of projective lines
$\PP^1_\CC$.
As an application, we define a variant of the Hodge realization functor for
the category of mixed Tate motives
defined by Bloch and Kriz \cite{BK}.
In the sequel to this paper, we prove that our construction
coincides with the original one defined by Bloch and Kriz.
The motivation of our series of papers is to 
understand  the Hodge realization 
functor via integral of logarithmic differential forms.

Before going into the detail,
we describe a simplest example of
the generalized Cauchy formula. 
Let $\omega_2=(2\pi i)^{-2}\dfrac{dz_1}{z_1}\wedge \dfrac{dz_2}{z_2}$,
$\omega_1=(2\pi i)^{-1}\dfrac{dz_2}{z_2}$
be a holomorphic two resp. one form on 
$(\CC-\{0\})^2$ resp. 
$\CC-\{0\}$. Let $0<a<b$ be
real numbers, and let 
$$\sigma=\{(z_1,z_2)\,|\,z_1=x_1+y_1i, z_2=x_2+x_1i,\,\,a\leq x_2\leq b,\quad x_1^2+y_1^2\leq 1\}$$ be a $3$-chain
in $\CC^2$.
Its 
topological boundary is denoted by $\delta\sigma$.
Then we have the following identity, called the generalized Cauchy formula:
\begin{align}
\label{easiest example} 
\int_{\delta \sigma}\omega_2
=&
\int_{\{x_1^2+y_1^2=1\}}\omega_2
+\int_{(\{x_2=b\})-( \{x_2=a\})
}\omega_2
\\
\nonumber
=&\int_{[a,b]}
\omega_1
=\int_{\sigma\cap (\{0\}\times \CC)}
\omega_1.
\end{align}
Although the differential form $\omega_2$ is
not defined on $\delta\sigma\cap (\{0\}\times\CC)$,
the integral is defined as an improper integral.
The important property of the chain $\si$ is that both $\si$ and $\delta\si$ meet
the pole of the form $\omega_2$ properly i.e. in correct dimensions.  

In this paper, we generalize this formula as follows. 
Let $P^n=(\PP^1_\CC)^n$ with the coordinates $(z_1, \dots, z_n)$ , and  
$\bold D^n$ be the divisor of $P^n$    
defined by $\prod_i^n(z_i-1)=0$. The {\it faces} of $P^n$ are
irreducible components of the polar divisors of the
differential form
$$
\omega_n=\frac1{(2\pi i)^n}\frac{dz_1}{z_1}\we \cdots  \we \frac{dz_n}{z_n}.
$$ 
and their intersections. More explicitly, They are intersections of $\{z_i=0\}$ and $\{z_i=
\infty\}$.  In \S \ref{sec:admissible chain complex} we define a certain complex of semi-algebraic chains
$AC_{\bullet}(P^n,\bold D^n; \QQ)$ which we call the complex of $\delta$-admissible chains. Roughly speaking, elements of the complex
$AC_\bullet$ are simplicial semi-algebraic chains $\ga$ such that both
$\ga$ and $\delta\ga$ meet the faces properly.


In \S \ref{sec:Face map and cubical chain}, 
we define a face map
$$
\partial_{i,\alpha}:
AC_i(P^n,\bold D^n;\QQ)\to
AC_{i-2}(P^{n-1},\bold D^n;\QQ)
\quad (1\leq i\leq n, \alpha=0,\infty)
$$
with respect to the hypersurface $H_{i,\alpha}=\{z_i=\alpha\}$. Roughly this map is
taking the intersection with the face $H_{i,\al}$. More precisely, it is defined
by the cap product
with a Thom cocycle $T$ of the face $H_{i,\al}$.
The cubical differential  $\part$ is defined to be the alternating sum of the face maps
$\part_{i,\al}$.  We point out that this face map is compatible with the algebraic face map in the following sense.
If $V$ be a closed subvariety of $P^n$ which meets all the faces properly, then the set of complex points of $V$ defines
an element of $AC_{2d}(P^n, \bold D^n; \QQ)$ where $d=\dim V$. Let $\part_{i,\al}^{alg}V$ be the intersection of $V$ with the subvariety
$H_{i,\al}$ of $P^n$ considered with multiplicity. The algebraic cycle $\part_{i,\al}^{alg}V$ defines an element 
of $AC_{2d-2}(P^{n-1},\bold D^{n-1};\QQ).$ In Proposition \ref{algtopeq} we show that
$$\part_{i,\al} V=\part_{i,\al}^{alg} V$$
in   $AC_{2d-2}(P^{n-1},\bold D^{n-1};\QQ).$

In \S \ref{sec:Generalized Caychy formula} we prove a generalized Cauchy formula.
Let $\ga=\sum a_\si\si$ be an element of $AC_{n}(P^n,\bold D^n;\QQ)$,
where $\sigma$'s are $n$-simplices in a triangulation $K$ of $P^n$
and $a_{\sigma}\in \QQ$. 
By a  result of Part I, the integral
$\displaystyle \int_\si \omega_n$ 
converges absolutely if $\si$ is admissible (Theorem \ref{convadmiss}) and
we define a homomorphism $I_{n}$ as the following integration:
\begin{align*}
&I_{n}:AC_{n}(P^n,\bold D^n;\QQ)\to \CC:\sum a_\si\si
\mapsto 
(-1)^{\frac{n(n-1)}{2}}\sum_{\sigma}\int_\si a_\si\omega_n.
\end{align*}
Then the generalized Cauchy formula 
(Theorem \ref{th:generalized Cauchy formula}) asserts that, for an element
$\ga\in AC_{n+1}(P^n,\bold D^n;\QQ)$,
 the equality
\begin{equation}
\label{commutative diagram for generalized cauchy formula}
I_{n-1}(\partial \gamma)+(-1)^nI_n(\delta\gamma)=0
\end{equation}
holds. When $\ga$ is a unit disc on $\CC$ whose center is the origin, 
the equality (\ref{commutative diagram for generalized cauchy formula}) is  Cauchy's
integral formula $\dis 1=\frac1{2\pi i}\int_{|z_1|=1}\frac{dz_1}{z_1}.$
If the chain  $\gamma$ does not meet any of the faces,
then we have $\part\ga=0$, and  (\ref{commutative diagram for generalized cauchy formula})
holds by  the Stokes formula.
In general, the correction term for the Stokes formula arising
from the logarithmic singularity is computed in terms of $\partial\ga$  . 
The main results of \cite{part I} are used in the proof
of Theorem \ref{th:generalized Cauchy formula} which  is  fairly complex.
If the reader is mainly interested in its applications  and is willing to accept this fact,
 it is possible to read \S \ref{sec: Hodge realization}
separately. We would like to mention a further generalization of the generalized Cauchy formula. Let $U$ be a smooth quasi-projective
complex variety, and let $\bold H$ be a simple normal crossing divisor on $U$. Let $\varphi$ be a smooth $p$-form on $U$ with compact support which has logarithmic
singularity along $\bf H$, and let $\ga$ be a $\delta$-admissible $(p+1)$-chain on $U$. Then in \cite{Ki} we show that  
the equality
\begin{equation}
\label{quasiproj}
2\pi i\int_{\part_{\bf H}\ga} \res \varphi=\int_{\delta\ga}\varphi-\int_\ga d\varphi
\end{equation}
holds. Here $\res \varphi$ is the Poincar\'{e} residue of $\varphi$ at $\bold H$. This generalized Cauchy  formula
can be used to describe the duality pairing between the de Rham cohomology groups of
the complement of $\bold H$ and and the  singular cohomology groups of $U$ relative to $\bold H$. See \cite{Ki}
for more detail. The essential technical facts needed to prove (\ref{quasiproj}) are proved in this paper.

We construct 
a variant of the Hodge realization functor
for the category of mixed Tate motives in  \S \ref{sec: Hodge realization}.
We briefly
recall the construction of the category of mixed Tate motives 
given
in the paper of Bloch and Kriz (\cite{BK}).
Let $\bold k$ be a subfield of $\CC$.
Bloch defines a graded DGA $N_{\bold k}$ of algebraic cycle complexes
of $\bold k$. The $0$-th cohomology $\chi_N:=H^0(B(N_{\bold k}))$
of the bar complex
$B(N_{\bold k})$ of $N_{\bold k}$ becomes a commutative 
Hopf algebra with a grading $\chi_N=\oplus_r(\chi_N)_r$.
They define the category of mixed Tate motives
as that of graded co-modules over the Hopf algebra $\chi_N$.
They also define the $\ell$-adic and the Hodge realization
functors
 from the category of
mixed Tate motives over $\Spec(\bold k)$ to 
that of $\ell$-adic Galois representations of the field $\bold k$,
and 
that of mixed Tate Hodge structures.
The paper \cite{G} is one of the works based on this construction.

In \cite{BK}, they also present an alternative construction
of the Hodge realization functor 
using integral of 
logarithmic differential forms 
$\omega_n$
on $P^n$. For this they assume   the existence of a
certain DGA $\cal D\cal P$ of topological chains satisfying the following conditions.

\begin{enumerate}
\renewcommand{\labelenumi}{(\alph{enumi})}
 \item 
The DGA $\cal D\cal P$ contains the DGA which is generated by the chains of
 the complex points  of the subvarieties generating $N_{\bold k}$ ,
\item
The integral of the form 
$\dfrac{dz_1}{z_1}\wedge \cdots \wedge \dfrac{dz_n}{z_n}$ on elements in $\cal D\cal P$
converges.
\item
The generalized Cauchy integral 
formula holds for the integral in (b).
\item
The natural map 
$\tau^*:H^*(B(N_{\bold k}))\to H^*(B(\cal D\cal P,N_{\bold k},\QQ))$ is $0$, which
implies the $E_1$-degeneracy of the 
spectral sequence obtained from a certain filtration 
on $B(\cal D\cal P,N_{\bold k},\QQ)$ (for the precise
statement see \cite{BK}  (8.6)).
\end{enumerate}

The cubical differential  $\part$ and  the topological differential $\delta$
make the direct sum of \newline $AC_i(P^n, \bold D^n; \QQ)^{\rm alt}$'s a double complex. Here
${}^{\rm alt}$ indicates that we take the subspace on which a certain group $G_n$  acts
by a certain character sign.  The associated
simple complex of this, which we denote by $AC^\bullet$, is a cohomological complex.  
Let $I$ be the map from $AC^\bullet$ to $\CC$ defined by  $I=\sum_n I_{n}$.
The equality
(\ref{commutative diagram for generalized cauchy formula})
implies that the map $I$
is a homomorphism of complexes. 
We use ${AC}^{\bullet}$ in place of $\cal D \cal P$, which enjoys the
following properties:
\begin{enumerate}
 \item 
There exists a natural injection $N_{\bold k}\to {AC}^\bullet$.
Via this map, we define an $\chi_N$ co-module $\cal H_{B}$ in 
 Definition \ref{def of universal mixed Hodge structure}.
\item
There is a canonical map $\QQ \to {AC}^\bullet$ which
is a quasi-isomorphism. 
\item
The above homomorphism $I$ induces a quasi-isomorphism from  $ {AC}^\bullet\ot \CC$ to $\CC$.
\end{enumerate}
Using the properties (2) and (3),
we show the $E_1$-degeneracy of the spectral sequence obtained 
from a similar filtration on $B(\QQ,N_{\bold k}, {AC}^\bullet)$
as in (d).


We give a rough sketch of our construction of the Hodge realization. We  define an $\chi_N$ co-module $\cal H_{B}$ resp. $\cal H_{dR}$ as $H^0(B(\QQ,N_{\bold k},AC^\bullet))$
resp. $H^0(B(\QQ,N_{\bold k},{\bold k})$.
The homomorphism $I$ in (3) yields a comparison isomorphism
$c:\cal H_B\otimes \CC \to \cal H_{dR}\otimes \CC$,
and via this comparison map,
we construct a ``universal'' mixed Tate Hodge structure
$\cal H_{Hg}=(\cal H_{B},\cal H_{dR},c)$ 
with a left ``coaction'' $\Delta_{Hg}$ of 
$\cal H$ (see (\ref{Hhodge})).
We define a functor $\Phi$ from
the category of graded right $\chi_N$-co-modules
$(V,\Delta_V)$ to that of mixed Tate Hodge structures by
the ``cotensor product''
$$
\Phi(V)=\ker\bigg(V\otimes  \cal H_{Hg}\xrightarrow{\Delta_V\otimes
id-id\otimes \Delta_{Hg}}
V\otimes \chi_N
\otimes \cal H_{Hg}\bigg).$$ 

In the sequel to this paper, we will prove that the 
 above functor $\Phi$ is isomorphic to that defined by
Bloch-Kriz.

In Subsection \ref{flat conn} we explain about  the relation of Bloch-Kriz mixed Tate motives to a more recent construction of the category
of mixed Tate motives. For each field $k$, Hanamura constructed a triangulated category ${\cal D}(k)$ of mixed motives in \cite{Ha1} 
and 
\cite{Ha2}.  Let ${\cal D}T(k)$ be the sub triangulated category of ${\cal D}(k)$
generated by Tate objects.  If one assumes the Beilinson-Soul\'e vanishing conjecture and the $K(\pi,1)$ conjecture, the heart of ${\cal D}T(k)$
for a certain $t$-structure should be equivalent to the category of {\it flat} $N_k$-{\it connections}, which is equivalent to
the category of Bloch-Kriz mixed Tate motives. We will show that there exists a natural Hodge realization functor on the category
of flat $N_k$-connections, which is compatible with the one we define for Block-Kriz mixed Tate motives. 
Hanamura's category ${\cal D}(k)$ is known to be equivalent to Voevodsky's category $DM(k)$ (\cite{Voe}) and to
Levine's category ${\cal DM}(k)$ (\cite{Lev2}). Deligne and Goncharov define a category of mixed Tate motives over the ring
of integers of a each number field $k$ in \cite{DG}. 

After the work of Bloch-Kriz, much progress has been made on the theory of motives.
This of course includes construction of the Hodge realization functors on the cateogry
of motives. On the other hand, we have been continuously interested in the 
periods arising from motives. 
Unfortunately, the Hodge realization functors constructed so far are not very
suitable to see the concrete period integral associated to algebraic cycles. We believe
that our construction of a Hodge realization of mixed Tate motives provides
a significant step towards the understanding of periods associated to motives.

\noindent
{\bf Acknowledgment}

We would like to thank Professor T. Suwa for helpful discussions 
on the intersection theory
of semi-algebraic sets. We are grateful to the referee for careful reading of the manuscript
and suggestions which improved the exposition.

\section{Admissible chain complexes}
\label{sec:admissible chain complex}


\subsection{Semi-algebraic triangulation}
\label{sectionsemi-alg}

Let $k$ be a non-negative integer. A $k$-simplex in a Euclidean 
space $\R^n$ is the convex hull of 
affinely independent points $a_0,\cdots,
a_k$  in $\R^n$.
{\it A finite simplicial complex} of $\R^n$ is a finite set $K$ consisting of  simplexes 
 such that (1) for all $s\in K$, all the faces of $s$
belong to $K$  and (2) for all $s,t\in K$, $s\cap t$ is either the empty set
or a common face of $s$ and $t$.
We denote by $K_p$ the set of $p$-simplexes
of $K$.
For a simplex $\sigma$ in $K$, the relative interior of $\sigma$ 
(=the complement of the union of the proper faces of $\sigma$)
is denoted by $\sigma^{\circ}$.
For a finite simplicial complex $K$, the union of  simplexes
in $K$ as a subset of $\R^n$ is denoted by $|K|$. A {\it semi-algebraic subset} of $\R^n$
is a subset of the form
\[\underset{i=1}{\overset{s}\cup}\underset{j=1}{\overset{r_i}\cap}\{x\in \R^n\,|\, 
f_{i,j}\ast_{i,j} 0\}\]
where $f_{i,j}\in \R[X_1,\cdots,X_n]$ and $\ast_{i,j}$ is either $<$ of $=$, for
$i=1,\cdots, s$ and $j=1,\cdots, r_i$. 
As for the  fundamental
properties of semi-algebraic sets, see \cite{BCR}.
\begin{theorem}[\cite{BCR}, Theorem 9.2.1]
\label{thm:semi-algebraic triangulation}
Let $P$ be a compact semi-algebraic subset of $\R^m$. The set $P$ is triangulable, i.e. there exists
a finite simplicial complex $K$ and a semi-algebraic homeomorphism
$\Phi_K:\,|K|\to P$. Moreover, for a given finite family 
$S=\{S_j\}_{j=1,\cdots,q}$ of semi-algebraic
subsets of $P$, we can choose a finite simplicial complex $K$
and a semi-algebraic homeomorphism
$\Phi_K:\,|K|\to P$ such that every $S_j$ is the
union of a subset of $\{\Phi_K(\sigma^\circ)\}_{\sigma\in K}.$
\end{theorem}

Here we note that a map between semi-algebraic sets is called semi-algebraic if the graph of the 
map is a semi-algebraic set. 
\begin{remark}
\label{facewise regular embedding}
\begin{enumerate}

\item By  \cite{BCR} Remark 9.2.3 (a), the map $\Phi_K$ can be taken so that the map 
$\Phi_K$ is {\it facewise regular embedding} i.e. for each $\sigma\in K$, 
$\Phi_K(\sigma^\circ)$  is a  regular
submanifold of $\R^m$. 

\item The pair $(K, \Phi_K)$ 
as in Theorem \ref{thm:semi-algebraic triangulation}
is called a 
{\it semi-algebraic triangulation}\index{semi-algebraic triangulation} of $P$; we will then identify $|K|$ with $P$. 
A projective real or complex variety $V$ is regarded  as a compact semi-algebraic subset 
of an Euclidean space by \cite{BCR} Theorem 3.4.4, thus the above theorem applies to $V$. 
\end{enumerate}
\end{remark}

\begin{notation}
\label{loose notation}
For a subcomplex $L$ of $K$,  the 
space $|L|$ is a subspace of $|K|$.  
A subset of $|K|$ of the form
$|L|$ is also called a subcomplex. If a subset $S$ of $|K|$ is equal to $|M|$ for a subcomplex $M$
of $K$, then $M$ is often denoted by $K\cap S$.
\end{notation}


\subsection{The complex
$AC_{\bullet}(P^n,\bold D^n;\QQ)$}
\label{subsec:def of admissible chain}
Let $(z_1,\cdots,z_n)$ be the coordinates of $P^n=(\P^1_\CC)^n$. For $i=1,\cdots,n$ and $\al=0,\infty$,
we denote by $H_{i,\al}$ the subvariety of $P^n$ defined by $z_i=\al$.  The intersection of several $H_{i,\al}$'s is called a {\it cubical face}.
We set
\begin{align}
\label{notation of divisors}
&\bold H^n=\bigcup_{\substack{1\leq i\leq n\\
\alpha=0, \infty}} H_{i,\alpha}, \quad
\bold D^n=\bigcup_{i=1}^n \{z_i=1\},  \quad \square^n=P^n-\bold D^n.
\end{align}
\index{$\bold H^n$}
\index{$\bold D^n$}
\index{$\square^n$}

In the following we suppose that each triangulation $K$ of  $P^n$ is semi-algebraic , and that the divisor
$\bold D^n$ is a subcomplex of $K$.  
Let $K$ be a  triangulation of $P^n$. We denote by $C_\bullet(K;\QQ)$ resp. $C_\bullet(K,\bold D^n;\QQ)
(=C(K;\QQ)/C(K\cap \bold D^n;\QQ))$ 
the chain complex resp. the relative chain 
complex of $K$. An element of $C_p(K;\QQ)$ for $p\geq 0$ is written as $\sum a_\si \si$
where the sum is taken over $p$-simplexes of $K$. By doing so, it is agreed upon that an orientation
has been chosen for each $\si$.  By abuse of notation, an element of $C_\bullet(K,\bold D^n;\QQ)$ 
is often described similarly. 

\begin{definition}
 For an element $\gamma=\sum a_\si \sigma$ of
$C_{p}(K;\QQ)$, we define
the support\index{support} $|\gamma|$ of $\gamma$ as the subset of $|K|$
 given by
\begin{equation}
\label{support}
|\gamma|=\bigcup_{
\substack{\sigma\in K_p\\ a_\sigma\neq 0}} \sigma.
\end{equation}
\end{definition}
Under Notation \ref{loose notation}, For
an element $\gamma=\sum a_\si \sigma$ of
$C_{p}(K;\QQ)$, 
$|\gamma|$ is sometimes regarded as a subcomplex of $K$.

\begin{definition}
\label{def:semi-alg current}
Let $p \geq 0$ be an integer.
\begin{enumerate}

\item(Admissibility) A semi-algebraic subset $S$ of $P^n$
is 
said to be admissible\index{admissible} if for each cubical face $H$, 
the inequality 
$$
\dim(S\cap (H-\bold D^n))\leq \dim S -2\,\codim H
$$
holds.  Here note that $\dim(S\cap (H-\bold D^n))$ and $\dim S$ are the dimensions as semi-algebraic sets,
and $\codim H$ means the codimension of the subvariety $H$ of $P^n$.

\item Let $\gamma$ be an element of $C_p(K,\bold D^n; \QQ)$.
Then $\ga$ is 
said to be admissible if 
the support of a representative of $\ga$ in $C_p(K;\QQ)$ is admissible. This condition is independent of the choice
of a representative. 
\item We set 
$$
AC_{p}(K,\bold D^n;\QQ)
=\{\gamma\in C_{p}(K,\bold D^n;\QQ))
\mid 
\gamma \text{ and } \delta \gamma \text{ are admissible }\}
$$
\index{$AC_{p}(K,\bold D^n;\QQ^{q})$}
\end{enumerate}
\end{definition}

\subsection{Subdivision and inductive limit}

\begin{definition}
Let $(K,\,\Phi_K:\,|K|\to P)$ be a  triangulation of  a compact semi-algebraic set $P$. Another 
triangulation $(K',\,\Phi_{K'}:\,|K'|\to P)$ is {\rm a subdivision of}\index{subdivision} $K$ if :
\begin{enumerate}
\item The image of each simplex
of $K'$ under the map $\Phi_{K'}$ is contained in  the image of a simplex of $K$
under the map $\Phi_K$.
\item The image of each simplex of $K$ under the map $\Phi_K$ is the union of the images of
simplexes of $K'$ under $\Phi_{K'}$.
\end{enumerate}
\end{definition}
If $K'$ is a  subdivision of a  triangulation $K$, there is 
a natural homomorphism of complexes
$\lambda:
C_{\bullet}(K,\bold D^n;\QQ)\to 
C_{\bullet}(K',\bold D^n;\QQ)$ called the subdivision operator. See for example  
 \cite{Mu} Theorem 17.3 for the definition. For a simplex $\si$ of $K$, the chain $\la(\si)$ is carried
by $K'\cap \si$, and so that the map $\la$ sends $AC_{\bullet}(K,\bold D^n;\QQ)$ to 
$AC_{\bullet}(K',\bold D^n;\QQ)$. 
By Theorem \ref{thm:semi-algebraic triangulation} two  semi-algebraic triangulations
have a common subdivision. 
Since the map $\lambda$
and the differential $\delta$ commute, the complexes 
$C_{\bullet}(K,\bold D^n;\QQ)$ and 
$AC_{\bullet}(K,\bold D^n;\QQ)$ form
inductive systems indexed by triangulations $K$ of $P^n$.
\begin{definition}

\label{def:definition of AC with inductive limit}
We set
\begin{align*}
C_{\bullet}(P^n,\bold D^n;\QQ) =
\underset{\underset{K}\longrightarrow}{\lim } \ 
C_{\bullet}(K,\bold D^n;\QQ), \quad
AC_{\bullet}(P^n,\bold D^n;\QQ) =
\underset{\underset{K}\longrightarrow}{\lim } \ 
AC_{\bullet}(K,\bold D^n;\QQ).
\end{align*}
Here the limit is taken on the directed set of  triangulations.

\end{definition}
\index{$C_{\bullet}(P^n,\bold D;\QQ^{\bullet})$}
\index{$AC_{\bullet}(P^n,\bold D;\QQ^{\bullet})$}


A proof of the following Proposition will be given in Appendix \ref{section moving lemma}.
\begin{proposition}[Moving lemma] 
\label{prop: moving lemma}
The inclusion of complexes
\begin{equation}
\label{moving quasi-iso}
\iota:\,\,AC_{\bullet}(P^n,\bold D; \QQ) 
\to C_{\bullet}(P^n,\bold D; \QQ)
\end{equation}
is a quasi-isomorphism.
\end{proposition}

\section{Face map and cubical differential}
\label{sec:Face map and cubical chain}

Recall that a subcomplex $L$ of a simplicial complex
$K$ is called a full subcomplex \index{full subcomplex}, 
if a simplex $\si$ in $K$ belong to $L$ whenever all the vertices of $\si$ are in $L$.
\begin{definition}[Good triangulation]
\label{good triangulation}
We define a family $\cal L$ of subsets of $P^n$ by
$$
\cal L=\{H_{J_1}\cup \cdots\cup H_{J_k}\}_{(J_1, \dots, J_k)},
$$ 
where $H_{J_j}$ are cubical faces of $P^n$. In short, a member of $\cal L$ is the union of several
cubical faces.
A finite semi-algebraic triangulation $K$ of $P^n$ is called a good triangulation \index{good triangulation} if $K$
satisfies the following conditions.
\begin{enumerate}
\item The divisor $\bold D^n$ is a subcomplex of $K$.
\item The map $\Phi_K:\,|K|\to P^n$ is facewise regular embedding. cf. Remark \ref{facewise regular embedding}.
\item 
\label{faces are full subcomplex}
Each element $L_i\in \cal L$ is a full subcomplex of $K$ i.e.  there exists a full subcomplex $M_i$ of $K$ such that $L_i=|M_i|$. 
\item  For $i=1,\cdots,n$, the subsets $\{|z_i|\leq 1\}$ are subcomplexes of $K$. 
\end{enumerate}

\end{definition}
In particular, if $K$ is a good triangulation,
then for any simplex $\si$ of 
$K$ and $L_i\in \cal L$,
the intersection $\si\cap  L_i$ is a (simplicial) face of $\si$. This is the primary reason to consider the condition (3).

\begin{remark} 
\label{barycentric subdiv is good}
For a simplicial complex $K$, we denote by $\sd K$ its barycentric subdivision. If $L$ is a subcomplex
of $K$, then $\sd L$ is a full subcomplex of $\sd K$ (See for example Exercise 3.2 of \cite{RS}).   It follows that 
if  $K$ is a semi-algebraic triangulation of 
$P^n$ which  is a facewise regular embedding, and such that
$\bold D^n$ and each $L_i\in \cal L$ are  subcomplexes of $K$, then $\sd K$ is a good triangulation.
\end{remark}

In the following, each triangulation $K$ of $P^n=(\P^1_\CC)^n$ is assumed to be good in the sense of Definition
\ref{good triangulation}.

\subsection{Thom cocycles}
\subsubsection{Definition of Thom cocycles}
\label{cap product fixing simplicial ordering}
We first recall the definition of the Thom class of a submanifold. Let $X$ be a compact orientable
manifold of dimension d, and let $S$ be a closed submanifold of $X$ of codimension
$c$. We have Lefschetz duality 
$$H^c(X,X-S; \ZZ)\simeq H_{d-c}(S;\ZZ).$$ 
The Thom class of $S$ which we denote by $\Psi(S)$, is the class in 
$H^c(X,X-S; \ZZ)$ which is mapped to the fundamental class of $S$ in $H_{d-c}(S)$.
Let $i:\,S\to X$ be the inclusion. Then the Poincar\'{e} dual of 
$i^*:\,H^\bullet(X)\to H^*(S)$ is given by the cap product
$$\al\mapsto \Psi(S)\cap \al,\,\,H_\bullet(X)\to H_{\bullet-c}(S).$$
cf. \cite{Iv} X.4.  
Let $K$ be a good triangulation of $P^n=(\P^1_\CC)^n$.  

Let $\Delta$ be the subset $\{|z_1|< 1\}$ of  $P^n$. 
We set $L_1=K\cap   H_{1,0}$ and
$$
\dis N= \underset{\si\in K,\,\si\cap   H_{1,0}\neq \emptyset}\cup \si
,\quad W=\underset{\si\in K,\,\si \cap   H_{1,0}= \emptyset}\cup \si  .
$$

For spaces $X$ and $Y$ such that $X\supset Y$, we denote by  $H^i_{\sing}(X;\QQ)$ (resp. $H^i_{\sing}(X,Y;\QQ)$) the singular cohomology
(resp. the relative singular cohomology) of $X$ (resp. $(X,Y)$) with $\QQ$-coefficients.
Under the comparison isomorphism 
$H^1_{\sing}(\Delta-  H_{1,0};\C)\simeq H^1_{dR}(\Delta-  H_{1,0})$,
the de Rham class $\big[\dfrac{dz_1}{2\pi i z_1}\big]$
of $\dfrac{dz_1}{2\pi i z_1}$
is contained in the subgroup
$H^1_{\sing}(\Delta-  H_{1,0};\QQ)$ of $H^1_{\sing}(\Delta-  H_{1,0};\C)$.
Since $L_1$ is a full subcomplex of $K$, $W$ is a deformation retract of $P^n-  H_{1,0}$.
(This can be seen as follows.
Each simplex $\si$ in $N$ not contained in $L_1$ is the join $\tau*\eta$ where $\tau\in L_1$ and $\eta\in W$.
The linear projection $\pi_{\si,\eta}:\,\si-\tau\to \eta$ in Definition \ref{linear proj} is glued  to a deformation
retract from $P^n-H_{1,0}$ to $W$.)
Therefore we have isomorphisms
\begin{equation}
\label{before thoem cycycle}
H^2(K, W;\QQ)\xleftarrow{\simeq}  H^2_{\sing}(P^n, P^n-  H_{1,0};\QQ)
\xrightarrow{\simeq} H^2_{\sing}(\Delta, \Delta-  H_{1,0};\QQ).
\end{equation}
Here the second isomorphism holds by excision, and $H^2(K,W;\QQ)$
is the simplicial cohomology group.

\begin{definition}
\label{Thom cocycle}
A simplicial cocycle T in $C^2(K,W;\QQ)$ is a Thom cocycle\index{Thom cocycle} if 
its cohomology class is equal to
$\delta \big[\dfrac{dz_1}{2\pi i z_1}\big]$
in $H^2_{\sing}(\Delta, \Delta-  H_{1,0};\QQ)$
via the isomorphism (\ref{before thoem cycycle}).
Here $\delta$ denotes the connecting homomorphism
$$
\delta:H^1_{\sing}(\Delta-  H_{1,0};\QQ)\to H^2_{\sing}(\Delta, \Delta-  H_{1,0};\QQ).
$$
A $\C$-valued Thom cocycle in $C^2(K, W;\CC)$ is defined similarly.
\end{definition}

We will give some examples of Thom cocycles. 

\subsubsection{Thom form}
\label{Thom form}
Let $\epsilon$ be a positive real number and
Let $\rho$ be a $[0,1]$-valued $C^{\infty}$-function 
on $\Delta$ 
such that
\begin{enumerate}
\item
$\rho=0$ on $\Delta\cap \{ |z_1|<\dfrac{1}{2}\epsilon \}$.
\item
$\rho=1$ on $\Delta\cap \{ \epsilon<|z_1|\}$.
\end{enumerate}
Then $c_{\rho}=\dfrac{1}{2\pi i}\rho \dfrac{dz_1}{z_1}$ defines an element of
$C^1_{\sing}(P^n-H_{1,\infty};\CC)$. 
We set $T_{\rho}=dc_{\rho}=\dfrac{1}{2\pi i}d\rho \wedge \dfrac{dz_1}{z_1}$.
For $\e$ sufficiently small we have $c_\rho= \dfrac{1}{2\pi i} \dfrac{dz_1}{z_1}$ on $W-H_{1,\infty}$, 
and $T_\rho$ defines the same class as $\delta[ \dfrac{dz_1}{2\pi i z_1}]$ in $H^2_{\sing}(P^n-H_{1,\infty},
W-H_{1,\infty};\CC)\simeq H^2_{\sing}(P^n,
W;\CC)$, so it is a rational class. Take a partial ordering on the set of vertices of $K$
which gives a total ordering  on the set of the vertices of each simplex of $K$.
Then we have a map of complexes
$$ c_{\cal O}:\,\,C^\bullet_{sing}(P^n,W;\CC)\to C^\bullet (K,W;\CC)$$
which is a quasi-isomorphism. Moreover, The maps on the cohomology
groups induced by $c_{\cal O}$ is independent of the ordering. 
\begin{definition}
\label{def Thom form}
\index{Thom form $T_{\rho}$}
The cocycle $T_{\rho}\in C^2(K, W;\CC)$ is a 
$\CC$-valued Thom cocycle.
We call the above cocycle $T_\rho$ 
a  Thom form. $T_\rho$ depends on the choice of an ordering $\cal O$ of $K$.
\end{definition}

\subsubsection{Singular Thom cocycle $T^B_{  H_{1,0}}$}
\label{subsubsect:Thom cochain}
We first note that this cocycle is not used in the rest of the paper,  and we omit the proof that
this is a Thom cocycle.
Suppose that $\Delta$ is a subcomplex of $K$ and $\Delta \cap W\subset \Delta-\{0\}$
is a deformation retract.
For a $1$-simplex $\sigma \in C^1(\Delta)$, we set 
$$
L^B(\gamma)=\begin{cases}
0 
&\quad\text{ if }\sigma\not\subset W \\	   
\bigg[\displaystyle\dfrac{1}{2\pi}
\bigg(\operatorname{Im} \int_{\sigma}\dfrac{dz_1}{z_1}+\arg(\gamma(0))\bigg)\bigg]
&\quad\text{ if }\sigma\subset W.
\end{cases}
$$
Here 
$\arg(z)$ is the argument of a complex number $z$ in $[0,2\pi)$.
Note that the cochain $L^B$ counts the intersection number (with sign) of $\sigma$ and
the positive part of real axis.
Then $T^B_{  H_{1,0}}=dL^B\in C^2(\Delta,\Delta\cap W)$ becomes a Thom cocycle.

We remark that the cocycle $T^B_{  H_{1,0}}$
counts the winding number of the boundary of relative $2$-cycle.

\subsection{The cap product with a Thom cocycle}
\label{subsec:cap prod and thom cocycle}
\subsubsection{Simplicial cap product}
\begin{definition}[Ordering of complex, 
good ordering]
\label{def:good simplical ordering}
Let  $K$ be a good triangulation of $P^n$. 
\begin{enumerate}
\item
A partial ordering on the set of vertices in $K$ is called 
an ordering\index{ordering} of $K$,
if the
     restriction of the ordering to each simplex is a total ordering.
\item  
Let $H_J$ be a cubical face of $P^n$.
An ordering of $K$ is said to be good with respect to $H_J$ if it satisfies the following condition.
If  a vertex $v$ is on $H_J$  and  $w\geq v$ for a vertex $w$, then  $w \in H_J$. 
\index{good ordering}
\end{enumerate}
\end{definition}

We denote by $[a_0,\cdots, a_k]$\index{$[a_0,\cdots, a_k]$} the simplex spanned by $a_0,\cdots, a_k$.
Let $\co$ be a good ordering of $K$ with respect to $  H_{1,0}$.
We recall the definition of the cap product\index{cap product}
\newline $\overset{\co}\cap:C^p(K;\QQ)\otimes C_k(K;\QQ) \to C_{k-p}(K;\QQ)$. 
For a simplex $\al=[v_0,\cdots,v_k]$ such that $v_0< \cdots <v_k$ and $u\in C^p(K;\QQ)$, we define 
\begin{equation}
\label{def:simplicial cap product}
u\overset{\co}\cap \alpha=u([v_0, \dots, v_p])[v_{p},\dots, v_k].
\end{equation}
 One has the boundary formula
\begin{equation}
\label{boundary formula}
\delta(u\overset{\co}\cap \alpha)=(-1)^p(u\overset{\co}\cap (\delta\alpha)
-(du)\overset{\co}\cap \alpha)
\end{equation}
where $du$ denotes the coboundary of $u$, see \cite{Hat}, p.239 (note the difference
in sign convention from \cite{Mu}). 
Thus if $u$ is a cocycle, $\delta(u\overset{\co}\cap \alpha)
=(-1)^pu\overset{\co}\cap (\delta\alpha).$

\begin{proposition}
\label{basic property of face map 1}
Let $T$ be a Thom cocycle and $\co$ a good ordering of $K$ with respect to $  H_{1,0}$.
\begin{enumerate}

\item The map $T\overset{\co}{\cap}$ 
and the
differential $\delta$ commute. 
\item 
The image of the homomorphism 
$T\overset{\co}{\cap} $ is contained in $C_{k-2}(L_1;\QQ)$,
where $L_1=K\cap   H_{1,0}$.
 As a consequence, we have a homomorphism of
 complexes
\begin{equation}
\label{cap prod no cond}
T\overset{\co}{\cap} :\,\,C_k(K, \bold D^n; 
\QQ) \to
C_{k-2}(L_1, \bold D^{n-1}; \QQ).
\end{equation}
\end{enumerate}
\end{proposition}
\begin{proof}
(1). Since $T$ is a cocycle of even degree, we have
$\delta(T\overset{\co}\cap \sigma)=T\overset{\co}\cap (\delta\sigma)$ for a simplex $\si$.

\noindent (2). If $v_2 \notin   H_{1,0}$, then $[v_0,v_1,v_2] \cap   H_{1,0}=\emptyset$ and
 $T([v_0,v_1,v_2])=0$ since the cochain $T$ vanishes on $W$.
If $v_2 \in  H_{1,0}$, then the vertices $v_2,\cdots, v_k$ are on $  H_{1,0}$, and we have $[v_2,\dots, v_k]\subset  H_{1,0}$
since $  H_{1,0}$ is a full subcomplex of $K$.
Thus the assertion holds.
\end{proof}

\subsubsection{Independence of $T$ and  ordering}
\label{independece of c}
Let $K$ be a good triangulation, and let $L_1$
be the subcomplex $K\cap  H_{1,0}$. 
\begin{proposition}
\label{prop face map first properties}
Let $T\in C^2(K,W;\QQ)$ be a Thom cocycle of the face $H_{1,0}$ and let
 $\gamma$ be an element of $AC_{k}(K,\bold D^n;\QQ)$.
Then we have the following. 
\begin{enumerate}
\item
The chain  $T\overset{\co}{\cap} \gamma$
is an element in $AC_{k-2}(L_1,\bold D^{n-1}; \QQ)$.
\item
\label{indep of T for face map}
The chain $T\overset{\co}{\cap} \gamma$ is
independent of the choice of a 
Thom cocycle $T$ and a good ordering $\co$. 
Thus the map 
$$
T\overset{\co}\cap:
AC_k(K, \bold D^n; 
\QQ) \to
AC_{k-2}(L_1, \bold D^{n-1}; \QQ).
$$
induced by (\ref{cap prod no cond})
 is denoted by $T\cap$.
\item
\label{comatibility for subdivision}
Let $K'$ be a good subdivision of $K$.
We set 
$$
W'=\underset{\substack{\sigma\in K'\\ \sigma'\cap  H_{1,0}=\emptyset}}\cup
\sigma'.
$$
Let $T'\in C^2(K', W';\QQ)$ be a 
Thom cocycle, and $\co'$ a good
ordering of $K'$ with respect to $  H_{1,0}$. 
We set $L'_1=K'\cap  H_{1,0}$.
Then we have the
following commutative diagram
\begin{equation}
\label{commutative diagram for subdivision}
\begin{matrix}
AC_k(K,\bold D^n;\QQ)
&\xrightarrow{T\overset{\cal O}\cap}&
    AC_{k-2}(L_1,\bold D^{n-1}; \QQ) 
\\
\lambda \downarrow & & \downarrow\lambda
\\
AC_k(K',\bold D^n;\QQ)
&\xrightarrow{T'\overset{\cal O'}\cap}&
    AC_{k-2}(L'_1,\bold D^{n-1}; \QQ)
\end{matrix}
 \end{equation}
where the vertical maps $\lambda$ are subdivision operators.
\end{enumerate}
\end{proposition}
\begin{proof} (1). 
For an element $z\in C_k (K, \bold D^n;\QQ)$, 
we have $T\overset{\co}{\cap} z\in  C_{k-2} (L_1, \bold D^n;\QQ)$
by Proposition \ref{basic property of face map 1} (2). By the definition of the cap product, we see that 
the set $|T\overset{\co}{\cap} z|\subset |z|\cap  H_{1,0}$. It follows that if  $z$ is admissible i.e. $|z|-\bold D^n$ meets all
the cubical faces properly, then  $|T\overset{\co}{\cap} z|-\bold D^{n-1} $ meets all  the cubical
faces of $  H_{1,0}$ properly.  Similarly, if  the chain  $\delta z$ is admissible, then  $T\overset{\co}{\cap}(\delta  z)$
is admissible in $  H_{1,0}$.  By Proposition \ref{basic property of face map 1} (1) we have the equality $\delta(T\overset{\co}{\cap} z)=
T\overset{\co}{\cap}(\delta  z)$.

(\ref{indep of T for face map}) A proof will be given in Section \ref{subsec orgering}.

(3)  A proof will be given in Section \ref{subsec subdivision}.

 \end{proof}
By taking the inductive limit 
of the homomorphism
$$
T\cap :AC_{\bullet}(K,\bold D^n;\QQ)\to
   AC_{\bullet-2}(L_1,\bold D^{n-1}; \QQ).
$$
for subdivisions,
we get 
a homomorphism
\begin{equation}
 \label{face map single sheaf}
T\cap:\,AC_\bullet(P^n, \bold D^n;\QQ)\to 
AC_{\bullet-2}(P^{n-1}, \bold D^n; \QQ).
\end{equation}

\begin{definition}
\label{face map}
The map $T\cap$ of (\ref{face map single sheaf}) is denoted by $\part_{1,0}$, and called the
face map of the face $  H_{1,0}$.
\end{definition}
Note that By Proposition \ref{prop face map first properties} (2) the map $T\cap$ does not depend on the 
choice of a Thom cocycle of the face $H_{1,0}$. 
By the symmetry, we have maps $\part_{i,\al}$ for $i=1,\cdots,n$ and $\al =0,\infty$.

\subsection{Cubical differentials}
\label{sec:cubical sheaves, cubical differential}
\begin{definition}[Cubical differential]
\label{def face map}
The map 
$\partial$ \index{cubical differential $\partial$} is defined by the equality
\begin{equation}
\label{total face map definition} 
\part =\sum_{i=1}^n (-1)^{i-1}
(\part_{i,0}-\part_{i,\infty}):\,\,
AC_{\bullet}(P^n,\bold D^n;\QQ)\to
AC_{\bullet-2}(P^{n-1},\bold D^n;\QQ).
\end{equation}
\end{definition}

\begin{proposition} 
\label{total face map is a differential}
The composite
$$
\part^2:
AC_{\bullet}(P^n,\bold D^n;\QQ)\to
AC_{\bullet-4}(P^{n-2},\bold D^n;\QQ)
$$
is  the zero map.
\end{proposition}
A proof of this proposition will be given in Section \ref{subsec cup}.

\subsection{Compatibility of the face map}
If $V$ is a closed subvariety of $P^n$ such that $V\cap \s$  meets all the faces properly, then the set of complex points of $V$ defines
as an element of $AC_{2d}(P^n, \bold D^n;\QQ)$ where $d=\dim V$. For each face $H_{i,\al}$ let $\part_{i,\al}^{alg} V$
be the intersection of the subvariety $V$ and $H_{i,\al}$  considered with multiplicity. The algebraic cycle
 $\part_{i,\al}^{alg} V$ defines an element of $AC_{2d-2}(P^{n-1}, \bold D^{n-1};\QQ)$.  
\begin{proposition}
\label{algtopeq}
Let $V$ and $\part_{i,\al}^{alg} V$ be as above. Then we have an equality
\[\part_{i,\al}^{alg} V=\part_{i,\al}V\]
in $AC_{2d-2}(P^{n-1}, \bold D^{n-1};\QQ)$.  
\end{proposition} 
\begin{proof}
We can and do assume that $H_{i,\al}=H_{1,0}=\{z_1=0\}$. The notation $H_{1,0}$ is abbreviated to $H$
 and $\part_{1,0}$ is abbreviated to $\part_H$. 
Let $V_{sm}$ resp. $V_{sing}$  be the smooth resp. singular locus of $V$. We begin with the case where $H$ meets transversely with $V_{sm}$ and $H$ meets $V_{sing}$ properly.  In this
case we need to show that the multiplicity of $\part_H V$ is 1. By Proposition \ref{prop face map first properties}
(3)  the map $\part_H$ is compatible with subdivisions, so the problem is local. Take a general point $x$ of
$V_{sm}\cap H$. The point $x$ has a neighborhood $V_x$ which is semi-algebraically homeomorphic to the polydisc
$$D^d=D^1\times D^{d-1}=\{(z_1,\cdots, z_d)\in \C^d\,|\, |z_i|\leq 1,\,\,1\leq i\leq d\}$$
and $V_x\cap H=\{z_1=0\}$.  We recall a triangulation of the product of two simplices. Let
$\si=[v_0,\cdots,v_i]$ and $\tau=[w_0,\cdots, w_j]$  be a $i$- resp. $j$- simplex. Consider a sequence $s$ 
of pairs 
\[ (a_0,b_0),\cdots, (a_k,b_k),\cdots, (a_{i+j},b_{i+j})\]
such that $(a_0,b_0)=(0,0)$ and $(a_{k+1},b_{k+1})$ is either $(a_k+1,b_k)$ or$(a_k,b_k+1)$. To the sequence
$s$ we associate a $(i+j)$-simplex $\Delta_s=[(v_{a_0},w_{b_0}),\cdots (v_{a_k},w_{b_k}),\cdots ,(v_{a_{i+j}},w_{b_{i+j}})]$
contained in $\si\times \tau$. Then the space $\si\times \tau$ is triangulated as the sum
$ \sum_{s} \sgn (s) \Delta_s$ where $\sgn (s)$ is the sign of the $(i+1,j+1)$-shuffle associated to $s$.
Let $D^1=\sum_n\si_n$ and $D^{d-1}=\sum_m\tau_m$ be triangulations of $D^1$ and $D^{d-1}$ respectively.
We triangulate each $\si_n\times \tau_m$ as above. This gives a triangulation $K= \sum_{n,m,s} \sgn (s) \Delta_{n,m,s}$
of $D^d$. We assume that for
each $n$, $\si_n=[v_0,v_1,v_2]$ with $\si_n\cap H=\{v_2\}.$ We use  a Thom  form $T_\rho$ defined in  \ref{Thom form}
to calculate the face map $\part_H.$ Let $p_1:\,\,D^d\to D^1$ be the projection to the first factor. We have $T_\rho=
p_1^*(\frac1{2\pi i}d\rho\we \frac{dz_1}{z_1})$.  We choose $\epsilon$ sufficiently small so that the form $T_\rho$
is zero on the simplices which do not meet $H$. Take a good ordering $\cal O$ of $K$. 
For a simplex $\Delta_{n,m,s}=[(v_{a_0},w_{b_0}),\cdots (v_{a_k},w_{b_k}),\cdots ,(v_{a_{2d}},w_{b_{2d}})]$, we have
\[T_\rho\overset{\cal O}\cap \Delta_{n,m,s}=\left(\int_{[(v_{a_0},w_{b_0}), (v_{a_1},w_{b_1}), (v_{a_2},w_{b_2})]}
p_1^*(\frac1{2\pi i}d\rho\we \frac{dz_1}{z_1})\right) [(v_{a_2},w_{b_2}),\cdots, (v_{a_{2d}},w_{b_{2d}})]\]

For each $\si_n=[v_0,v_1,v_2]$ the face $[v_0,v_1]$ does not meet the face $H$. Hence the above integral 
can be nonzero only if 
$$[(v_{a_0},w_{b_0}), (v_{a_1},w_{b_1}), (v_{a_2},w_{b_2})]=[(v_0,w_0), (v_1,w_0), (v_2,w_0)].$$
It follows that
$$\begin{aligned}
\part_H(V_x)&=T_\rho\cap V_x\\
&=\sum_m\left(\int_{D^1}\frac1{2\pi i}d\rho\we \frac{dz_1}{z_1}\right)\tau_m\\
&=\sum_m \left(\int_{|z_1|=1}\frac1{2\pi i} \frac{dz_1}{z_1}\right)\tau_m\\
&=\sum_m\tau_m =D^{d-1}
\end{aligned}$$
and so that the multiplicity of $\part_H V$ at $x$ is one. 
\newline We prove  the general case
by induction on $\dim V$. Let $d$ be an integer greater than 1, and assume that the assertion 
holds when $\dim V=d-1$. We consider the case where $\dim V=d.$ 
Suppose that $\dis \part_H^{alg}V=\sum_D m_D^a D$ where the sum is taken over the
irreducible components of $V\times_{P^n} H.$ Let $S$ be a hyperplane section of $P^n$
which meets $V_{sm}$ and $D_{sm}$ transversely 
for each $D$. Let $K$ be a good triangulation of $P^n$ such
that $V$, $V\cap S$ and $D\cap S$ are subcomplexes of $K$ for each $D$. Then the chain
$\part_H V$ is of the form $\dis \sum_\si m_\si \si$ where the sum is taken over 
$(2d-2)$-simplexes of $K$ contained in some $D$. 
\begin{lemma}
\label{const}
When $\si$ and $\si'$ are $(2d-2)$-simplexes of $K$ contained in the same irreducible 
component $D$, then we have $m_\si=m_{\si'}.$
\end{lemma}
\begin{proof}
\begin{lemma}
\label{exac2}
Let $D$ be an irreducible component of $V\times_{P^n} H$ and $F$ be a $(2d-3)$-
simplex of $K$ contained in $D$. Then $F$ is a common face of exactly two
$(2d-2)$-simplexes $\si$ and $\si'$ contained in $D$.
\end{lemma}
\begin{proof}
By Proposition 2.9.10 of \cite{BCR} the set $F\cap D_{sm}$ is a union of regular 
submanifolds of $D_{sm}.$ Let $x$ be a smooth point of $F\cap D_{sm}$ not contained
in a simplex of $K$ of smaller dimension than $2d-3$. Then there is a neighborhood
$U_x$ of $x$ in $D_{sm}$ and coordinates $(x_1,\cdots, x_{2d-2})$
of $U_x$ as a real manifold, such that $U_x\cap F=\{x_1=0\}.$
The connected components of $U_x-F$ are $\{x_1<0\}$ and $\{x_1>0\}.$
If $F$ is  a common face of $(2d-2)$-simplexes $\si_1,\cdots, \si_k$, then the
connected components of $U_x-F$ are $(\si_1-F)\cap U_x, \cdots, (\si_k-F)
\cap U_x$. It follows that  $k=2.$
\end{proof}
By Lemma \ref{exac2}, if $\si$ and $\si'$ are $(2d-2)$-simplexes in $D$ with
a common $(2d-3)$-face, then we have $m_\si=m_{\si'}$ since 
$\dis \part_H V=  \sum_\si m_\si \si$ is topologically a cycle. For a positive
integer $m$ let $C_m$ be the chain $\dis m \sum \si$
where the sum is taken over $(2d-2)$-simplexes in $D$ such that $m_\si=m$.
We have $\dis D=\sum_m |C_m|$ where the left hand side of the equality is
understood to be the set of complex points of $D$. We need to show that 
$D=|C_m|$ for a single $m$. 
\begin{lemma} 
\label{intsing}
For distinct positive integers $m$ and $m'$  we have $|C_m|\cap |C_{m'}|
\subset D_{sing}.$ 
\end{lemma}
\begin{proof} We will show that if $F$ is a simplex in $D$ contained in more than
one $|C_m|$, then $F\subset D_{sing}.$ We give a proof of this by contradiction.
We can assume that any simplex in $D$ which strictly contains $F$
does not belong to more than one $|C_m|.$ Suppose that $F\cap D_{sm}$
is not empty. By Proposition 2.9.10 \cite{BCR} $F\cap D_{sm}$is a union
of regular submanifolds of $D_{sm}.$ Let $x$ be a smooth point of
$F\cap D_{sm}$ which is not on any face of $K$ strictly smaller than $F$.
Let $\dim F=2d-2-j$. By Lemma \ref{exac2} and its consequence $j$ must be bigger than
1. There is a neighborhood $U_x$ of $x$ in $D_{sm}$ and coordinates
$(x_1,\cdots, x_{2d-2})$ of $U_x$ such that $U_x\cap F=
\{x_1=\cdots =x_j=0\}$. If $F$ belongs to k $|C_m|'$s, then $U_x-F$ is a disjoint 
union of k closed sets $(|C_m|-F)\cap U_x$.  
Since $U_x-F$ is connected k must be one. This is a contradiction. 
\end{proof}
By Lemma \ref{intsing} we see that the set of complex points of $D_{sm}$ is a disjoint union
of closed sets of the form $|C_m|-D_{sing}$. Since $D_{sm}$ is a connected topological space,
it must be equal to $|C_m|-D_{sing}$ for a single $m$. We conclude that $D=|C_m|$.
\end{proof}
By Lemma \ref{const} the chain $\part_H V$ is of the form
$\dis \sum_D m_D D$. We need to show that $m_D=m_D^a$ for each $D$.
We can define the face map $\part_S$ in the same way as we defined
$\part_{i,\al}$. We have the equality 
\[\part_S\part_H V=\sum_D m_D\part_S D.\]
Since $S$ meets $D_{sm}$ transversely for each $D$, the intersection multiplicity 
is one for each $D$ i.e. $\part_S D$ is of the form $\sum \tau$ where the sum
is taken over all $(2d-4)$-simplexes contained in $D\cap S.$ So for a $(2d-4)$-simplex
$\tau$ in $D$ not contained in other irreducible components of $V\times _{P^n} H$,
the coefficient of $\tau$ in the chain $\part_S\part_H V$ is equal to $m_D.$
By Proposition \ref{cup and cap projection formula for simplicial} and the proof of Proposition 
\ref{total face map is a differential} we have the equality
\[\part_S\part_H V=\part_H\part_S V.\]  Since $S$ meets $V_{sm}$
transversely, $\part_S V$ is the chain defined by the set of complex points of
$V\times_{\P^n} S$, which is equal to the chain $\part^{alg}_S V$. By induction
hypothesis we have the equality
\[\part_H^{alg}\part_S^{alg} V=\part_H\part_S V.\] 
The equality
\[\part^{alg}_S\part^{alg}_H V=\part^{alg}_H\part^{alg}_SV \] 
certainly holds, and we have
\[\part^{alg}_S\part^{alg}_H V=\sum_D m_D^a \part_S^{alg} D.\]
As $S$ meets $D_{sm}$ transversely for each $D$, we have 
$\part_S^{alg} D=\part_S D$ for each $D$. It follows that for a $(2d-4)$-simplex 
$\tau$ in $D$ not contained in other irreducible components of $V\times _{P^n} H$,
the coefficient of $\tau$ in the chain $\part^{alg}_S\part^{alg}_H V$ is equal to $m_D^a$.
It follows that $m_D=m_D^a$ for each $D$.

By the above argument We are reduced to the case where $\dim V=1.$ In this case let $x\in V\cap H$, and let $V_x$ be a neighborhood of $x$ in $V$
for the complex topology. The function $z_1$ defines a map $f_{z_1}$ from $V_x$ to $\CC$. 
If we take $V_x$ sufficiently small, then the multiplicity of $\part_H^{alg} V$ at $x$ is equal to the degree of the map $f_{z_1}$
which we denote $\deg f_{z_1}$.  We can assume that $f_{z_1}^{-1}(0)=x$.
Let $\epsilon$ be sufficiently small so that 
$$B_{\epsilon}:=\{z\in \C\,|\,|z|\leq \epsilon\}\subset f_{z_1}(V_x)$$
and let $B_\epsilon =\sum_j\si_j$ be a triangulation of $B_\epsilon$ such that $0$ is a vertex. 
 Since the map $f_{z_1}$ is a local isomorphism on $V_x-\{x\}$,  the sum
$\sum_jf_{z_1}^{-1}(\si_j)$ defines a triangulation of $f_{z_1}^{-1}(B_{\epsilon}).$ The chain $\part_H(V_x)$ is calculated
as
\[\sum_j\int_{f_{z_1}^{-1}(\si_j)}f_{z_1}^*(\frac1{2\pi i}d\rho\we \frac{d z}{z})=\deg f_{z_1} \int_{|z|=\epsilon}
 \frac1{2\pi i}\frac{d z}{z}=\deg f_{z_1}.\]

\end{proof}

\section{The Generalized Cauchy formula}
\label{sec:Generalized Caychy formula}

\subsection{Statement of the generalized Cauchy formula}

Let $K$ be a good triangulation of $P^n$. 
We define the rational differential form $\omega_n$ on $P^n$ by
$$\omega_n =\frac{1}{(2\pi i)^n}\frac{dz_1}{z_1}\wedge
\cdots \we\frac{dz_n}{z_n}.
$$ 
\index{$\omega_n$}
By applying  the Main Theorem 4.4 of \cite{part I} to the case where $A=\si$, $m=\dim \si=n$ and $\omega=\omega_n$
we have the  following theorem.
 \begin{theorem} 
\label{convadmiss}
Let $\si$ be an admissible $n$-simplex of $K$.  
The integral 
\begin{equation}
\label{convergenct integral1}
\int_\sigma \omega_n
\end{equation}
converges absolutely.

\end{theorem}
Using Theorem \ref{convadmiss}, the following is well defined.
\begin{definition}
Let $\gamma$  
be an element in $AC_n(K,\bold D^n;\QQ)$, and let $\sum_{\sigma}a_\si \sigma$
be a representative of $\gamma$ in $C_n(K;\QQ)$.
We define $I_n(\gamma)\in \CC$
by\index{$I_n$}
\begin{equation}
\label{def: def of I}
I_n(\gamma)=(-1)^{\frac{n(n-1)}2}\sum_{\sigma}\int_{\sigma}a_\sigma\omega_n
\end{equation}
where the sum is taken over the $n$-simplexes $\si$ not contained in $\bold D^n$.
The right hand side of (\ref{def: def of I}) does not depend on the choice of a representative of $\gamma$.
The map $I_n$ is compatible with subdivisions of triangulations, and we obtain a map
$$
I_n:AC_n(P^n,\bold D^n;\QQ)\to \CC.
$$
\end{definition}
In this section, we prove the following theorem.
\begin{theorem}[Generalized Cauchy formula]
\label{th:generalized Cauchy formula}
\index{generalized Cauchy formula}
Let 
$\ga$ be an element in
 $AC_{n+1}(P^n,\bold D^n;\QQ)$. 
Then we have the equality
\begin{equation}
 \label{eq:cauchy formula} 
I_{n-1}(\part \gamma)+(-1)^nI_n(\delta \gamma)=0.
\end{equation}
\end{theorem}
Let 
$\ga=\sum a_\si \si$ be an element of $AC_{n+1}(K,\bold D^n;\QQ)$ 
for a good triangulation $K$.
By setting 
$\part \ga=\sum c_\tau\tau $ and $\delta\ga=\sum  b_\nu\nu$, 
the equality (\ref{eq:cauchy formula}) can be written as
\begin{equation}
\label{explicit equation}
\sum_\nu\int_\nu b_\nu  \omega_n=
\sum_\tau\int_\tau c_\tau \omega_{n-1}.
\end{equation}
Here  it is understood that  the integral of $\omega_n$ on a simplex contained in $\bold D^n$
is zero by definition.

\begin{remark}
For a $n$-simplex $\nu$ contained in $\bold D^n$, the integral $\int_\nu\omega_n$ is defined to be zero.
This will be justified as follows. Let $\si$ be an admissible $(n+1)$-simplex of a good triangulation $K$.
If $\si\subset \bold D^n$, then $\part \si$ resp.  $\delta \si$ is contained in 
$\bold D^{n-1}$ resp. $\bold D^n$, and $\si$ is irrelevant to the problem. 
Suppose that $\si\not\subset \bold D^n$ and let $\nu$ be an $n$-simplex contained in $\si\cap \bold D^n$. If $\nu\not\subset \bold H^{n}$, then
the  integral $\int_\nu\omega_n$ should be zero since the restriction of $\omega_n$ to $\bold D^n-\bold H^{n}$
is zero. The critical case is that $\nu\subset \bold H^{n}$. In this case we have $\nu=\si\cap \bold H^{n}$
since $K$ is a good triangulation, $\si$ is admissible and $\si\not\subset \bold D^n$. This case will be considered in Proposition \ref{cauchy: dn case},
and it will be clear that our definition is correct.
\end{remark}
\subsection{Outline of the proof of Theorem \ref{th:generalized Cauchy formula}}

Let $\gamma$ be an element in $AC_{n+1}(K,\bold D^n;\QQ)$ and
$\sum_{\sigma} a_\sigma\sigma$
be  a representative of $\ga$ in $C_{n+1}(K;\QQ)$.
We define elements $\gamma_{\bold D}$ and $\gamma_{\bold D^c}$
in 
$C_{n+1}(K,\bold D;\QQ)$ by
\begin{align}
\label{def def of localization w.r.t. D}
\gamma_{\bold D}=&\sum_{\sigma\cap \bold H^n\subset \bold D^n}
a_\si \si,
\\
\nonumber
\ga_{\bold D^c}=&\gamma-\gamma_{\bold D}.
\end{align}
Then 
$\gamma_{\bold D}$ is an element in 
$AC_{n+1}(K,\bold D^n;\QQ)$, 
and as a consequence 
$\gamma_{\bold D^c}$ is also an element in 
$AC_{n+1}(K,\bold D^n;\QQ)$.
Theorem \ref{th:generalized Cauchy formula} is a consequence of the following
Proposition \ref{cauchy: dn case} and Proposition \ref{cauchy not contained in D}.
\begin{proposition}
\label{cauchy: dn case}
Let 
$\ga$ be an element in 
$AC_{n+1}(K,\bold D^n;\QQ)$.
Then 
we have $I_n(\delta\gamma_{\bold D})=0$. 
\end{proposition}
Since $\part\ga_{\bold D}=0$ by definition, Theorem \ref{th:generalized Cauchy formula} holds
for $\ga_{\bold D}$.
The proof of Proposition \ref{cauchy: dn case}
is given in \S \ref{subsubsection:contained in D}.
\begin{proposition}
\label{cauchy not contained in D}
Let $\gamma$ be an element in $AC_{n+1}(K,\bold D^n;\QQ)$. Suppose that $\ga$
has a representative  $\sum_{\sigma}a_\si \sigma$
 in $C_{n+1}(K;\QQ)$ such that $a_{\sigma}=0$ if $\si\cap \bold H^n\subset\bold D^n$. 
Then Theorem \ref{th:generalized Cauchy formula} holds for $\gamma$.
\end{proposition}
Note that if $\ga$ has a representative which satisfies the above condition, then it is  uniquely determined.  Therefore we denote $|\sum_{\sigma}a_\si \sigma|$
by $|\ga|$. 
\newline In the above notation the chain $\ga_{\bold D^c}$ satisfies the assumtion of Proposition 
\ref{cauchy not contained in D}.  If the assertion of Theorem \ref{th:generalized Cauchy formula} holds
for $\ga_{\bold D}$ and $\ga_{\bold D^c}$, then it also holds for $\ga=\ga_{\bold D}+\ga_{\bold D^c}$
since the maps $I_n$ are additive for the addition of chains.

Let $\{H_{i,\alpha}\}$ be the set of codimension one cubical
faces defined in \S \ref{subsec:def of admissible chain}. 
We define $\bold H_h$ as the union of
higher codimensional cubical faces, i.e.
$$
\bold H_h=\underset{\substack{1\leq i< i'\leq n,\\ \alpha\in \{0,\infty\},\,\beta\in \{0,\infty\}}}
\bigcup  (H_{i,\alpha}\cap H_{i',\beta}).
$$
In \S \ref{sec:Cauchy formula for condimension one face},
we prove the following theorem.
\begin{proposition}
[Generalized Cauchy formula for codimension one face]
\label{prop:Cauchy formula for codimension one}
Let $\gamma$ 
be an element in $AC_{n+1}(K,\bold D^n;\QQ)$. Suppose that $\ga$ has the representative $\ga=\sum_{\sigma}a_\si \sigma$ 
 in $C_{n+1}(K;\QQ)$ such that (1) $a_\sigma=0$ if $\sigma\cap \bold H^n \subset \bold D^n$ and 
(2)  $|\ga|\cap \bold H_h=\emptyset$.
Then  Theorem \ref{th:generalized Cauchy formula} holds
for $\ga$.
\end{proposition}

Proposition \ref{prop:Cauchy formula for codimension one} will be proved in \S \ref{sec:Cauchy formula for condimension one face}.
Proposition \ref{cauchy not contained in D} is a consequence of Proposition \ref{prop:Cauchy formula for codimension one}
and a limit argument which will be given in \S \ref{limit argument}.

\subsection{Proof of Proposition \ref{prop:Cauchy formula for codimension one}}
\label{sec:Cauchy formula for condimension one face}
First we observe the following fact.
\begin{lemma}
\label{unique codim 1 face}
Let $K$ be a good triangulation of $P^n$, and let $\si\in K$ be a simplex
such that $\si\cap \bold H^n\neq \emptyset.$ If $\si\cap \bold H_h=\emptyset$,
then there is a unique codimension one cubical face $H_{i,\al}$ such that $\si\cap \bold H^n\subset
H_{i,\al}$. Moreover we have $\si\cap H'=\emptyset$ for any other cubical face $H'$.
\end{lemma}
\begin{proof}
This follows from the facts that $\bold H^n=\underset{i,\al}\cup H_{i,\al}$, 
that each $H_{i,\al}$ is a subcomplex of $K$ and that $\si\cap \bold H^n$ 
is a simplicial face  of $\si$.
\end{proof}  
Let $\gamma$ 
be an element in $AC_{n+1}(K,\bold D^n,\QQ)$ , and let $\sum_{\sigma}a_\si \si$ be the representative of $\ga$ 
satisfying the condition of
Proposition \ref{prop:Cauchy formula for codimension one}. 
We set    $\displaystyle \gamma_{i,\alpha}=\sum_{\substack{\sigma\cap \bold H^n\neq \emptyset \\
\sigma\cap \bold H^n\subset H_{i,\alpha}}}a_\si \si$. Then by Lemma \ref{unique codim 1 face} we have the equality
\begin{equation} 
\label{decomp codim one to components}
\gamma= \sum_{(i,\alpha)}\gamma_{i,\alpha}
\end{equation} and $|\gamma_{i,\alpha}|\cap H'=\emptyset$ for any 
cubical  face $H'$ other than $H_{i,\alpha}$.  
Therefore,
each $\gamma_{i,\alpha}$ is an element in
$AC_{n+1}(K,\bold D^n,\QQ)$. It suffices to prove the assertion for
each $\ga_{i,\al}$. 
Therefore to prove Proposition \ref{prop:Cauchy formula for codimension one},
we may and do  assume that $|\gamma|\cap \bold H^n\subset H=  H_{1,0}$.
\begin{notation}
\label{front face}
We fix a good ordering of $K$ with respect to $ H$. 
For an $(n+1)$-simplex $\sigma=[v_0,v_1,\cdots ,v_{n+1}]$ such that that $v_0<v_1<\cdots <v_{n+1}$, we denote
$\sigma_f=[v_0,v_1, v_2]$ and $ \sigma_b=[v_2,\cdots, v_{n+1}]$.
\end{notation}

To compute the image $\partial\gamma$ of the face map, we choose a Thom form as follows.
Let $\rho:\bold R_+\to [0,1]$ be a $C^{\infty}$ function
such that 
$$
\rho(r)=
\begin{cases}
0 &\quad (r\leq \dfrac{1}{2}),\\
1 &\quad (r\geq 1).
\end{cases}
$$
Let $\e$ be a  positive number, and let $\rho_{\epsilon}$ be the function on 
$P^1=\PP^1_{\CC}$
defined by $\rho_{\epsilon}(z_1)=\rho(\dfrac{|z_1|}{\epsilon})$. 
The function on $P^n$ given by
$(z_1, \dots, z_n)\mapsto \rho_{\epsilon}(z_1)$
is also denoted by $\rho_\e$. We take $\e$ small enough so that the set $W=\underset{\si\in K,\, \si \cap H=\emptyset}
\cup \si$ is contained in the set $\{|z_1|\geq \e\}$. 
Then $T_\e=d\rho_{\epsilon}\we \omega_1$ is a Thom form i.e. $T_\e(\si)=\int_\si T_\e=0$ for each 2-simplex $\si$  in $W$.

\noindent We set $\delta \gamma=\sum_{\nu}b_\nu \nu$.  
Using the above Thom form, the image of $\gamma$ under the face map is
computed as
$$
\part\ga
=\sum_{\si}
\biggl(\int_{\si_f}
d\rho_{\epsilon}\we \omega_1\biggr) 
a_{\sigma}
\si_b
$$
Note that if $\si_f \subset W$, then the above integral is zero. If $\si_f\not\subset W$, then we have
$\si_b\subset H$ since $K\cap H$ is a full subcomplex and we take a good ordering.  Therefore the assertion (\ref{explicit equation})
is written as follows:
\begin{equation}
\label{codim one cauchy}
\sum_\nu\int_\nu b_\nu \omega_n=
\sum_{\si_b\subset H} a_{\sigma} \int_{\si_f}d
 \rho_{\epsilon}\we \omega_1
\cdot \int_{\si_b}
  \omega_{n-1}.
\end{equation}

Let $\si$ be an $(n+1)$-simplex of $K$ such that $a_\si\neq 0$. Since the form $a_{\sigma}\rho_{\epsilon}\omega_n$ is smooth on a neighborhood
of $\sigma$,  we have the equality
\begin{equation}
\label{dump function and stokes}
\int_{\sigma}a_{\sigma}d\rho_{\epsilon}  \we
 \omega_n=
\int_{\delta\sigma}a_{\sigma}\rho_{\epsilon}  \omega_n
\end{equation}
by the Stokes formula Theorem 2.9 \cite{part I}.
The summation of the right hand side of 
(\ref{dump function and stokes}) for $\sigma$ is
equal to
$$
\sum_{\sigma}\int_{\delta\sigma} a_{\sigma}\rho_{\epsilon}
\omega_n
=\sum_{\nu}\int_{\nu}b_{\nu}\rho_{\epsilon} \omega_n.
$$
By Theorem 4.1 and Lebesgue's convergence theorem, we have
\begin{align*}
\lim_{\epsilon \to 0}\int_{\nu}\rho_{\epsilon} b_{\nu} \omega_n&=
\int_{\nu} b_{\nu} \omega_n
\end{align*}
for each $\nu$ with $b_\nu\neq 0$. 
By summing up  (\ref{dump function and stokes})
for all $\sigma$ and taking the limit for $\epsilon \to
0$, we have
$$
\lim_{\epsilon \to 0}\sum_{\sigma}\int_{\sigma}a_{\sigma}d\rho_{\epsilon} \we 
 \omega_n=\sum_\nu\int_{\nu}b_\nu \omega_n.
$$

Comparing with (\ref{codim one cauchy}), 
to prove Proposition \ref{prop:Cauchy formula for codimension one},
it is enough to show the
equality:
\begin{equation}
\label{first goal} 
\lim_{\epsilon \to 0}\sum_{\sigma}\int_{\sigma} a_{\sigma}d\rho_{\epsilon} \we
 \omega_n
=\sum_{\si_b\subset H} a_{\sigma} \int_{\si_f}d
 \rho_{\epsilon}\we \omega_1
\cdot \int_{\si_b}
  \omega_{n-1}
\end{equation}
We reduce the proof of (\ref{first goal})
to the case where $|\gamma|\cap H$ is a simplex.
For this purpose, we prepare the following definition.
\begin{definition}
\label{def of localization at tau1}
Let $\gamma=\sum a_\si \si$
be an element in $C_{n+1}(K;\QQ)$, and 
let $\tau$ be a simplex such that $\tau \subset H=H_{1,0}$.
We define an element
$\gamma^{(\tau)}$\index{$\gamma_{H}^{(\tau)}$} in 
$C_{n+1}(K,\bold D^n;\QQ)$ by
\begin{equation}
\label{def def of localization w.r.t. tau}
\gamma^{(\tau)}=\sum_{\si\cap H= \tau }
a_\si \si.
\end{equation}
\end{definition}
Suppose that $\ga\in AC_{n+1}(K,\bold D^n;\QQ)$ has the representative $\ga=\sum_\si a_\si \si$ such that
(1) $a_\si=0$ if $\si\cap \bold H^n\subset \bold D^n$ and (2) $|\ga|\cap \bold H^n\subset H=H_{1,0}.$ 
Then we have an equality
\begin{equation}
\label{written as sum}
\ga=\sum_{\substack{\tau\subset |\ga|\cap H\\
\tau\not\subset \bold D^n}}\ga^{(\tau)}.
\end{equation}

\begin{proposition}
\label{proper meeting}
Let $\ga$ be an element of $ AC_{n+1}(K,\bold D^n;\QQ)$ and
let $\ga=\sum_\si a_\si \si$ be  its representative
in $C_{n+1}(K;\QQ)$
such that 
\begin{enumerate}

\item  $a_\sigma=0$ if $\si \cap \bold H^n\subset \bold D^n$,
\item $|\gamma|\cap \bold H^n\subset H=  H_{1,0}$ and

\item $|\gamma|\cap \bold H_h=\emptyset$.
\end{enumerate}
Let $\tau$ be a simplex in $|\gamma|\cap H$ not contained
in $\bold D^n$. Then 
$\gamma^{(\tau)}$
is an element in
$AC_{n+1}(K,\bold D^n;\QQ)$.
\end{proposition}
\begin{proof}
We prove that
$|\delta \gamma^{(\tau)}|$ is admissible. Since $|\ga|\cap \bold H_{h}=\emptyset$, we have 
$|\ga|\cap H'=\emptyset$ for any other cubical face $H'$  than $H$ by Lemma \ref{unique codim 1 face}. So it suffices to show that   $|\delta \gamma^{(\tau)}|
-\bold D^n$ meets $H$ properly.
Since $|\ga|- \bold D^n$ meets $H$ properly, 
we have $\dim \tau \leq n-1$. If $\dim \tau<n-1$, 
then $|\delta(\ga^{(\tau)})|- \bold D^n$ 
meets $H$ properly since $|\delta(\ga^{(\tau)})|\cap H\subset \tau$. 
We consider the case where $\dim \tau=n-1$. We have
$$ 
\delta(\ga^{(\tau)})=
\sum_{\nu\subset|\gamma^{(\tau)}|} b_\nu \nu,\, \, 
b_\nu=\sum_{\substack{
\nu\p1 \si \\
\si \cap {H}=\tau}}
[\si:\nu]a_\si
$$
and 
$$
\delta\ga =\sum_{\nu\subset|\gamma|} c_\nu \nu,\,\, 
c_\nu=\sum_{\nu\p1 \si}[\si:\nu]a_\si
$$
Here $\nu\prec \si$ means that $\nu$ is a codimension one face of $\si$, and
$[\si:\nu]\in \{\pm 1\}$ is the index of $\si$ with respect to $\nu$.  
To prove the admissibility of $\delta\ga^{(\tau)}$,
it is sufficient to show the following claim. 

\noindent
{\bf Claim}. Let  $\nu$ be an $n$-simplex in $K$ such that 
(1) $\nu\subset |\ga^{(\tau)}|$,
and (2) $\nu-\bold D^n$ does not meet $H$ properly. Then 
we have  $b_\nu=0$. 

\noindent
\noindent
{\it Proof of the claim}.
Let $\nu$ be a simplex with the conditions in the claim. 
Then
$\nu\cap H$ is a face of $\nu$,
since $K$ is a good
triangulation.
 By condition (1),
we have $\nu\cap H\subset |\gamma^{(\tau)}|\cap H=\tau$
and $\dim(\nu\cap H)\leq \dim(\tau)$.
By condition (2), we have $n-1\leq \dim \nu\cap H$.
Since
$\dim \tau= n-1$,
the above inequalities are equalities,  and 
we have $\nu\cap H=\tau$.
We consider  each term appeared on the right hand side of the equality defining $c_{\nu}$.
Let $\si$ be a $n+1$-simplex 
such that $\nu\subset \si$ and 
$a_\si\neq 0$. The set $\si\cap H$ is a face of $\si$. 
By the admissibility condition, we have $\dim \si\cap H\leq n-1$. 
Since $\dim \tau=n-1$ and $\tau\subset \nu\subset \si$,
we have 
$\si\cap H=\tau$.  
Thus this term appears in the right hand side of the equation defining  $b_{\nu}$.
So we have $b_\nu=c_\nu$. Since $c_\nu=0$ by the admissibility of $\delta\ga$, we have $b_\nu=0$.

\end{proof}
By Proposition \ref{proper meeting} the chain $\ga^{(\tau)}\in 
AC_{n+1}(K, \bold D^n;\QQ)$
for each $\tau$. It suffices to prove (\ref{first goal}) 
for each $\ga^{(\tau)}$ which we do.  So until the end of 
\S \ref{sec:Cauchy formula for condimension one face}, 
we assume
that $\ga=\ga^{(\tau)}$ for a simplex $\tau\subset H$.

Let $\si$ be a $(n+1)$-simplex in $\ga$ and assume given a smooth $(n-1)$-form $\varphi$
on a neighborhood $U$ of $\si$. 
The inclusion $i:\,H\to P^n$ restricts
to an inclusion $i_U:\, U\cap H\to U$
(when there is no fear of confusion, we abbreviate $i^*_U\varphi$ to $i^*\varphi$.) 
Since $\tau^\circ \subset H$ is a smooth submanifold, 
$i^*\varphi$ restricts to a smooth form
on $\tau^\circ$, denoted by the same $i^*\varphi$ (this is where the facewise regularity is used); it is zero
if $\dim \tau<n-1$. Here $\tau^\circ$ is the relative interior of $\tau$ defined in Section \ref{sectionsemi-alg}.

\begin{definition}[Barycentric coordinate, linear projection]
\label{linear proj}

Let $\sigma=[a_0,\cdots,a_p]$ be a $p$-simplex.
A point  $x$ in $\sigma$ is expressed uniquely as
$x=\sum_{i=0}^k \la_ia_i$ with $\sum_{i=0}^k\la_i=1, \lambda_i\geq 0$.
The vector $(\lambda_0,\dots, \lambda_p)$ is called the barycentric coordinate
of $x$.

Let $\sigma=[v_0, \dots, v_{p}]$ be a $p$-simplex
and $\tau=[v_k,\dots, v_{p}]$ be a proper $(p-k)$-face of $\sigma$  
$(0<k \leq p)$.
We set $\tau'=[v_0,\cdots, v_{k-1}]$. 
We define a linear projection\index{linear projection $\pi_{\si,\tau}$, $\pi_{\si}$} 
$
\pi_{\si,\tau}:\si-\tau'\to
 \tau 
$ 
by
$$
\pi_{\si,\tau}(x)
= \frac{1}{\sum_{i=k}^{p}\lambda_i}(\lambda_k,\dots, \lambda_{p}),
$$ 
where $(\lambda_0,\dots, \lambda_p)$ is the barycentric coordinate
of $x$.
\end{definition}

Let $\pi_{\sigma}=\pi_{\sigma,\tau}$ be the linear projection 
$\si^\circ \to \tau$
defined in Definition \ref{linear proj}.
The map $\pi_\si$ restricts to a smooth map between submanifolds
$\si^{\circ}\to \tau^\circ$,
thus the pull-back $\pi_\si^*i^*\varphi$ defines a smooth form on 
$\si^\circ$.

The following proposition will be proved in 
\S \ref{subsec:proof of appriximation}.
\begin{proposition}
\label{difference}
Let $\si$ be a $(n+1)$-simplex in $\gamma$. 
\begin{enumerate}
\item 
We have 
\begin{align}
&\lim_{\epsilon \to 0}
\int_{\sigma} d\rho_{\epsilon} \we\omega_1\we
\big(
a_\si \omega_{n-1}
-
\pi_\si^*i^*(a_\si \omega_{n-1})\big)=0
\end{align}

\item If the dimension of $\tau<n-1$, then the equality
\[
\lim_{\epsilon \to 0}\int_{\sigma}a_{\sigma} d\rho_{\epsilon}  \we \omega_1
 \we\omega_{n-1}
=0\]
holds.
\end{enumerate}
\end{proposition}
We assume Proposition \ref{difference} for the moment and prove the equality
 (\ref{first goal}). For an $(n+1)$-simplex $\si$ in $\ga$, if $\dim \tau=\dim \si\cap H<n-1$, then $\tau$ is a proper face of $\si_b$,
and it follows that the right hand side of  (\ref{first goal}) is zero since $\si_b\not\subset H$ for each $(n+1)$-simplex $\si$. 
Therefore by Proposition \ref{difference} (2),
it is sufficient to prove (\ref{first goal})
for the case where $\gamma=\gamma^{(\tau)}$
and $\dim\tau=n-1$.
Under this assumption, we have $\si_b=\tau$
for each $(n+1)$-simplex $\sigma$ such that $\si\subset |\gamma|$.
\begin{proposition}
\label{nearly fiber 2}
For a  positive $\epsilon$, we have
 the equality
\begin{align}
\label{key equality}
\sum_{\sigma}\int_{\sigma}a_{\sigma} d \rho_{\epsilon}\we \omega_1
\we \pi_\si^*i^*( \omega_{n-1})=
\sum_{\sigma}a_{\sigma}\biggl(\int_{\si_f}d
 \rho_{\epsilon}\we \omega_1\biggr)
\cdot \biggl( \int_{\tau}
 i^*(  \omega_{n-1})\biggr)
\end{align}
\end{proposition}
\begin{proof} 
We recall  
the formulation of projection formula for integrals of differential forms.
Let $M, N$ be oriented smooth manifolds of dimension $m$, $n$, respectively.
Then $M\times N$ is equipped with the product orientation. 
Let $\pi: M\times N\to N$ be the projection to $N$.
For $\vphi$ an $m$-form on $M\times N$ and $\psi$ an $n$-form on $N$, 
we have {\it projection formula}
$$\int_{M\times N} \vphi\we \pi^*\psi= \int_N (\pi_*\vphi)\psi\,.$$
Here $\pi_* \vphi$ is the function 
$$(\pi_* \vphi)(y)=\int_M \vphi|_{M\times \{y\}}\,.$$
(The precise meaning of the equality is that, if the left hand side is absolutely convergent, then 
the function $\pi_*\vphi$ is measurable, the right hand side is also absolutely convergent,
and the equality holds.)  This formula follows from Fubini's theorem for Lebesgue integrals.

For a $(n+1)$-simplex $\si=[v_0,\cdots, v_{n+1}]$  in $\ga$ with
$v_0<\cdots <v_{n+1}$, we have 
$\tau=[v_2,\cdots,v_{n+1}]$ and for a point $t\in \tau$, we have $\pi_\si^{-1}(t)=[v_0.v_1,t]-[v_0,v_1]$
where  $[v_0,v_1,t]$ is the simplex spanned by the points
$v_0,v_1,t$.
By the projection formula, we have the equality
\begin{align}
\label{Fubini and fiber counting multiplicity}
&\sum_{\sigma}a_\si\int_{\sigma} d \rho_{\epsilon}\we \omega_1
\we \pi_\si^*i^*(  \omega_{n-1})
\\
\nonumber
=&
\sum_{\sigma}\int_{\tau}
\bigg(\int_{[v_0,v_1,t]}a_\sigma d \rho_{\epsilon}\we \omega_1\bigg)
i^*\bigg( \omega_{n-1}\bigg)
\\
\nonumber
=&
\int_{\tau}
\sum_{\sigma}\bigg(\int_{[v_0,v_1,t]}a_\sigma d \rho_{\epsilon}\we \omega_1\bigg)
i^*\bigg(  \omega_{n-1}\bigg).
\end{align}

\begin{lemma}
\label{independence}
We have the equality
\begin{equation}
\label{independent of fiber for linear projection}
\sum_{\sigma }
\int_{[v_0,v_1,t]}a_\sigma 
d\rho_{\epsilon} \we \omega_1
=\sum_{\sigma }
\int_{[v_0,v_1,v_2]} a_\sigma 
d\rho_{\epsilon} \we \omega_1
\end{equation}
for all $t\in \tau$.  
\end{lemma}
Note that we have $\si_f=[v_0,v_1,v_2]$ cf. Notation \ref{front face}.


\begin{figure}[h]
 \includegraphics[width=4cm]{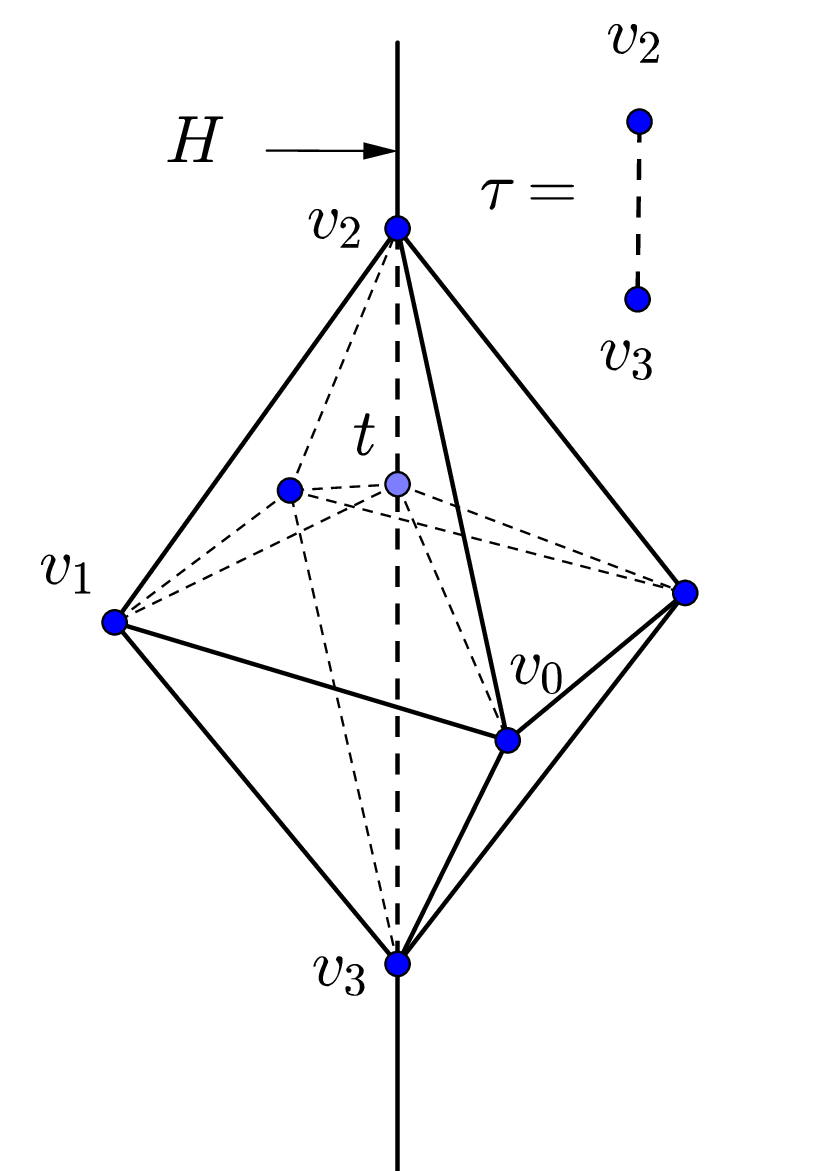}
\end{figure}

\begin{proof}
Let $\sigma$ and $\nu$ be 
an $(n+1)$-simplex
and an $n$-simplex respectively,
such that
$\sigma\succ \nu\succ \tau$.
We set
$\si=[v_0,v_1, v_2, \cdots, v_{n+1}]$, 
$\tau=[v_2, \cdots, v_{n+1}]$ and 
$\nu=[v, v_2, \cdots, v_{n+1}]$. Here the vertex $v$ is $v_0$ or $v_1$. 
For a point $t\in \tau$ we set $\si_t=[v_0,v_1,v_2,t]$
 and $\nu_t=[v,v_2,t]$. Then we have
$[\sigma:\nu]=[\sigma_t:\nu_t]$.
Since $\dim \nu\cap H=\dim \tau=n-1$, $\nu$ is not admissible. Therefore  the coefficient of $\nu$ in $\delta\ga$
is zero by the admissibility of $\delta\ga$. It follows that we have
$$
0=\sum_{\{\sigma|\sigma\succ\nu\}}[\sigma:\nu]a_{\sigma}=
\sum_{\{\sigma|\sigma\succ\nu\}}[\sigma_t:\nu_t]a_{\sigma}.
$$
and obtain the equality
\begin{align}
\label{linking not depend on t}
\sum_\si a_\si\delta \si_t 
=&\sum_\si a_{\sigma}\bigg([v_0,v_1,t]-[v_0,v_1,v_2]\bigg)
+\sum_{\{\nu\mid\nu\succ \tau\}}
\bigg(\sum_{\{\sigma|\sigma\succ\nu\}}[\sigma_t,\nu_t]a_{\sigma}\bigg)\nu_t
\\
\nonumber
=&\sum_\si a_{\sigma}\bigg([v_0,v_1,t]-[v_0,v_1,v_2]\bigg).
\end{align}
The equality
(\ref{linking not depend on t}) implies
the equality
$$
\sum_{\sigma }a_{\sigma}\bigg(
\int_{[v_0,v_1,t]-[v_0,v_1,v_2]}
d\rho_{\epsilon} \we \omega_1
\bigg)
=\sum_{\sigma }a_\si \int_{\delta\si_t}d \rho_{\epsilon}\we \omega_1=0
$$
by the Stokes formula and we finish the proof.
 \end{proof}
Therefore the last line of (\ref{Fubini and fiber counting multiplicity}) is equal to
\begin{align*}
&\int_{\tau}
\sum_{\sigma}\bigg(\int_{[v_0,v_1,t]} a_\si d \rho_{\epsilon}\we \omega_1 \bigg)
i^*\bigg( \omega_{n-1}\bigg)
\\
=&\int_{\tau}
\sum_{\sigma}\bigg(\int_{\si_f} a_\si d \rho_{\epsilon}\we \omega_1 \bigg)
i^*\bigg( \omega_{n-1}\bigg)
\quad(\text{Lemma \ref{independence}})
\\
=&
\bigg(\sum_\si
\int_{\si_f} a_\si d \rho_{\epsilon}\we \omega_1\bigg)
\cdot \int_{
\tau
}
i^*(  \omega_{n-1})\\
\end{align*}
Thus we have proved the assertion.
\end{proof}
Since $\ga=\sum a_\si \si\in AC_{n+1}(K,\bold D^n;\QQ)$, the chain 
$T_\e\cap \ga=\bigg(\sum_\si
\int_{\si_f} a_\si d \rho_{\epsilon}\we \omega_1\bigg)\tau$ is independent of the choice of a Thom cocycle
i.e. the choice of a sufficiently small $\e$ by Proposition \ref{prop face map first properties} (2). 
By Proposition \ref{nearly fiber 2} and Proposition \ref{difference} (1), we have
\begin{proposition}
\label{limit formula}
If  $\dim \tau=n-1$, then for a sufficiently small real number
$\epsilon_0>0$, we have the equality
\[
\begin{array}{l}
\dis \lim_{\epsilon \to 0}\sum_{\sigma}a_{\sigma}\int_{\sigma}d\rho_{\epsilon}  \we \omega_1
  \we \omega_{n-1}
\\\dis = 
\sum_{\si} \biggl(\int_{\si_f}a_\si d\rho_{\epsilon_0} \we  \omega_1\biggr)
\biggl(\int_{\tau} i^*(
 \omega_{n-1})\biggr).
\end{array}
\]
Here the sum is taken over the $(n+1)$-simplexes of $\ga$.

\end{proposition}

The equality (\ref{first goal}) follows from Proposition \ref{limit formula}.

\subsubsection{Proof of Proposition \ref{difference}}
\label{subsec:proof of appriximation}
We prove the following proposition from which
Proposition \ref{difference} follows.
\begin{proposition}
\label{nearly fiber 1} 
Let $\si$ be an $(n+1)$-simplex  in $\gamma$, and let $\varphi$ be a smooth $(n-1)$-form on a neighborhood of $\si$. 
\begin{enumerate}
\item When $\e$ is sufficiently small, the integral
$\int_{\sigma} d\rho_{\epsilon} \we\omega_1
\we\big(
\varphi
-
\pi_\si^*i^*\varphi\big)$
converges absolutely.

\item We have the equality
\begin{align}
&\lim_{\epsilon \to 0}
\int_{\sigma} d\rho_{\epsilon} \we \omega_1\we 
\big(
\varphi
-
\pi_\si^*i^*\varphi\big)=0.
\end{align}
\end{enumerate}

\end{proposition}

\begin{proof}[Proof of Proposition \ref{nearly fiber 1}]
The form $\vphi$ is a sum of the forms $f\, du_1\we\cdots\we du_{n-1}$, where 
$u_i$ are from the set $\{x_1, y_1, \cdots, x_n, y_n\}$ (we denote $z_j=x_j+iy_j$), and $f$ is a smooth function.
One may thus assume $\vphi=f\, du_1\we \cdots\we du_{n-1}$.

(1)
We wish to apply \cite{part I} Theorem 2.6, which reads as follows:  
Let $S$ be a compact semi-algebraic set of dimension $m$, 
$h: S\to \RR^\ell$ be a continuous semi-algebraic map, and $\psi$ be a smooth $m$-form
defined on an open set of $\RR^\ell$ containing $h(S)$. Then the integral
$\int_S |h^*\psi|$
is convergent. 

It is useful to note that differential forms on $S$ of the form $h^*\psi$, with $h:S\to \RR^\ell$ continuous 
semi-algebraic, and $\psi$ a smooth $p$-form
($0\le p \le m=\dim S$) are closed under wedge product. Indeed, if 
$h': S\to \RR^{\ell'}$ is another continuous semi-algebraic map, and $\psi'$ a smooth $p'$-form
on an neighborhood of  $h'(S')$, then $(h^*\psi)\we ({h'}^* \psi')$ equals 
the pull-back by the product map $(h, h'): S\to \RR^\ell\times \RR^{\ell'}$ of 
the smooth form $(p_1^*\psi)\we(p_2^*\psi')$ defined on a neighborhood of $(h, h')(S)$
in $\RR^{\ell+\ell'}$.  

In order to show the absolute
convergence of $\int_\sigma d\rho_\eps \we\omega_1 \we\pi_\sigma^*i^*\vphi$, 
let $S$ be the compact semi-algebraic set obtained from $\si$ by removing a 
small neighborhood of $\tau'$, and note that the integral in question equals
$\int_S d\rho_\eps\we \omega_1\we \pi_\sigma^*i^*\vphi$ for   $\e$ sufficiently small. 
We consider the projection
$\pi_\si: \si-\tau'\to \tau\subset H$
restricted to $S$, 
$$\pi_\si:S\to \tau\subset H\,,$$
and the smooth form $i_U^*\varphi$ defined on a neighborhood of $\tau$; then $\pi_\si ^*i_U^*\varphi$ is a form of the above-mentioned shape $h^*\psi$. 
Also, pull-back by the inclusion $S\injto P^n$ of the smooth form $d\rho_\eps \we\omega_1$
gives us another form of the shape $h^*\psi$.
Thus the wedge product of them, $d\rho_\eps\we \omega_1\we \pi_\sigma^*i^*\vphi$, 
is also a form of the same kind, and we conclude absolute convergence
of $\int_S d\rho_\eps\we \omega_1\we \pi_\sigma^*i^*\vphi$ 
by the theorem we recalled. 

Similarly (and more easily)  the absolute convergence of $\int_\sigma d\rho_\eps\we \omega_1\we\vphi$ is obtained by applying
the same theorem to the inclusion $\si \injto P^n$ and the smooth form 
$d\rho_\eps\we \omega_1\we \vphi$.

(2).  We need the following lemma.
\begin{lemma}
\label{lemma:bounding after integral}
For a complex number $\zeta_1$, we set 
$\sigma(\zeta_1)=\si\cap \{z_1=\zeta_1\}$.  There exists a closed semi-algebraic set $C$ of
$\C$ of dimension $\leq 1$ for which the equality 
\begin{equation}
\label{inequality for fixed r}
\underset{|\zeta_1|\to 0,\,\zeta_1 \notin C}\lim \int_{\sigma(\zeta_1)}
|\varphi-\pi_\si^*i^*\varphi|=0.
\end{equation}
holds.
\end{lemma}

\begin{proof}[Proof of Lemma \ref{lemma:bounding after integral}]
By Semi-algebraic triviality of semi-algebraic maps as stated in Theorem 9.3.2, \cite{BCR},  there exists a closed semi-algebraic
set $C$ of $\C$ of dimension $\leq 1$ such that if $\zeta_1\not\in C$, then the inequality
$\dim \si(\zeta_1)\leq \dim \si-2$ holds.  We have an equality 
\begin{equation}
\label{first equality}
\begin{array}{ll}
& \varphi-\pi_\si^*i^*\varphi\\
=&f
du_1\we\cdots\we  du_{n-1}
-
 \pi_\si^*i^*(f 
du_1\we\cdots\we du_{n-1})
\\
=&
\big(f-\pi_\si^*i^*f\big) 
du_1\we\cdots \we du_{n-1}
\\
&+\sum_{k=1}^{n-1}
\pi_\si^*i^*f

du_1\wedge\cdots \wedge du_{k-1}
\wedge \big( 
du_k-
\pi_\si^*  i^* du_k \big)\wedge
\pi_\si^*i^*(
du_{k+1}\wedge\cdots \wedge du_{n-1}
%
)
\end{array}
\end{equation}
We estimate the integral of the first term on the right hand side of (\ref{first equality}). Let $g$ be the map defined by 
$$
\sigma \to \CC\times \RR^{n-1}:z\to (z_1,u_i).
$$ 
By Proposition 2.7 of \cite{part I}, we have the inequality
\[\bigg|\int_{\sigma(\zeta_1)}
(f-\pi_\si^*f|_\tau) \we
du_1\we\cdots \we du_{n-1}
\bigg|
\leq \underset{\si(\zeta_1)}{{\rm Max}}
|f-\pi_\si^*i^*f|\delta(g)\int_{g(\si(\zeta_1))}|
du_1\we\cdots \we du_{n-1}
|\]
Here $\delta(g)$ denotes the maximal of the cardinalities of finite
fibers of $g$.
For the precise definition, see Definition 2.2. of \cite{part I}.
Note that $\delta(g|_{\si(\zeta_1)})\leq
\delta(g)$.

\vskip 0.5cm

{\bf Claim.}\quad $\Max_{\sigma(\zeta_1)}|f-\pi_\si^*i^*f| \to 0$ as $\zeta_1$ tends to 0
(outside $C$). 
\begin{proof}[Proof of the claim] The function $f-\pi_\si^*i^*f$ is continuous semi-algebraic on $\sigma(\zeta_1)$, and 
vanishes on $\tau$. 
If the claim were false, there exists an $\eps>0$ and a sequence $P_j\in\si$ with $|z_1(P_j)|\to 0$
and  $|(f-\pi_\si^*i^*f)(P_j)|\ge \eps$.
Taking a subsequence we may assume that the sequence converges to a point $P\in \si$.
Then $z_1(P)=0$, thus $P\in \tau $, while $|(f-\pi_\si^*i^*f)(P)|\ge \eps$, 
contradicting the function $f-\pi_\si^*i^*f$ being zero on $\tau$. 
\end{proof}
The integral $\displaystyle\int_{g(\si(\zeta_1))}|\we_i du_i|$ is bounded 
by the volume of $g(\sigma)$ which is independent of $\zeta_1$.
We conclude that
 the integral of the first term on the right hand side of (\ref{first equality}) converges to zero as
$|\zeta_1|\to 0$. 

We estimate the integral of the second term of the right hand side of (\ref{first equality}). 
Let $h$ be the map defined by 
$$
\sigma \to \CC\times \R^{n-1}:
z\mapsto (z_1,v_i)=(z_1,u_1,
\dots, u_{k-1},u_k-\pi_\si^*i^*u_k,\pi_\si^*i^*u_{k+1},
\dots \pi^*_\si i^*u_{n-1}).
$$
By Proposition 2.7 of \cite{part I}, we have the inequality
\[
\begin{array}{l}
\dis \int_{\si(\zeta_1)}|\pi_\si^*i^*f(z)
\we
du_1\we\cdots \we du_{k-1}
\we \big( 
du_k-
\pi_\si^*  i^* du_k \big)\we
\pi_\si^*i^*(
du_{k+1}\we\cdots \we du_{n-1})
|\\
\dis \leq \underset{\si(\zeta_1)}{\rm Max}|\pi_\si^*i^*f|\delta(h|_{\sigma(\zeta_1)})\int_{h(\si(\zeta_1))}
|
dv_1\we\cdots \we dv_{n-1}
|
\end{array}\]
where $v_1,\cdots,v_{n-1}$ are the coordinates of $\R^{n-1}$. 
 Note that $\delta(h|_{\si(\zeta_1)})$ is bounded by $\delta(h)$ which
 is independent of $\zeta_1$. The function $|\pi_\si^*i^*f|$ is bounded by
$\max\{|f(z)|\mid z \in \tau\}$. 
By the same proof as for the Claim, $M_{\zeta_1}:=\Max_{\sigma(\zeta_1)}|u_k-\pi_\si^*i^*u_k|$ tends to 
zero as $|\zeta_1|\to 0$.
 There exist numbers $a<b$ such that
$$\begin{array}{cl}
u_i(\si)\subset [a, b] &\text{for $1\le i\le k-1$, and}  \\
\pi_\si^* i^*u_i(\si)\subset [a, b] &\text{for $k+1\le i\le n-1$}\,,
\end{array}
$$
thus 
$$h(\sigma(\zeta_1))\subset [a, b]^{k-1}\times [-M_{\zeta_1}, M_{\zeta_1}]\times [a, b]^{n-1-k}\,,$$
hence 
$\displaystyle\int_{h(\sigma(\zeta_1))} dv_1\we\cdots \we dv_{n-1}\to 0$.
\end{proof}

We go back to the proof of Proposition \ref{nearly fiber 1} (2).  One has 
$$\int_\CC d\rho\we \omega_1=1\,,$$
as follows from the identity
$$d\rho \we\frac{dz_1}{z_1}=i\rho'(r_1)dr_1 \we d\theta_1\,.$$
For the form $d\rho_\eps(z_1)\wedge \omega_1$,  the change of variables $z_1'=z_1/\eps$
yields
$$d\rho_\eps(z_1)\wedge \omega_1 = d\rho(z'_1)\wedge \frac{1}{2\pi i} \frac{dz'_1}{z'_1}\,.$$
So we have 
$$\int_\CC d\rho_\eps\wedge \omega_1=1\,.$$
Also, $d\rho_\eps\wedge \omega_1$ has support in $|z_1|\le \eps$.
Therefore

\begin{align*}
&\left| \int_\sigma d\rho_\eps\wedge \omega_1 \wedge (\vphi- \pi_\si^*i^*\vphi)\right|
\le \underset{|\zeta_1|\le\eps,\,\zeta_1\notin C}\Max(\int_{\sigma(\zeta_1)} |\vphi- \pi_\si^*i^*\vphi| )\int_\CC d\rho_\eps\wedge \omega_1 \\
&=\underset{|\zeta_1|\le\eps,\,\zeta_1\notin C}\Max(\int_{\sigma(\zeta_1)} |\vphi- \pi_\si^*i^*\vphi| )
\end{align*}

and the assertion follows from Lemma \ref{lemma:bounding after integral} .
\end{proof}

\subsection{Proof of Proposition \ref{cauchy not contained in D}}
\label{limit argument}
We prove 
Proposition \ref{cauchy not contained in D} assuming 
Proposition \ref{prop:Cauchy formula for codimension one}
by a limit argument.
Let $\epsilon$ be a  positive real number
and for $i=1,\cdots, n$, set $z_i^{(0)}=z_i$ and $z_i^{(\infty)}=z_i^{-1}$. 
We define a neighborhood $N_{\epsilon}$ of $\bold H_h$ by
\begin{align*}
N_{\epsilon}&=\underset{\substack{1\leq i< i'\leq n,\\ \alpha\in \{0,\infty\} ,\beta\in \{0,\infty\}}}{\bigcup}
\{z\in P^n\mid |
z_i^{(\alpha)}|\leq \epsilon,
|z_{i'}^{(\beta)}|\leq \epsilon\}.
\end{align*}
We set $N^*_{\epsilon}=\overline{P^n-N_\epsilon}$.

By the semi-algebraic triviality of the semi-algebraic maps (\cite{BCR}, Theorem 9.3.2), the following
 {\bf condition (P)} holds when $\e$ is sufficiently small. 

\vskip 0.3cm

 (P)\quad For each simplex $\si$ of $K$  and each pair $(i,\al)$, 
we have  $\dim \si\cap \{|z_i^{(\al)}|=\e\}\leq \dim \si-1$ if the set  $ \si\cap \{|z_i^{(\al)}|=\e\}$
is not empty. Here if the dimension of a set is negative, then the set must be empty.

\vskip 0.3cm

In the rest of this section and  \S \ref{subsubsection:contained in D} we assume that $\e$
satisfies the condition (P). Let $K_\e$ be a good subdivision of $K$ such that $N_\e$ and $N_\e^*$ are subcomplexes of $K_\e$.
We denote by $\la$ the subdivision operator from $C_\bullet(K)$ to $C_\bullet(K_\e)$.
Note that by the condition (P) we have the following.  If $\si$ is a simplex of $K$ and $\la(\si)=
\sum\si'$, 
then each  $\si'$ is contained in one of $\{N_\e,\,N_\e^*\}$ exclusively.  
 
\begin{definition}
\begin{enumerate}
\item 
Let $\si$ be a $m$-simplex in $K$ and we set $\la(\si)=\sum\si'$. 
 We put $\si_{\geq\e}=\underset{\si'\subset N_\e^*}\sum \si'$. 
 For an element 
$\ga=\sum_{\sigma }  a_\si \si$ in 
$C_m(K,\bold D^n;\QQ)$, we set \newline $\ga_{\geq\e}=\sum_{\sigma }  a_\sigma\sigma_{\geq\epsilon} $.
\item We put $\si_{=\e}=
\delta(\si_{\geq\e})-(\delta\si)_{\geq \e}$. 

\item For an element 
$\ga=\sum_{\sigma }  a_\si \si$ in 
$C_m(K,\bold D^n;\QQ)$,
we set 
$\ga_{=\e}=\sum_{\sigma }  a_\sigma \sigma_{=\epsilon}.$

\end{enumerate}
\end{definition}
\begin{lemma}
\label{lem: admiss}
Let $\e$ be a positive number which satisfies the condition (P). 
\begin{enumerate}

\item Let $\si$ be an $m$-simplex in $K$. We have $|\si_{=\e}|\subset N_\e^*\cap N_\e$.

\item Let $\sigma$ be an admissible $m$-simplex in $K$.
Then the chain  
$\sigma_{=\epsilon}$ is admissible.
\item Let $\ga=\sum  a_\si\si$ be an element of $AC_m(K,\bold D^n; \QQ)$. We have $(\part \ga)_{\geq \e}=\part(\ga_{\geq \e})$.

\end{enumerate}

\end{lemma}
\begin{proof}

(1). Let $\la(\si)=\sum \si'$ and we put $\si_{\leq \e}=\underset{\si'\subset N_\e}\sum \si'$. 
By condition (P), we have $\la(\si)=\si_{\geq\e}+\si_{\leq\e}$ and $\delta(\la( \si))=\la(\delta\si)=(\delta\si)_{\geq\e}+(\delta\si)_{\leq\e}$.
Let $\si_{\e}^-=\delta (\si_{\leq\e})-(\delta\si)_{\leq\e}$.  Then we have 
$\si_{=\e}+\si_{=\e}^-=0$.  Since $|\si_{=\e}|\subset N_\e^*$ and $|\si_{=\e}^-|\subset N_\e$, we have the assertion.

(2). We set $\tau=\si\cap \bold H^n$. If $\tau \subset \bold D^n$, then $\si_{=\e}$ is admissible
by definition. We assume that $ \tau \not \subset \bold D^n$. If $\tau \subset \bold H_h$, then $|\sigma_{=\epsilon}|\cap \bold H^n=\emptyset$
by the definition of $N_{\epsilon}^*$. 
If $ \tau \not\subset \bold H_h$, then there is a unique 
codimension one face $H_{i,\alpha}$
such that $\tau \subset H_{i,\alpha}$.  Since $|\sigma_{=\epsilon}|\cap \bold H_h=\emptyset$, $\sigma_{=\epsilon}$
does not meet other cubical face than $H_{i,\al}.$ Hence it suffices to show that $\sigma_{=\epsilon}$ meets
$H_{i,\al}$ properly. 
By the admissibility of $\si$ we have
$\dim \si \cap H_{i,\alpha}\leq m-2. $ 
 Since the set $\tau=\si \cap H_{i,\al}$ is a face of $\si$,  we have the inequality 
$\dim \tau \cap \{|z_{i'}^{(\alpha')}|=\e\}\leq m-3$
for any $(i',\alpha')$ by the condition (P).
By the assertion (1) we have 
$|\sigma_{=\epsilon}|\cap H_{i,\alpha}\subset \si\cap H_{i,\alpha}\cap \bigl(\underset{(i',\al')}\cup\{|z_{i'}^{(\alpha')}|=\e\}\bigr)$.
Therefore we have $\dim(|\sigma_{=\epsilon}|\cap H_{i,\alpha}) \leq m-3$.

(3).   We denote the cubical face $\{z_1=0\}$ by $H$. 
Let $\si$ be a $m$-simplex in $\ga$. We show that $\part(\si_{\ge \e})
=(\part\si)_{\ge \e}$. We denote the simplex $\si\cap H$ by $\tau$. 
We set $\la(\si) =\sum \si'$ and $\la(\tau)=\sum \tau'$. Then we have $\dis \si_{\geq \e}=\sum_{\si' \subset N_\e^*} \si'.$ 
  Let $T$ be a Thom cocycle for $K_\e\cap H$
and $\cal O$ be a good ordering of $K_\e$ with respect to $H$. 
We set $T\overset{\cal O} \cap \la(\si)=\sum  b_{\tau'}\tau'$ and $T\overset{\cal O} \cap
(\si_{\geq \e})=\sum c_{\tau'}\tau'$. We have $\dis  b_{\tau'}\tau'=\sum_{\si'\cap H=\tau'}
T\overset{\cal O} \cap \si'$ and $\dis  c_{\tau'}\tau'=\sum_{\si'\cap H=\tau',
\, \si'\subset N_\e^*} T\overset{\cal O} \cap \si'$. 

Suppose first that $\tau$ is a $(m-2)$-simplex which is not contained in $\bold D^n$.  
We will show that for each $(m-2)$-simplex $\tau'$ contained
in $N_\e^*$, we have $b_{\tau'}=c_{\tau'}.$ Such a $\tau'$ is not contained in $N_\e$  by the condition (P).
Each $m$-simplex $\si'$ in $\la(\si)$ is contained either $N_\e$ or $N_\e^*$. If $\si'\cap H=\tau'$
then we have $\si'\subset N_\e^*$. It follows that $b_{\tau'}=c_{\tau'}$. 

Suppose that  $\dim \si \cap H\leq m-3$.  
We have $\dim \si'\cap H\leq m-3$ for each $\si'$. If  $\si'=\pm [v_0,\cdots,v_m]$ with $v_0<\cdots< v_m$, then 
then we have  $T\overset{\cal O} \cap \si'=\pm T([v_0,v_1,v_2])[v_2,\cdots, v_m]=0$ since $[v_0,v_1,v_2]\cap H=\emptyset$, and 
$T\in C^2(K_\e, W)$ where $\dis W=\underset{\xi\in K_\e, \xi\cap H=\emptyset}\bigcup
\xi$.  Therefore we have $b_{\tau'}=c_{\tau'}=0$ for each $\tau'$.  The same argument
applies to each codimension one cubical face. 
\end{proof}

\begin{proof}[Proof of
Proposition \ref{cauchy not contained in D} 
assuming Proposition \ref{prop:Cauchy formula for codimension one}]

Let $\ga=\sum_{\sigma}a_\si \si$ be  the representative of $\ga$ 
such that $a_{\sigma}=0$ if $\si\cap \bold H^n \subset \bold D^n$.
We show that $\gamma_{\geq \epsilon}$ is an element in
$AC_{n+1}(K_{\epsilon}, \bold D^n; \QQ)$ as follows.
The  set $|\gamma_{\geq\epsilon}|$ is contained in $|\ga|$ which  is admissible by the assumption.
Hence  $|\gamma_{\geq\epsilon}|$ is admissible. 
We have 
$|\delta(\gamma_{\geq\epsilon})|=|(\delta\gamma)_{\geq\epsilon}|\cup
|\gamma_{=\epsilon}|$.  The set $|(\delta\gamma)_{\geq\epsilon}|$ is admissible by the same reason as above. 
For an $\epsilon>0$ which satisfies (P), 
$|\gamma_{=\epsilon}|$ is admissible by Lemma \ref{lem: admiss} (2).
Therefore
$\gamma_{\geq \epsilon}$ is an element in
$AC_{n+1}(K_{\epsilon}, \bold D^n; \QQ)$.

Since $N_\e^*\cap \bold H_h=\emptyset$,
the element $\ga_{\geq\e}$
satisfies the condition of Proposition \ref{prop:Cauchy formula for codimension one} and
we have
$I_{n-1}(\part(\ga_{\geq\e}))+(-1)^nI_n(\delta(\ga_{\geq\e}))=0$.
We prove the equality
\begin{equation}
\label{goal to limit arguement}
\lim_{\epsilon \to 0}\bigl(I_{n-1}(\part(\ga_{\geq\e}))+(-1)^nI_n(\delta(\ga_{\geq\e}))\bigr)
=
I_{n-1}(\part\ga)+(-1)^nI_n(\delta\ga).
\end{equation}

We set
$\delta\ga=\sum b_\nu \nu$ and
$\part \ga=\sum \tau\cdot c_\tau$.
By the admissibility of $\delta \gamma$ and $\partial \ga$,
the integrals
$\dis \int_{\nu}b_{\nu} \omega_n$ 
and $\dis \int_{\tau}c_\tau \omega_{n-1}$
converge absolutely 
by Theorem \ref{convadmiss}. 
By Lebesgue's convergence theorem,
we have
\[
 \lim_{\epsilon\to 0}\int_{\nu_{\geq\epsilon}}b_{\nu} \omega_n
= \int_{\nu}b_{\nu} \omega_n  \text{ and }
\lim_{\epsilon\to 0}
\int_{\tau_{\geq\epsilon}}c_{\tau}\omega_{n-1}
=\int_{\tau}c_{\tau}\omega_{n-1}.
\]
Therefore we have
\begin{align*}
\nonumber
&\lim_{\epsilon \to 0}I_n((\delta\gamma)_{\geq\epsilon})=I_n(\delta\gamma)
\\
\nonumber
&\lim_{\epsilon \to 0}
I_{n-1}(\part(\gamma_{\geq\epsilon}))=
\lim_{\epsilon \to 0}I_{n-1}((\part\gamma)_{\geq\epsilon})=
I_{n-1}(\part\gamma).
\end{align*}
By the equality $\delta(\gamma_{\geq\epsilon})=(\delta\gamma)_{\geq\epsilon}+\gamma_{=\epsilon}$,
to show the equality (\ref{goal to limit arguement}), 
it is enough to prove the equality 
\begin{equation}
\label{eq: limit 4th term}
\lim_{\epsilon \to 0}I_{n}(\gamma_{=\epsilon})=0.
\end{equation}
For a positive real number $t$ and $1\leq i \neq j \leq n$, $\al\in \{0,\infty\},\,\beta\in \{0,\infty\}$, 
we set
$$
A^{(i,\al),(j,\beta)}_{t}=\{z\in P^n|\  |z_i^{(\al)}|\leq|z_j^{(\beta)}|=t\}.
$$
Let $\si$ be an $(n+1)$-simplex  of $K$.  By Lemma \ref{lem: admiss} (1), the set $|\si_{=\e}|$ is contained
in the topological boundary of $N_\e$. Therefore 
we have the relation
$$
|\sigma_{=\epsilon}|\subset
\underset{\substack{1\leq i\neq j \leq n,\\ \alpha\in \{0,\infty\},\beta\in \{0,\infty\}}}{\bigcup} \si \cap A_{\epsilon}^{(i,\al),(j,\beta)}.
$$
By applying Theorem 4.7 \cite{part I} to the case where
$A=\si$, $m=n+1$ and $\omega=\omega_n$, we obtain  the following proposition.
\begin{proposition}
\label{prop:Contributions from higher codimension 4.7} 
Let $\sigma$ be an $(n+1)$-simplex.
Assume that $\sigma$ is admissible. 
Then for a sufficiently small $t>0$, 
the dimension of $\si\cap A_{t}^{(i,\al),(j,\beta)}$ is 
equal to or less than $n$, and we have
$$
\lim_{t \to 0}
\int_{\sigma\cap A_{t}^{(i,\al),(j,\beta)}}|\omega_n|=0.
$$
\end{proposition}
By Proposition \ref{prop:Contributions from higher
 codimension 4.7}, 
we have
\begin{equation*}
\lim_{\epsilon \to 0}
\int_{\sigma_{=\epsilon}} |\omega_n|\leq
\sum_{\substack{ i\neq j\\ \al,\beta}}
\lim_{\epsilon \to 0}
\int_{\sigma \cap A_{\epsilon}^{(i,\al),(j,\beta)}} |\omega_n|=0.
\end{equation*}
and as a consequence, we have the equality (\ref{eq: limit 4th term}).
\end{proof}

\subsection{Proof of Proposition \ref{cauchy: dn case}}
\label{subsubsection:contained in D}
Supose that an element $\gamma\in AC_{n+1}(K,\bold D^n;\QQ)$
satisfies the condition $|\gamma|\cap \bold H^n \subset \bold D^n$. If $\sum a_\si \si$ is a representative of
$\ga$ in $C_{n+1}(K;\QQ)$, then for each $\si$ with $a_\si\neq 0$ we have $\si\cap \bold H^n\subset \bold D^n$,
and $\si\in AC_{n+1}(K,\bold D^n;\QQ)$. So it is necessary and sufficient to prove the assertion for each such $\si$.
Since $K$ is a good triangulation, for an $(n+1)$-simplex $\sigma$, 
$\si \cap \bold H^n$
 is a face of 
$\si$. 
Therefore $\si\cap  \bold H^n\subset H_{i,\alpha}$
for some $(i,\alpha)$
because $\bold H^n$ is the union of the codimension one faces $H_{i,\alpha}$.
We may assume that $H=  H_{1,0}$.

 For a positive number $\epsilon$ which satisfies the condition (P), we set
$$
N_\epsilon^*=\{|z_1|\geq \epsilon\}\subset P^n \text{ and }
N_\e=\{|z_1|\leq \e\}\subset P^n.
$$
\index{$N$}
Let $K_{\epsilon}$ be a subdivision of $K$ such that $N_{\epsilon}^*$ and $N_\e$
are subcomplexes of $K_{\epsilon}$, and let $\la:\,C_\bullet(K)\to C_\bullet(K_\e)$
be the subdivision operator.  Let $\la(\si)=\sum \si'$, and put $\si_{\geq\e}=\underset{
\si', \si'\subset N_\e^*}\sum \si$. We set $\si_{=\e}=\delta(\si_{\geq \e})-(\delta\si)_{\geq\e}.$ 
By the same argument as the proof of Lemma \ref{lem: admiss} (1),
we have $|\si_{=\e}|\subset \{|z_1|=\e\}$. 

 Since
$\sigma_{\geq\epsilon}$ does not meet $\bold H^n$, we have the equality
\begin{equation}
\label{stokes contained in D}
0=\int_{\delta(\sigma_{\geq\epsilon})} a_\sigma  \omega_n=
\int_{\delta(\sigma)_{\geq\epsilon}} a_\sigma  \omega_n+
\int_{\sigma_{=\e}} a_\sigma  \omega_n
\end{equation}
by the Stokes formula.
We consider the limit of this as $\epsilon \to 0$. 
As the chain $\delta \si$ is admissible, the integral
$\dis \int_{\delta\si}a_\si \omega_n$ converges absolutely. 
By Lebesgue's convergence theorem, 
we have
$$
\lim_{\epsilon \to 0}
\int_{(\delta\sigma)_{\geq\epsilon}}a_\sigma  \omega_n
=\int_{\delta\sigma} a_\sigma \omega_n
$$
By applying Theorem 4.8 \cite{part I} to the case where $A=\si$, $\omega=\omega_n$
,  $m=n+1$ and $\e=1$, we have  the following proposition. 
\begin{proposition}
\label{prop:Contributions from higher codimension 4.8} 
Let $\sigma$ be an $(n+1)$-simplex.
Assume that $\si\cap \bold H^n\subset
\bold D^n$. Then we have
$$
\lim_{t\to 0}\int_{\sigma\cap\{|z_1|=t\}}|\omega_n|=0.
$$
\end{proposition}
By Proposition \ref{prop:Contributions from higher codimension 4.8}, 
we have the equality
$$
\lim_{\epsilon \to 0}
\int_{\sigma_{=\epsilon}}a_\sigma \omega_n=0.
$$
By taking the limit of (\ref{stokes contained in D})
for $\epsilon \to 0$, we have
$$
\int_{\delta\sigma} a_\sigma  \omega_n=
0
$$
and we finish the proof of Proposition \ref{cauchy: dn case}.

\section{Construction of the Hodge realization functor.}
\label{sec: Hodge realization}
In this section, we give a construction of 
Hodge realization functor for the category of mixed Tate motives.

\subsection{Cycle complexes and graded DGA $N$}

Let $\bold k$ be a field.    
Following \cite{BK}, we recall that the cycle complex of $\Spec {\bold k} $ may be viewed as a DGA over $\QQ$. 

Bloch defined the cycle complex for any quasi-projective variety, but we will restrict to the
 case of $\Spec {\bold k}$. 
Let $\sq^n=\sq^n_{\bold k}=(\bold P^1_{\bold k} -\{1\})^n$, which is isomorphic to affine $n$-space 
as a variety (and which coincides with $\sq^n$ of \S 2 if ${\bold k}=\CC$). 
As in \S 2,  $(z_1, \cdots, z_n)$ are the coordinates of $\sq^n$.

For $n\ge 0$ and $r\ge 0$, let $Z^r(\Spec \bold k, n)$\index{$Z(r, n)$} be the $\QQ$-vector space with basis irreducible 
closed subvarieties of $\sq^n$ of codimension 
$r$ which meet the faces properly.   
The cubical differential $\partial: Z^r(\Spec \bold k, n)\to Z^r(\Spec \bold k, n-1)$ is defined by the same formula
as Definition \ref{def face map}. 

The group $G_n= \{\pm 1\}^n\rtimes S_n$\index{$G_n$} acts naturally on  $\s$ as follows. The subgroup
$\{\pm 1\}^n$ acts by the inversion of the coordinates $z_i$, and  the symmetric group $S_n$ acts by
permutation of $z_i$'s. This action induces an action of  $G_n$ on $Z^r(\Spec \bold k, n)$.
Let $\sign: G_n \to \{\pm 1\}$ be the character which sends 
$(\eps_1, \cdots, \eps_n; \sigma)$ to $\eps_1\cdot \cdots\cdot\eps_n\cdot\sign(\sigma)$. 
The idempotent $\Alt=\Alt_n:=({1}/{|G_n|})\sum_{g\in G_n} \sign(g) g $ \index{$\Alt$}
in the group ring $\QQ[G_n]$ is called the alternating 
projector. For a $\Q[G_n]$-module M, 
the submodule
$$
M^{\alt}=\{\al \in M\mid \Alt \al=\al\}=\Alt(M)
$$
\index{$M^{\alt}$}
is called the alternating part of $M$. By Lemma 4.3 \cite{BK}, one knows that the cubical differential
$\part $ maps $Z^r(\Spec \bold k,n)^{\rm alt}$ to $Z^r(\Spec \bold k,n-1)^{\rm alt}$. 
For convenience let $ Z^r(\Spec {\bold k}, n)=0$ if $n<0$. 
Product of cycles induces a map of complexes 
$\times : Z^r(\Spec \bold k, n)\otimes Z^s(\Spec \bold k, m)\to Z^{r+s}(\Spec \bold k, n+m)$, $z\otimes w\mapsto z\times w$. 
This induces a map of complexes on alternating parts 
$$
Z^r(\Spec {\bold k}, n)^{\rm alt}\otimes Z^s(\Spec {\bold k}, m)^{\alt}\to Z^{r+s}(\Spec {\bold k}, n+m)^{\rm alt}
$$
given by $z\otimes w\mapsto z\cdot w=\Alt (z\times w)$.
We set $N_r^i= Z^r(\Spec {\bold k}, 2r-i)^{\rm alt}$ for $r\ge 0$ and $i\in \ZZ$
\index{$N_r^i$, $N^{\bullet}_r$}.


One thus has an associative product map
$$N_r^i\otimes N_s^j\to N_{r+s}^{i+j}\,,\qquad z\otimes w\mapsto z\cdot w\,,$$
One verifies that the product is 
graded-commutative:
$w\cdot z=(-1)^{ij}z\cdot w$ for $z\in N_r^i$ and $w\in N_s^j$. 

Let $\dis N=\underset{r\ge 0,\, i\in \ZZ}\oplus  N_r^i$, and $N_r=\underset{i\in \ZZ}\oplus N_r^i$. 
\index{$N$, $N^i$} By  Lemma 4.3 \cite{BK}, $N$ with the above product and the differential $\part$
is a differential graded algebra (DGA) over $\QQ$.

\noindent By definition, one has $N_0=\QQ$ and $1\in N_0$ is the unit for the product. 
Thus the projection $\epsilon : N\to N_0=\QQ$\index{augmentation $\epsilon$} is an augmentation, namely it is a map of DGA's and the composition 
with the unit map $\QQ\to N$ is the identity.

\subsection{The complex $ {AC}^\bullet$}
Assume now that  $\bold k$ is a subfield of $\CC$. 

\begin{definition}[The complex $AC^\bullet$] 
\label{def:admissible triple complex}
For an integer $k$, let 
$$
 AC^k=\underset{n-i=k}\bigoplus
AC_i(P^n, \bold D^n; \QQ)^{\rm alt}
$$
Here we set 
$$
AC_i(P^0, \bold D^0; \QQ)^{\rm alt}=
\left\{
\begin{array}{cc}
\QQ & i=0\\
0& i\neq 0
\end{array}
\right.
$$
We define the  differential $d$ of $AC^\bullet$ 
by the equality
\begin{equation}
\label{eq:definition of differential D}
d(\alpha)=\partial \alpha+(-1)^n\delta\alpha \text{ for } \al\in AC_i(P^n,\bold D^n;\QQ)^{\rm alt}.
\end{equation}
Here we use the fact that the maps $\part$ and $\delta$ commute with the projector $\rm Alt$. 
cf. Lemma 4.3 \cite{BK}.

\end{definition}
We have a natural inclusion $Z^r(\Spec {\bold k}, 2r-i)\to AC_{2r-2i}(P^{2r-i},\bold D^{2r-i}; \QQ)$
which induces an inclusion $\iota:\, N\to AC^\bullet$.
\begin{proposition}
\label{properties of tc}
\begin{enumerate}
\item Let $u:\,\QQ\to AC^\bullet$ be the inclusion defined by identifying
$\QQ$ with $AC_0(P^0,\bold D^0;\QQ)^{\rm alt}$. Then the map $u$ is a quasi-isomorphism.

\item The inclusion $\iota$ and the product defined by
$$
\begin{array}{c}
N_r^i\otimes AC_j(P^n,\bold D^n;\QQ)^{\rm alt}\to AC_{2r-2i+j}(P^{n+2r-i},\bold D^{n+2r-i};\QQ)^{\rm alt}\\
z\times \ga\mapsto z\cdot \ga:={\rm Alt}(\iota(z)\times \ga)
\end{array}
$$
makes $AC^\bullet$ a differential graded $N$-module i.e. the product
sends $N^i_r\otimes AC^j$ to $AC^{i+j}$, and one has the derivation formula; for $z\in N_r^i$ and $\ga\in AC^\bullet$,
one has 
$$ d(z\cdot \ga)=\part z\cdot \ga+(-1)^iz\cdot d\ga$$

\item The map $I$ defined by 
$I=\sum_{n\geq 0}I_n:\,{AC}^{\bullet}\to \CC$ is a map of complexes. Here the map $I$
on $AC_0(P^0, D^0; \QQ)^{\rm alt}=\QQ$ is defined to be the natural inclusion of $\QQ$ to $\CC$. The map
$I$ is a homomorphism of $N$-modules; for $z\in N$ and $\ga\in AC^\bullet$,
we have an equality
$$I(z\cdot \ga)=\e(z)I(\ga)$$

\item Let $I_\CC: AC^\bullet \otimes \CC \to \CC$ be the map defined by the composition
\newline $\,AC^\bullet\ot \CC \overset {I\ot {\rm id}}\to \CC\ot \CC\overset{m}\to \CC $. Here the map $m$ is the multiplication.
Then  this map $I_\CC$ is a quasi-isomorphism.
\end{enumerate}

\end{proposition}
\begin{proof} 
(1)  
 Since $(P^n, \bold D^n)= (\BP^1, \{1\})^n$ (the $n$-fold self product), the K\"unneth formula tells us
 $$
H_*(C_{\bullet}(P^n,\bold D^n;\Q))= H_*(P^n, \bold D^n;\Q)= H_*(\BP^1, \{1\};\Q)^{\otimes n}\,. 
$$
 It follows that $H_i(P^n, \bold D^n;\Q)=0$ for $i\neq 2n$, and $H_{2n}(P^n, \bold D^n;\Q)=\QQ [P^n]$, where 
 $ [P^n]$ denotes  the image of the orientation class $[P^n]\in H_{2n}(P^n;\Q)$.
 Since $[P^n]$ is fixed by all $g\in G_n$, 
the alternating part $H_*(P^n, \bold D^n;\Q)^{\rm alt}$ is zero for $n>0$.
By Proposition \ref{prop: moving lemma}
, the complex $AC_\bullet(P^n, \bold D^n;\Q)$ is quasi-isomorphic to
$C_\bullet(P^n, \bold D^n;\Q)$.
It follows that $AC_\bullet(P^n, \bold D^n;\Q)$ is an acyclic complex for $n>0$. 
Therefore the map $u:\Q \to AC^{\bullet}(\Q)$ is a quasi-isomorphism.

(2) The proof is similar to that of Lemma 4.3 \cite{BK}. The points are that the cubical differential $\part$
and the topological differential $\delta$ commute with the projector Alt, and that we have the equality
$$\delta(z\times \ga)=z\times \delta(\ga)$$
since $z$ is a topological cycle of even dimension. 

(3) By Theorem \ref{th:generalized Cauchy formula}  the map $I$ is a homomorphism of
complexes. The equality holds because the integral of the logarithmic differential form
$\omega_n$ on an algebraic cycle is zero since the restriction of the from $\omega_n$ to a strict
subvariety of $P^n$ is zero.

(4) This follows from the assertion (1) and the fact that the map $I_\CC\circ (u\otimes {\rm id}):
\,\CC=\QQ\otimes \CC\to \CC$ is the identity map. 
\end{proof}


\subsection{The bar complex}
\label{bar complex}
Let $M$ (resp. $L$) be a complex which is a differential left $N$-module  (resp. right  $N$-module).  We recall the definition
of the bar complex $B(L, N, M)$ (see \cite{EM}, \cite{Ha} for details.) 

Let $N_+= \oplus_{r>0} N_r$.  As a module, $B(L, N,M)$\index{$B(L, N,M)$} 
is equal to 
$L\otimes \bigl(\bigoplus_{s\ge 0} (\otimes ^s N_+)\bigr)\otimes M$, 
with the convention $(\otimes ^s N_+)=\QQ$ for $s=0$. 
An element $l\otimes (a_1\otimes \cdots \otimes a_s) \otimes m$  of 
 $L\otimes  (\otimes ^s N_+)\otimes M$ 
  is written as
$l[a_1|\cdots |a_s] m$ (for $s=0$, we write $l[\,\,]m$ for $l\otimes 1\otimes m $ in $L\otimes \QQ\otimes M$ ).

 The internal differential $d_I$ is defined by
\begin{align*}
 d_I&(l[a_1|\cdots |a_s] m)\\
=&dl  [a_1|\cdots |a_s]m\\
+&\sum_{i=1}^{s}
(-1)^i Jl[Ja_1|\cdots |Ja_{i-1}|da_{i}|\cdots |a_s] m
+(-1)^sJl[Ja_1|\cdots |Ja_s] dm
\end{align*}
where
$Ja=(-1)^{\deg a}a$.
The external differential $d_E$ is defined by
\begin{align*}
 d_E&( l [a_1|\cdots |a_s] m)\\
=&-(Jl)a_1[a_2|\cdots |a_s] m\\
+&\sum_{i=1}^{s-1}
(-1)^{i+1}Jl[Ja_1|\cdots |(Ja_{i})a_{i+1}|\cdots |a_s] m
\\
+&(-1)^{s-1}Jl[Ja_1|\cdots |Ja_{s-1}] a_sm.
\end{align*}
Then we have $d_Id_E+d_Ed_I=0$ and the map $d_E+d_I$ defines
a differential on $B(L, N,M)$. The degree of an element
$l[a_1|\cdots |a_s] m$
is defined by $\sum_{i=1}^s \deg a_i +\deg l+\deg m-s$. 

If $L=\QQ$ and the right $N$-module structure is given by the augmentation
$\epsilon$,
the complex $B(L,N,M)$ is denoted by $B(N,M)$
and omit the first factor ``$1\otimes$''.
If $L=M=\Q$ with the $N$-module structure given by the augmentation $\e$,
we set 
$$
B(N):=B(\Q,N, \Q).
$$
we omit the first and the last tensor factor ``$1\otimes$'' and
``$\otimes 1$'' for
an element in $B(N)$. \index{$B(N), B(N)_r$}

The complex $B(N)$ is graded by non-negative integers as a complex: $B(N)=\oplus_{r\ge 0}B(N)_r$  where 
$B(N)_0 = \QQ$
and, for $r > 0$,
$$
B(N)_r=\oplus_{r_1+\cdots +r_s=r,\,r_i> 0}
N_{r_1}\otimes \cdots \otimes N_{r_s}.
$$
Let $\Delta:B(N)\to B(N)\otimes B(N)$
\index{$\Delta:B(N)\to B(N)\otimes B(N)$}
be the map given by
$$
\Delta([a_1|\cdots |a_s])
=
\sum_{i=0}^s\big([a_1|\cdots |a_i]\big)\otimes
\big([a_{i+1}|\cdots |a_s]\big).
$$
and $e : B(N) \to \QQ$ be the projection to $B(N)_0$ . These are maps of complexes, and they satisfy
coassociativity $(\Delta \otimes 1)\Delta = (1\otimes \Delta)\Delta$ and counitarity 
$(1\otimes e)\Delta = (e\otimes 1)\Delta = id$, in other words
$\Delta$ is a coproduct on $B(N)$ with counit $e$.
In addition, the shuffle product (see e.g. \cite{EM}) makes $B(N)$ a
DG algebra with unit $\QQ = B(N)_0\subset B(N)$. The shuffle product is graded-commutative. Further, the
maps $\Delta$ and $e$ are compatible with product and unit. We summarize:
\begin{enumerate}
\item
$B(N) =
\oplus_{r\geq 0} B(N)_r$
is a DG bi-algebra over $\QQ$. (It follows that $B(N)$ is a DG Hopf algebra, since it
is a fact that antipode exists for a graded bi-algebra.)
\item
$B(N) =\oplus_{r\geq 0} B(N)_r$
is a direct sum decomposition to subcomplexes, and product, unit, co-product and counit are compatible with this decomposition.
\item
The product is graded-commutative with respect to the cohomological degree which we denote by $i$.
\end{enumerate}
With due caution one may say that $B(N)$ is an ``Adams graded'' DG Hopf algebra over $\QQ$, with graded-
commutative product; the first ``Adams grading'' refers to $r$, and the second grading refers to $i$, while
graded-commutativity of product is with respect to the grading $i$ (the product is neither graded-commutative or commutative with respect to $r$). We recall that graded Hopf algebra in the
literature means a graded Hopf algebra with graded-commutative product, so our $B(N)$ is a graded
Hopf algebra in this sense with respect to the grading $i$, but is not one with respect to the
``Adams grading'' $r$.
In the following
we will denote $H^0(B(N))$ by $\chi_N$.
The product, unit, coproduct, counit on $B(N)$ induce the corresponding
maps on ${\chi_N}$, hence ${\chi_N}$ is an ``Adams graded'' Hopf algebra over $\QQ$ in the following sense:
\begin{enumerate}
\item
${\chi_N}$ is a Hopf algebra over $\QQ$.
\item
With $({\chi_N})_r := H^0 (B(N)_r )$, one has 
${\chi_N}=
\oplus_{r \geq 0} ({\chi_N})_r$
a direct sum decomposition to subspaces;
the product, unit, coproduct and counit are compatible with this decomposition. 
\end{enumerate}

We also have the coproduct map 
$\Delta:H^0(B(\Q,N, M))\to \chi_N \otimes H^0(B(\Q,N, M))$
obtained from the homomorphism of complexes
$\Delta:B(\Q,N, M)\to B(N)\otimes B(\Q,N, M)$
given by
$$
\Delta([a_1|\cdots |a_s] m)
=
\sum_{i=0}^s\big([a_1|\cdots |a_i]\big)\otimes
\big([a_{i+1}|\cdots |a_s] m\big).
$$

We define the category of mixed Tate motives after Bloch-Kriz \cite{BK}. 
\begin{definition}[Adams graded $\chi_N$-co-modules, 
mixed Tate motives\index{mixed Tate motives}, \cite{BK}]

\begin{enumerate}
\item
Let $V=\oplus_iV_i$ be  an Adams graded vector space (to be precise,  a finite dimensional $\Q$-vector space equipped with a grading 
by integers $i$).  
A linear map
$$
\Delta_V:V \to V\otimes {\chi_N}
$$
is called a graded coaction of ${\chi_N}$
if the following conditions hold.
\begin{enumerate}
\item
$\Delta_V(V_i)\subset\oplus_{p+q=i}V_p\otimes {\chi_N}_q$.
\item
(Coassociativity) 
The following diagram commutes.
$$
\begin{matrix}
V &\xrightarrow{\Delta_V} &V\otimes {\chi_N} \\
\Delta_V\downarrow \phantom{\Delta_V}& & \phantom{id_V\otimes \Delta_{{\chi_N}}}\downarrow 
{\rm id}_V\otimes \Delta_{{\chi_N}}
  \\
V\otimes {\chi_N} & \xrightarrow{\Delta_V\otimes {\rm id}_{{\chi_N}}} & V\otimes {\chi_N} \otimes {\chi_N}
\end{matrix}
$$
\item
(Counitarity)
The composite
$$
V \to V\otimes {\chi_N} \xrightarrow{{\rm id}_V\otimes e} V,
$$
is the identity map,
where $e$ is the counit of ${\chi_N}$.
\end{enumerate}
An Adams graded vector space $V$ with a coaction $\Delta_V$ of ${\chi_N}$
preserving the Adams grading is called an Adams graded right co-module
over ${\chi_N}$. For an Adams graded right co-modules $V$, $W$ over ${\chi_N}$, a linear map
$V\to W$ is called a homomorphism of Adams graded right co-modules over 
${\chi_N}$ if it preserves the Adams gradings 
and the coactions of ${\chi_N}$.
The category of Adams graded right co-modules over ${\chi_N}$ is denoted by 
${\rm co-rep}^f(\chi_N)$.
\index{$(\operatorname{Com}^{gr}_{H^(B(N))})$}
\item
The category of mixed Tate motives 
$(MTM)=(MTM_{\bold k})$ 
\index{$(MTM), (MTM_{\bold k})$} 
over  $\Spec(\bold k)$
is defined as the category 
${\rm co-rep}^f(\chi_N)$
of finite dimensional Adams graded right co-modules over ${\chi_N}$.
\end{enumerate}
\end{definition}
\subsection{Ind-mixed Hodge structure $\cal H_{Hg}$}
\label{Hhodge}

We recall the definition of the Tate Hodge structure.  For an integer $r$, let
$\Q(r)=\Q(2\pi i)^r$\index{$\Q(r)$} with the weight filtration $W$ defined by
$\Q(r)=W_{-2r}\supset W_{-2r-1}=0$, and let $\C(r)=\C$ with
the Hodge filtration $F$ defined by 
$\C(r)=F^{-r}\supset F^{-r+1}=0$.\index{$\C(r)$}
We define the mixed Tate Hodge structure $\Q_{Hg}(r)$\index{$\Q_{Hg}(r)$} 
of weight $-2r$ by
the $\QQ$-mixed Hodge structure $(\QQ(r),\CC(r),F,W)$ where the 
comparison map $c:\,\Q(r)\to \C(r)$ is defined by the inclusion map.
This is a Hodge structure of type $(-r,-r)$.
For  a mixed Hodge structure $H_{Hg}$,
$H_{Hg}\otimes \QQ_{Hg}(r)$ is denoted by
 $H_{Hg}(r)$.
\index{$H(r)$, $H_{Hg}(r)$}  A (finite dimensional) mixed Hodge structure is called a mixed Tate
Hodge structure if the weight graded quotients
are isomorphic to direct sums of Tate Hodge structures.
An inductive limit of mixed Tate Hodge structure is called
a ind-mixed Tate Hodge structure. 

\label{MTHS bar complex}
In this section, we define a mixed Hodge structure $\cal H_{Hg}$
with a left coaction of $\cal H$.

We define the bar complexes $\cal B_B$ and $\cal B_{dR}$ by 
$$
\cal B_B=B(N,AC^\bullet) \,\text{ and }\, \cal B_{dR}=B(N,\bold k).
$$

We introduce the weight filtration $W_{\bullet}$ on 
$\cal B_B$ and
$\cal B_{dR}$ as follows.
\index{$W_n\cal B_B, W_n\cal B_{dR}$} 
$$
 W_n\cal B_B= 
\bigoplus_{2(r_1+\cdots +r_s)
\leq n,s\geq 0,\,r_i>0}N_{r_1}\otimes \cdots \otimes N_{r_s}\otimes AC^{\bullet}.
$$
$$
 W_n\cal B_{dR}= 
\bigoplus_{2(r_1+\cdots +r_s)\leq n,s\geq 0,\,r_i>0}N_{r_1}\otimes
\cdots \otimes N_{r_s} \otimes\bold k.
$$
The
Hodge filtration $F^{\bullet}$ on
$\cal B_{dR}$ is defined by 
$$
F^p\cal B_{dR}=
\bigoplus_{r_1+\cdots +r_s\geq p,s\geq 0,\,r_i>0}
N_{r_1}\otimes
\cdots \otimes N_{r_s} \otimes\bold k.
$$
We have a canonical isomorphism 
$$\text{gr}^W_{2r}{\cal B}_B=
\bigoplus_{r_1+\cdots +r_s=r,\,s\geq 0,\,r_i>0}N_{r_1}\otimes \cdots \otimes N_{r_s}\otimes AC^{\bullet}
$$
resp.
$$\text{gr}^W_{2r}{\cal B}_{dR}=
\bigoplus_{r_1+\cdots +r_s=r,\,s\geq 0,\,r_i>0}N_{r_1}\otimes \cdots \otimes N_{r_s}\otimes
\bold k.$$

We define the comparison map $c:\,\,{\cal B}_B\to
{\cal B}_{dR}\otimes_{\bold k}\CC $ to be $\text{id}\otimes (2\pi i)^{-r} I$
on $\text{gr}^W_{2r}{\cal B}_B$. By Proposition \ref{properties of tc} (3), the comparison map $c$ induces a homomorphism of complexes
from ${\cal B}_B$ to ${\cal B}_{dR}\otimes_{\bold k}\CC$. 

\begin{definition}
\label{def of universal mixed Hodge structure}
 We define the Betti part $\cal H_B$
and  the de Rham part $\cal H_{dR}$
of $\cal H_{Hg}$
by $\cal H_B=H^0(\cal B_B)$ and 
$\cal H_{dR}=H^0(\cal B_{dR})$.
\index{$\cal H_{B}, \cal H_{dR}$}

\end{definition}
By 
Proposition \ref{properties of tc} (3), 
the map $c$ induces  a homomorphism
\index{$F^p\cal B_{dR}$}
\begin{equation}
\label{weight compatible}
c:\, W_n\cal B_B\ot\CC \to W_n\cal B_{dR}\otimes_{\bold k}\CC
\end{equation}
The weight and Hodge filtrations on 
$\cal B_B$ and 
$\cal B_{dR}$
induce those of 
$\cal H_B$ and
$\cal H_{dR}$, respectively.

\begin{proposition-definition}
\label{hodge tate}
\begin{enumerate}
\item The map $c$ induces a filtered quasi-isomorphism
$\cal B_B\otimes \CC\to \cal B_{dR}$ with respect to $W$.
\item
\label{gr of r is degree r part}
We have a canonical isomorphism of vector spaces
$$
Gr^W_{2r}\cal H_B\to (\chi_N)_r
$$
\item
Via the isomorphism $c:\cal H_B\otimes \CC\to \cal H_{dR}$,
the pair of filtered vector spaces
$\cal H_{Hg}=(\cal H_B,\cal H_{dR},W,F)$
\index{$\cal H_{Hg}$} becomes a
ind-mixed
Tate Hodge structure.
\end{enumerate}

\end{proposition-definition}
\begin{proof} 
(1)  One sees that the quotient $Gr_{2r}^W\cal B_B$  is the tensor product
\[ 
B(N)_r\ot AC^{\bullet}
\]
as a complex. By Proposition \ref{properties of tc} (4) the map $c$ induces a quasi-isomorphism
\newline $Gr_{2r}^W\cal B_B\ot \CC\to Gr_{2r}^W\cal B_{dR}$.
 
(2)
 We consider the spectral sequences for the filtration $W$:
 \begin{align*}
E_1^{p,q}=H^{p+q}(Gr^W_{-p}\cal B_B) &\Rightarrow E^{p+q}=
 H^{p+q}(\cal B_B) \\
\ '  E_1^{p,q}=H^{p+q}(Gr^W_{-p}\cal B_{dR}\otimes_{\bold k}\CC) &\Rightarrow \ 'E^{p+q}=
 H^{p+q}(\cal B_{dR}\otimes_{\bold k}\CC). 
 \end{align*}
Since the morphism of complexes $c:\,\cal B_B\otimes \CC\to \cal B_{dR}\otimes_{\bold k}\CC$
 is a filtered quasi-isomorphism, the morphism of spectral sequences
 
  $$E_*^{*,*}\otimes \CC \to \ 'E_*^{*,*}$$
 is an isomorphism. Note that we have 
$$E_1^{p,q}\ot \C=H^{p+q}(Gr^W_{-p}({\cal B}_B\ot \C))$$ since tensoring with $\C$ is an exact functor.
Since the complex $\cal B_{dR}$ is  the direct sum 
$\dis\underset{r}\oplus B(N)_r\ot \bold k$, 
the spectral sequence $\ 'E_*^{*,*}$ degenerates at $E_1$-term, and
as a consequence  $E_*^{*,*}$ also degenerates at $E_1$-term.
 Therefore the vector space
$Gr_{2r}^WH^0(\cal B_B)$ is canonically isomorphic to $H^0(Gr_{2r}^W\cal B_B)$.
By Proposition \ref{properties of tc} (1) we have
$$
H^0(Gr^W_{2r}\cal B_B)
={\chi_N}_r\ot H^0(AC^{\bullet})={\chi_N}_r\otimes \Q.
$$

(3)  We need to show the following.
\begin{enumerate}
\item The filtrations $F$ and $\ov{F}$ on  $Gr_{2r}^W\cal H_{dR}\otimes_{\bold k}\CC$ are
$2r$-opposite (\cite{D1}, (1.2.3)).

\item $(F^p\cap \ov{F^{2r-p}})Gr_{2r}^W\cal H_{dR}\otimes_{\bold k}\CC=0$
for $p\neq r$.
\end{enumerate}
We denote $Gr_{2r}^W\cal H_{dR}=A$. 
We have 
\[F^p(A)=\left\{
\begin{array}{ll}
A& p\leq r\\
0&p>r
\end{array}
\right.\]
By taking complex conjugate, a similar fact holds for $\ov{F^p}(A)$. The assertions (1) and (2) follow from this.

\end{proof}

\subsection{A Hodge realization functor for mixed Tate motives}
\label{subs:Hodge real}

Let $V=\oplus_r V_r$ be an Adams graded $\Q$-vector space. The mixed Tate 
Hodge structure associated to $V$ which is denoted by
$V_{Hg}$, is defined as follows. $V_{Hg}$ is the triple $(V_b, V_{dR},c)$.
$V_b=V$ as a $\Q$-vector space, and the weight filtration is defined
so that $W_nV_b=\oplus_{2r\leq n}V_r$. $V_{dR}=V\otimes \C$ as a vector space,
and  $W_nV_{dR}=\oplus_{2r\leq n}V_r\ot \C$. The Hodge filtration
is defined by $F^pV_{dR}=\oplus_{r\geq p} V_r\ot \C$. We have a canonical isomorphism
$\text{gr}^W_{2r} V_b=V_r$. The comparison map $c$ is defined on 
$\text{gr}^W_{2r} V_b=V_r$ to be multiplication by $(2\pi i)^{-r}$.  
Let $V$ be an Adams graded ${\chi_N}$-co-module.
The coproduct $\Delta_V:\,V\to V\ot {\chi_N}$ induces a map of
mixed Tate Hodge structures
$$\Delta_{V_{Hg}}:\,\,V_{Hg}\to V_{Hg}\ot {\chi_N}_{Hg}.$$
Similarly the coproduct 
$$\Delta_b:\,{\cal H}_B\to \chi_N
\ot {\cal H}_B$$ 
and
$$\Delta_{dR}:\,{\cal H}_{dR}\to \chi_N
\ot {\cal H}_{dR}$$ 
induce a map of mixed Tate Hodge structures
$$\Delta_{Hg}:\,\,{\cal H}_{Hg}\to {\chi_N}_{Hg}\ot {\cal H}_{Hg}.$$

\begin{definition}[Realization functor]
\label{defphi}
We define the functor $\Phi$ 
from the category ${\rm co-rep}^f(\chi_N)$ 
of Adams graded $ {\chi_N}$-co-modules to that of
mixed Hodge structures $(MH)$ by
\index{$\Phi(V)$}
$$
\Phi(V)=\ker(\Delta_{V_{Hg}}\otimes \id- \id\otimes \Delta_{Hg}:\,\,
V_{Hg}\otimes {\cal H}_{Hg}\to V_{Hg}\ot {\chi_N}_{Hg}\ot {\cal H}_{Hg}).
$$
It is called the Hodge realization functor.

\end{definition}

\subsection{Hodge realization of the category of flat connections}
\label{flat conn}
We will show that the Hodge realization functor $\Phi$ given in Definition \ref{defphi} is equivalent to a simpler functor 
under the assumption that the Beilinson-Soul\'e vanishing conjecture holds. 
Let $A$ be an Adams graded commutative dga (cdga) and let $M$ be an Adams graded $\Q$-vector space.
An $A$-connection for $M$ is a map of graded vector spaces
\[\Gamma:\quad M\to M\otimes A^1\] 
of Adams degree 0. This name is taken from \cite{Lev}. This is called a twisting matrix in \cite{KM}.  $\Gamma$ is said to be flat if $d\Ga+\Ga^2=0.$ Here
$d\Ga$ is the map
\[M\overset{\Ga}\to M\otimes A^1\overset{1\otimes d}{\to} M\otimes A^2\]
and $\Ga^2$ is
\[M\overset{\Ga}\to M\otimes A^1\overset{\Ga\otimes 1}{\to} M\otimes A^1\otimes A^1\overset{1\otimes m}\to M\otimes A^2\]
where $m$ is the multiplication of $A$.  We denote by $\text{Conn}_A^0$ the category of flat $A$-connections, and
let $\text{Conn}_A^{0f}$ be its full subcategory of  flat connections on finite dimensional Adams graded $\Q$-vector spaces.

 Now we assume that $A$ is cohomologically connected. We denote by $A\{1\}$ the 1-minimal model of $A$. 
For the defintion of 1-minimal model, see for example Definition 1.30 of \cite{Lev}. 
Set $QA:=A\{1\}^1$ and let $\part: QA\to \wedge ^2QA$ denote the differential $d:\,A\{1\}^1\to \wedge^2A\{1\}^1
=A\{1\}^2$. Then $(QA,\part)$ is a co-Lie algebra over $\Q$. Since $A$ is Adams graded, $QA$ is also Adams graded.
Let $\text{co-rep}(QA)$ be the category of Adams graded co-modules over $QA$, and let 
$\text{co-rep}^f(QA)$ be the category of finite dimensional Adams graded co-modules over $QA$. 
We denote by $\chi_A$ the Adams graded Hopf algebra $H^0(B(A))$.  Let 
$\text{co-rep}^f(\chi_A)$ denote the category of finite dimensional Adams graded co-modules over 
$\chi_A$. By Theorem 1.2 and Theorem 2.10 of \cite{KM} we have  the equivalence of filtered neutral Tannakian categories
\[\text{co-rep}^f(\chi_A)\sim \text{co-rep}^f(QA)\sim \text{Conn}^{0f}_{A\{1\}}.\]

We now consider the case where $A$ is the cycle algebra $N$ defined in \cite{BK}. Beilinson-Soul\'e
conjecture implies that $N$ is cohomologically connected. For each field $k$, Hanamura constructed a triangulated category ${\cal D}(k)$
of motives over $k$ in \cite{Ha1} and  \cite{Ha2}. Let ${\cal D}T(k)$ be the sub triangulated category of ${\cal D}(k)$
generated by Tate objects. 
If we assume Beilinson-Soul\'e 
vanishing conjecture and $K(\pi,1)$ conjecture which asserts that $A\{1\}$ is quasi-isomorphic to $A$, 
the category $\text{Conn}^{0f}_{N\{1\}}$ should  be equivalent 
to the heart of the category ${\cal D}T(k)$ for a certain $t$-structure. We will show that there is a natural Hodge realization of $\text{Conn}^{0f}_{N\{1\}}$,
which is compatible with the functor $\Phi$ defined  in Definition \ref{defphi} under the equivalence of the categories
$MTM_k=\text{co-rep}^f(\chi_N)\sim \text{Conn}^{0f}_{N\{1\}}$. Let $M=\oplus_r M_r$ be a finite dimensional Adams graded
$\Q$-vector space with a flat connection $\Ga:\,\,M\to M\otimes QN$. The Hodge complex associated to
$M$ is the triple $(M_b,\,M_{dR},\,c)$ defined as follows. $M_b=M\otimes AC^\bullet$ as a bigraded $\Q$-vector space.
Here  $AC^{\bullet}$ is the complex of semi-algebraic chains defined in Definition \ref{def:admissible triple complex}. The differential of $M_b$
is the sum $d_1+d_2$ where $d_1=1\otimes d_{AC^\bullet}$ and $d_2$ is the composite $M\otimes AC^\bullet
\overset{\Ga\otimes 1}\to M\otimes N\{1\}^1\otimes AC^\bullet \overset{1\otimes m} \to M\otimes AC^\bullet$.
Here $m$ is the multiplication $ N\{1\}^1\otimes AC^\bullet\to AC^\bullet$. The weight filtration
of $M_b$ is defined by 
$$W_nM_b=\bigoplus_{2r\leq n}M_r\otimes AC^\bullet.$$
The complex $M_{dR}:=M\otimes \bold k$ concentrated in cohomological degree 0. 
The weight filtration of $M_{dR}$ is defined by
$$W_nM_{dR}=\bigoplus_{2r\leq n}M_r\otimes \bold k$$
and the Hodge filtration of $M_{dR}$
is defined by $$F^pM_{dR}=\bigoplus_{p\leq r} M_r\otimes \bold k.$$ 
We have a canonical isomorphism 
$$\text{gr}^W_{2r}(M_b)\simeq M_r\otimes AC^\bullet$$
and
$$\text{gr}^W_{2r}(M_{dR})\simeq M_r\otimes \bold k.$$
The comparison map $c:\,\,M_b\to M_{dR}\otimes_{\bold k}\CC$ is defined to be $\text{id}\otimes(2\pi i)^{-r} I$ on
$M_r\ot AC^\bullet$ where $I:\,AC^\bullet \to \C$
is the map defined in Proposition \ref{properties of tc} (3).

\begin{proposition}
\label{hodgeconn}
The triple $(H^0(M_b),\,H^0(M_{dR})(=M_{dR}),\,c)$ is a mixed Tate Hodge structure. 
\end{proposition}
\begin{proof}
Consider the spectral sequence
$$E_1^{p,q}=H^{p+q}(\text{gr}^W_{-p}M_b)\Rightarrow E_\infty^{p,q}
=\text{gr}^W_{-p}H^{p+q}(M_b)$$
resp.
$$'E_1^{p,q}=H^{p+q}(\text{gr}^W_{-p}M_{dR}\otimes_{\bold k}\CC)\Rightarrow 'E_\infty^{p,q}
=\text{gr}^W_{-p}H^{p+q}(M_{dR}\otimes_{\bold k}\CC).$$ The map $c$ induces
a map of weight filtered complexes from $M_b$ to $M_{dR}$,
so that a map of spectral sequences from $E_r^{p,q}$ to
${}'E_r^{p,q}$. 
Since tensoring with $\C$ is an exact functor, we have 
$$E_1^{p,q}\otimes \C= H^{p+q}(\text{gr}^W_{-p}(M_b\otimes \C)).$$
We have $\text{gr}^W_{-p}(M_b\otimes \C)=M_{-p/2}\otimes AC^\bullet
\otimes\C$, and it follows that the map $c$ induces an isomorphism
from $E_1^{p,q}\otimes \C$ to ${}'E_1^{p,q}$. Hence we see that the 
map of spectral sequences  $E_r^{p,q}\otimes \C\to {}'E_r^{p,q}$ induced
by $I$ is an isomorphism. It follows that we have an
isomorphism of weight filtered vector spaces
$$H^0(M_b)\otimes \C\to H^0(M_{dR}\otimes_{\bold k}\CC).$$
we also need to check that
\begin{enumerate}
\item The filtrations $F$ and $\ov{F}$ on  $Gr_{2r}^WH^0(M_{dR})$ are
$2r$-opposite (\cite{D1}, (1.2.3)).

\item $(F^p\cap \ov{F^{2r-p}})Gr_{2r}^WH^0(M_{dR})=0$
for $p\neq r$.
\end{enumerate} This follows  easily from the definition.
\end{proof}

\begin{definition}
\label{defpsi}
The Hodge realization $\Psi(M)$ of $M$ is defined to be the triple
\newline $(H^0(M_b),\,H^0(M_{dR})(=M_{dR}),\,c)$. 
\end{definition}

\begin{remark} There should be an explanation of the reason for this definition of the Hodge realization of
$\text{Conn}^{0f}(N_{\{1\}}).$ It comes from an idea of construction of the Hodge realization
$$\rho:\,\,{\cal D}T(\C)\to  \text{MHC}$$ where MHC denotes the category of
mixed Hodge complexes. We note that what follows is an explanation of an idea,
and is not a rigorous argument. An object of ${\cal D}T(\C)$ is of the form
$$K=(K^m;\,f^{m,n})$$
where each $K^m$ is a direct sum of the elements of the form
$$(pt,\,r),\,\,r\in \ZZ,$$
and each $f^{m,n},\,m<n$ is an element of the cycle complex
$Z^*({\rm Spec}\,\C,*)$. Moreover the system $\{f^{m,n}\}_{m,n}$ satisfies
a certain cocycle condition. An object like $K$ is called a diagram.
For each cube $\square^t$ consider the Hodge complex $\Ga(\square^t)$.
Since Hodge complex is a contravariant functor,
we have a double complex
$$\cdots \to \Ga(\square^t)\overset{\part}\to \Ga(\square^{t-1})
\to \cdots \to \Ga(\square^0)\to 0.$$ Let $\Ga(pt)$ be the total complex associated 
to this double complex. For an object like $K$, there corresponds a Hodge complex $\Ga(K^m)$ for each $m$.
Moreover, for each $f^{m,n}$ there exists a map of Hodge complexes
$$F^{m,n}:\,\,\Ga(K^m)\to  \Ga(K^n)$$ of degree $-(n-m-1),$
and the system $\{F^{m,n}\}_{m,n}$ satisfies a cocycle condition. It follows that there exists a
diagram
$$(\Ga(K^m);\,F^{m,n})$$
of Hodge complexes. We denote this by $\Ga(K)$. The functor $\rho$ is obtained by associating 
$\Ga(K)$ to each $K$.  A more explicit description of $\Ga(K)$ is given as follows. 
The betti part of $\Ga(pt)$ is equal to
$AC^\bullet$. For each $\square^t$ let $A^*_{dR}(P^t(\log {{\bf D}^t}))^{\rm alt}$ be the complex of the 
smooth forms on $P^t$ with logarithmic poles along ${{\bf D}^t}$, which are alternating for the group action. Consider the double complex
$$\cdots \to A^*_{dR}(P^t(\log {{\bf D}^t}))^{\rm alt}\overset{\part}\to A^*_{dR}(P^{t-1}(\log {{\bf D}^{t-1}}))^{\rm alt}
\to \cdots \to A^*_{dR}(P^0)\to 0.$$ The de Rham part of $\Ga(pt)$ which we denote by $A^*_{dR}(pt)$, is the associated total complex to the above double complex.
The map $F^{m,n}$ for $m<n$ is described as follows. The element $f^{m,n}$ belongs to the cycle algebra $N$. The map 
$F^{m,n}$ on $AC^\bullet$  is induced by the $N$-module structure on $AC^\bullet$. As for the de Rham  
part, let $C(f^{m,n})$ be a differential form which represents the cohomology class of the cycle $f^{m,n}$.
Then the map 
$$F^{m,n}:\,\,A^*_{dR}(K^m)\to A^*_{dR}(K^n)$$  is given  by the cup product with $C(f^{m,n})$. 
For a positive integer $p$, consider the truncated complex
$$0\to \Ga(\square^p)\overset{\part}\to \Ga(\square^{p-1})\to \cdots \to \Ga(\square^0)\to 0.$$ The total
complex associated to this complex which we denote by
$\Ga(\square^p,\part \square^p)$, is the Hodge complex of the relative variety $(\square^p,\part \square^p)$. We denote the dual of 
$\Ga(\square^p,\part \square^p)$ by  $\Ga_c(\square^p-\part \square^p)$. The de Rham part of $\Ga_c(\square^p-\part \square^p)$ which we denote by
${\Ga_c}_{dR}(\square^p-\part \square^p)$,
is isomorphic to the complex generated by  the logarithmic form
$\omega_p$ concentrated in cohomological degree 0. Suppose that for a pair $(m,n)$ with $m<n$,
the element $f^{m,n}$ belongs to $Z^r({\rm Spec}\,\C, q)$. Consider the projector
$${\rm Alt}:\,\square^p\times \square^q\to \square^{p+q}$$
and the projection
$$pr_1:\,\,\square^p\times \square^q\to \square^p.$$
We denote by  $|f^{m,n}|$ the support of the cycle $f^{m,n}$. Let $\iota:\,\,\square^p\times |f^{m,n}|\to \square^{p+q}$
be the inclusion. 
Then the dual of the de Rham part of the map $F^{m,n}$ is equal to ${pr_1}_*({\rm Alt}\circ \iota)^*(\omega_{p+q})$
where $${pr_1}_*:\,\,A^*(\square^p\times |f^{m,n}|)\to A^{*-2\dim |f^{m,n}|}(\square^p)$$
is the integration along $|f^{m,n}|.$ We see that de Rham part of the dual of the map $F^{m,n}$ is zero for reason of type.  Hence the de Rham part of the map $F^{m,n}$ is also zero.
The betti part of the complex $\Ga(\square^p,\part \square^p)$,
which we denote by $\Ga_B(\square^p,\part \square^p)$,  is the total complex 
associated to the double complex
$$0\to AC_*(P^p,{{\bf D}})^{\rm alt}\overset{\part}\to
AC_*(P^{p-1},{{\bf D}})^{\rm alt}\to\cdots \to AC_*(P^0)\to 0.$$
The duality pairing
$$ {\Ga_c}_{dR}(\square^p-\part \square^p)\times \Ga_B(\square^p,\part \square^p)
\to \C$$
is given by the map $I$. See \cite{Ki} for more detail.
\end{remark}
 We will show that the  the Hodge realizations $\Psi$ and $\Phi$ 
defined in definition \ref{defphi}  are equivalent.
 By iterating the map $\Ga$ we obtain a map $\Ga_n:\,M\to M\otimes (N\{1\}^1)^{\otimes n}$
for each $n\geq 0$. Taking the sum $\sum_{n\geq 0}\Ga_n$, we obtain a map
$\Delta_M\to M\otimes T(N\{1\}^1)$. Here $T(N\{1\}^1)=\oplus_{s\geq 0}(N\{1\}^1)^{\ot s}$. 
Since the connection $\Ga$ is flat, $\Delta_M$ is actually a map
$M\to M\otimes H^0(B(N\{1\}))$. The maps $\Delta_M$ induces an equivalence of the categories
\[ \text{Conn}^{0f}_{N\{1\}}\to \text{co-rep}^f(\chi_A).\]
Let $\Delta_{M_b}$ be the map
\[M\otimes AC^\bullet\overset{\Delta_M\otimes 1}\to M\otimes T(N\{1\}^1)\otimes AC^\bullet.\]
It follows from the definition of the bar complex $B(N,AC^\bullet)$ that $\Delta_{M_b}$ defines a map
of complexes
\[\Delta_{M_b}:\,\,M_b\to M\otimes B(N,AC^\bullet).\]
Similarly, we have the map
\[\Delta_{M_{dR}}:\,\,M_{dR}\to M\otimes B(N,\C).\]
Let ${\cal M}_b$ resp. ${\cal M}_{dR}$ be
\[
\text{Ker}(\Delta_M\otimes 1-1\otimes \Delta_b:\,\,
M\otimes H^0(B(N,AC^\bullet))\to M\otimes {\chi_N}\otimes
H^0(B(N,AC^\bullet)))\]
resp.
\[ 
\text{Ker}(\Delta_M\otimes 1-1\otimes \Delta_{dR}:\,\,
M\otimes H^0(B(N,\C))\to M\otimes {\chi_N}\otimes
H^0(B(N,\C))).\] The Hodge structure $\Phi(M)$ is equal to the triple
$({\cal M}_b, {\cal M}_{dR},c)$ where 
$c$ is the map defined in Subsection \ref{Hhodge}.

\begin{theorem}
\label{comparison}
The pair $(\Delta_{M_b},\Delta_{M_{dR}})$ induces a map of mixed Hodge 
structures
\[(H^0(M_b),H^0(M_{dR}),c)\to ({\cal M}_b,{\cal M}_{dR},c)\]
which is an isomorphism.
\end{theorem}
\begin{proof} Since $M$ is a co-module over ${\chi_N}$, $\Delta_{M_b}(M_b)$
is contained in the kernel of the map 
$$\Delta_M\otimes 1-1\otimes \Delta_b:\,\,
M\otimes B(N,AC^\bullet)\to M\otimes B(N)\otimes B(N,AC^\bullet).$$
Hence $\Delta_{M_b}$ induces a map from $H^0(M_b)$ to
${\cal M}_b$.  We need to show that it is an isomorphism of weight filtered
vector spaces. 
By considering the spectral sequence used in the proof 
of Proposition-Definition \ref{hodge tate}, we have an isomorphism of weight 
filtered vector spaces
$$M\otimes H^0(B(N,AC^\bullet))\otimes \C
\to M\otimes H^0(B(N,\C))$$
resp. 
$$M\otimes {\chi_N}\otimes H^0(B(N,AC^\bullet))\otimes \C
\to M\otimes {\chi_N}\otimes H^0(B(N,\C)).$$
So our assertion is follows from the fact that the map
induced by $\Delta_{M_{dR}}$

$$H^0(M_{dR})=M\otimes \C
\to \text{Ker}(\Delta_M\otimes 1-1\otimes \Delta:\,\,
M\otimes H^0(B(N,\C))\to M\otimes {\chi_N}\otimes H^0(B(N,\C)))$$
is an isomorphism of weight filtered vector spaces. This is true since $M$ is an ${\chi_N}$-co-module, and ${\chi_N}$ is an Adams graded vector
space. The compatibility of the Hodge filtrations also follows from this fact.

\end{proof}
\subsection{The case of dilogarithm}
Using the above notations, we describe the Hodge realization
of the co-module over $\chi_N$
associated to dilogarithm functions after Section 9 of \cite{BK}.

We assume $\bold k\subset\CC$ and $a\in \bold k^{\times}-\{1\}$.
We define elements $\rho_1(a) \in N^1_1$ and $\rho_2(a) \in N^1_2$ by
\begin{align*}
&\rho_1(a)=\{(1-a)\in \PP^1\}^{\rm alt} \\
&\rho_2(a)=\{(x_1,1-x_1,1-\dfrac{a}{x_1})\in (\PP^1)^3\} ^{\rm alt}
\end{align*}
Then we have the following relations:
$$
\partial \rho_2(a)=-\rho_1(1-a)\cdot \rho_1(a),\quad \partial \rho_1(a)=\partial \rho_1(1-a)=0.
$$
Therefore the elements $\bold{Li}_1(a), \bold{Li}_2(a)$ defined as follows
are closed elements in $B(N)$,
and thus they define elements in $\chi_N$.
\begin{align*}
&\bold{Li}_2(a)=[\rho_2]-[\rho_1(1-a)|\rho_1(a)], \\
&\bold{Li}_1(a)=[\rho_1(a)].
\end{align*}
Let $V=V_2\oplus V_1\oplus V_0$ be the sub $\chi_N$-co-module of $\chi_N$ generated by
$e_2:=\bold{Li}_2(a),\,e_1:=\bold{Li}_1(1-a)$ and $e_0:=1$. We have
\begin{align*}
\Delta_V(e_2)&=e_2 \otimes 1-e_{1}\otimes \bold{Li}_1(a)+e_{0}\otimes \bold{Li}_2(a),\\
\Delta_V(e_{1})&=e_{1} \otimes 1+e_{0}\otimes \bold{Li}_1(1-a),\\
\Delta_V(e_{0})&=e_{0}\otimes 1.
\end{align*}
We regard $V$ as an object of $\text{Conn}^{0f}_{N\{1\}}$. Then we have
\begin{align*}
\Ga(e_2)&=-e_1\otimes \rho_1(a)+e_0\otimes \rho_2(a),\\
\Ga(e_1)&=e_0\otimes \rho_1(1-a),\\
\Ga(e_0)&=0.
\end{align*}
We will give a description of $\Psi(V)$.
Assume that $a$ is contained in $\RR$ and assume that $0<a<1$.
We consider elements $\eta_1(0), \eta_2(1),\eta_2(0)$ in $AC^0$
defined by
\begin{align*}
&\eta_1(0)=\{(1-t_0)\in \PP^1\mid 0<t_0<a\}^{\rm alt}, \\
&\eta_2(1)=\{(x_1,1-x_1,1-\frac{t_1}{x_1})|\,x_1\in \P^1_\C-\{1\},\,0<t_1<a\}^{\rm alt},\\
&\eta_2(0)=\{(t_1,1-t_0)|\,0<t_0<t_1<a\}^{\rm alt},
\end{align*}
where the orientation of $\eta_2(0)$ is defined so that $(t_1,t_0)$ are  positive coordinates. 
We have the relations
\begin{align*}
&\delta\eta_1(0)=\rho_1(a),\\
&\part\eta_1(0)=0,\\
&\delta\eta_2(1)=\rho_2(a),\\
&\delta \eta_2(0)=-\part \eta_2(1)+\rho(1-a)\cdot \eta_1(0),\\
&\part\eta_2(0)=0.
\end{align*}
Let $\xi_1(a)=\eta_1(0)$, and $\xi_2(a)=\eta_2(1)+\eta_2(0)$ be  chains of $AC^0.$
We have the equalities
\begin{align*}
&d(\xi_1(a))=-\rho_1(a),\\
&d(\xi_2(a))=-\rho_2(a)+\rho_1(1-a)\cdot \xi_1(a).
\end{align*}
The elements 
\begin{align*}
&v_2:=e_2-e_1\otimes \xi_1(a)+e_0\otimes \xi_2(a)\\
&v_1:=e_1+e_0\otimes \xi_1(1-a)\\
&v_0:=e_0
\end{align*}
form a basis of $H^0(V_b)$. 
We have the relation
\begin{align*}
c(v_2)&=(2\pi i)^{-2}(e_2+e_{1}Li_1(a)+e_{0}Li_2(a)),\\
c(v_{1})&=(2\pi i)^{-1}(e_{1}+ e_{0}\log a),\\
c(v_{0})&=e_{0}.
\end{align*}
So the period matrix of the Hodge realization  $\Psi(V)$ of $V$ with respect to
the de Rham basis $e_0\in W_0\cap F^0$, $e_1\in F^1\cap W_1$ and 
$e_2\in F^2\cap W_2$ is equal to
\[
\left(
\begin{array}{ccc}
 (2\pi i)^{-2}&0&0\\
Li_1(a) (2\pi i)^{-2}&(2\pi i)^{-1}&0\\
Li_2(a)(2\pi i)^{-2}&\log a(2\pi i)^{-1}&1
\end{array}
\right)
\] which is compatible with the one given in \cite{BD}.

\appendix

\section{Proof of the moving lemma}
\label{section moving lemma}
Before the proof, we recall the following three theorems.

\begin{theorem}{(\cite{Ze} Ch.6, Theorem 15}) 
\label{Zeeman moving}
Let $M$ be a compact $PL$-manifold and 
let $K$ be a $PL$-triangulation of $M$. Let $X$,
$X_0$ and $Y$ be subpolyhedra of $M$ such that $X_0\subset X$. Then there exists
an ambient $PL$ isotopy $h:\,\,M\times [0,1]\to M$ such that $h(x,t)=x$ for any point
$x\in X_0$ and $t\in [0,1]$ (we refer to this fact by saying that {\it h fixes $X_0$}), and such that $h_1(|X|-X_0)$ is in general
position with respect to $Y$ i.e. the inequality
\[\dim (h_1(|X|-X_0)\cap Y)\leq \dim (|X|-X_0)+\dim Y-\dim M\]
holds. Here $h_t(m)=h(m,t)$ for $m\in M$ and $t\in [0,1]$.
\end{theorem}
We need the following variant of this theorem.

\begin{theorem} 
\label{Zeeman moving2}
Let $M$ be a compact $PL$-manifold and 
let $K$ be a triangulation of $M$. Let $X_0, X$ and $Y_1,\cdots, Y_n$ be subpolyhedra of $M$. Then there
exists an ambient $PL$ isotopy $h:\,M\times [0,1]\to M$ which fixes $X_0$, and such that $h_1(|X|-X_0)$ are in general
position with respect to  $Y_i$ for  $1\leq i\leq n$.
\end{theorem}
\begin{theorem}[Lemma 1.10,\,\cite{Hu}]
\label{simplicial}
Let $f:\, |K|\to |L|$ be a $PL$ map of the realizations of simplicial complexes $K$ and $L$.
 Then there exist subdivisions $K'$ and $L'$ of $K$ and $L$ respectively, such that
$f$ is induced by a simplicial map $K'\to L'$.
\end{theorem}

\begin{proof} [Proof of Proposition \ref{prop: moving lemma}]

(1) We prove the surjectivity of the map $\iota$ on homology groups. Let $\gamma$ be a closed element in 
$C_p(K,\bold D^n;\QQ)$ for a  triangulation $K$.  
Let 
$
\gamma=\sum a_\si \sigma 
$ 
be a representative of
 $\gamma$ in $C_p(K;\QQ)$. 
By applying Theorem \ref{Zeeman moving2} to the case where $M=P^n$, $X=|\ga|$, $X_0=\bold D^n$ 
and $Y_i$ the set of cubical faces, 
we have a $PL$ isotopy $h: P^n\times [0,1] \to P^n$
such that 
\begin{enumerate}
\item
$h$ fixes $\bold D^n$, and 
\item
$h_1(|\gamma|-\bold D^n)$ meets each
cubical face properly.
\end{enumerate}
By Theorem \ref{simplicial} there exists a triangulation
${\cal K}$ of $ P^n\times [0,1]$ and a  subdivision $K'$ of $K$  such that 
\begin{enumerate}
\item
The map $h$ is induced by a simplicial map from $\cal K$ to $K'$, and
\item
$\sigma \times [0,1]$, $\sigma \times \{0\}$ and $ \sigma \times \{1\}$ are subcomplexes of ${\cal K}$ for each
simplex $\sigma$ of $ K$.
\end{enumerate}
 Let $\la:\,C_p(K)\to C_p(K')$ be the subdivision operator.
Let
$$
h_*:\,\,C_\bullet (\cal K, \bold D^n\times [0,1];\QQ) \to 
C_\bullet (K', \bold D^n;\QQ)
$$ be the map of complexes induced by $h$. 
For a simplex $\sigma \in K$, the product
$\sigma\times [0,1]$, $\sigma\times \{0\}$ and $\sigma\times \{1\}$
are regarded  elements of $C_{\bullet}({\cal K})$.
For an element $\sigma\in K_p$, we set
\begin{align}
\label{notation isotopy and image}
h_{\sigma}&=h_*(\sigma\times [0,1]) \in C_{p+1} (K', \bold D^n), \\
\nonumber
h_i(\sigma)&=h_*(\sigma\times \{i\})\in C_p (K', \bold D^n)\quad (i=0,1).
\end{align}
Then we have $h_0(\sigma)=\lambda(\sigma)$.
Let $\theta$ be the chain  $\sum_\si  a_\si h_{\sigma}$.
Then we have
\begin{align*}
\delta\theta
=&\sum_{\sigma}
a_\si \delta h_{\sigma} 
\\
=&\sum_{\sigma}
a_\si((-1)^p(h_1(\sigma)-h_0(\sigma))+h_{\delta\sigma})
\\
=&\sum_{\sigma}
a_\si h_{\delta\sigma}
+(-1)^p\biggl(
\sum_{\sigma}
a_\si h_1(\sigma) 
-\sum_{\sigma}a_\si \lambda(\sigma)  \biggr)
\\
=&
\sum_{\rho} \bigg(
\sum_{\rho\p1 \sigma}[\si:\rho]a_\si \bigg)h_{\rho}
+
(-1)^p\biggl(\sum_{\sigma}
a_\si h_1(\sigma) 
-\sum_{\sigma}a_\si \lambda(\sigma)\biggr)
\end{align*}
in $C_p(K';\QQ)$. Here in the last line, the sum in the first term is taken over $(p-1)$-simplexes
of $K$, and $\rho \prec \si$ means that $\rho$ is a face of $\si$. The index of $\si$ with respect to $\rho$
is denoted by $[\si:\rho]$. It is in the set $\{\pm 1\}$. 
Since $\ga$ is closed,  we have $\sum_{\rho\p1 \sigma}[\si:\rho]a_\si=0$ if $\rho\not\subset \bold D^n$. 
Since $\bold D^n$ is fixed by $h$, if $\rho\subset \bold D^n$, then we have $h_\rho=0$.
Thus we have the equality
$$ 
\delta\theta =
(-1)^p\biggl(\sum_{\sigma}
a_\si h_1(\sigma) 
-\lambda(\gamma)\biggr).
$$
Since $|\sum_{\sigma}
a_\si h_1(\sigma) |=h_1(|\ga|)$
by the construction of $h$, it follows that 
$\sum_{\sigma}
a_\si h_1(\sigma) \in AC_p(K', \bold D^n; \QQ)$.

(2) We prove the injectivity of the map $\iota$ on homology.  
Let $\gamma$ be an element in $AC_p(P^n,\bold D^n;\QQ)$ and suppose that 
$\gamma$ is the boundary of an
element $\xi$ in $C_{p+1}(P^n,\bold D;\QQ)$.  
Representatives of $\gamma$ and
$\xi$ in $AC_p(K;\QQ)$ and  $C_{p+1}(K;\QQ)$
are also denoted by $\gamma$ and ${\xi}$. 
By applying Theorem \ref{Zeeman moving2} to the case where $M=P^n$, $X=|\xi|$, $X_0=\bold D^n\cup |\ga|$ 
and $Y_i$ the set of cubical faces, 
we obtain a $PL$ isotopy $h: P^n\times [0,1] \to P^n$ of
$P^n$ such that
\begin{enumerate}
\item
$h$ fixes $\bold D^n\cup |{\gamma}|$, and
\item
 $h_1(|{\xi}|-(\bold D^n\cup |{\gamma}|))$ intersects the cubical faces properly.
\end{enumerate}
By a similar argument as the surjectivity, we see that the chain $h_1(\xi)$ is an element of
$AC_{p+1}(K',\bold D^n; \QQ)$ for a subdivision $K'$ of $K$, and that $\delta h_1(\xi)=\la(\ga)$
where \newline $\la:\,C_\bullet(K,\bold D^n;\QQ)\to C_\bullet(K',\bold D^n, \QQ)$ is the subdivision
operator.
\end{proof}

\section{Properties of cap products}
\label{section homotopy}
\subsection{Some facts on homotopy}
The following proposition is known as acyclic carrier
theorem (See \cite{Mu}, Theorem 13.4, p76).
\begin{proposition}
\label{existence restricted homotopy}
Let $K$ be a simplicial complex and $D_{\bullet}$ be a 
(homological) complex such that $D_i=0$ for $i<0$.
Let $p$ be a non-negative integer, and let $\varphi_{\bullet}:C_{\bullet}(K)\to D_{\bullet-p}$ 
be a homomorphism of complexes.
We suppose that there exists a family of subcomplexes $\{L^{\sigma}_{\bullet}\}_{\sigma\in K}$
of $D_{\bullet}$  
indexed by $\sigma \in K$ which satisfies the following conditions.
\begin{enumerate}
\item
$L^{\tau}\subset L^{\sigma}$ for $\tau,\sigma \in K$ such that $\tau \subset \si$.
\item
$\varphi_{\bullet}(\sigma)\in L^{\sigma}_{\bullet-p}$ for all $\sigma\in K$.
\item
The homology groups $H_k(L^{\sigma}_{\bullet})=0$ for $k>0$.
\item
The homology class of the cycle $\varphi_p(\sigma)$ in $H_0(L^{\sigma})$
is zero for each $p$-simplex $\sigma$.
\end{enumerate}
Under the above assumptions, there exist homomorphisms $\theta_{p+q}:C_{p+q}(K)\to D_{q+1}$ $(q\geq 0)$, 
satisfying the following conditions:
\begin{enumerate}
\item[(a)]
$\delta\theta_{p+q}+\theta_{p+q-1}\delta=\varphi_{p+q}$ for $q\geq 0$. 
Here we set $\theta_{p-1}=0$.
\item[(b)]
$\theta_i(\sigma)\in L^{\sigma}_{\bullet}$.
\end{enumerate}
\end{proposition}
\begin{proof}
We construct maps $\theta_{p+q}$ inductively on $q$.
We consider the case where $q=0$.
Let $\sigma$ be a simplex of $K_p$.
Since the homology class of $\varphi_p(\sigma)$ in $H_0(L^{\sigma}_{\bullet})$
is zero, there exists an element $t_{\sigma}\in L^{\sigma}_{1}$ 
such that $\delta t_{\sigma}=\varphi(\sigma)$.
By setting $\theta_p(\sigma)=t_{\sigma}$, we have a map $\theta_p:C_p(K) \to D_1$.

We assume that $\theta_{p+q}$ is constructed and construct $\theta_{p+q+1}$.
Let $\sigma$ be a $(p+q+1)$-simplex of $K$ regarded as an element of $C_{p+q+1}(K)$.
Using the inductive assumption of the equality (a), we have
\begin{align*}
\delta(\varphi(\sigma)-\theta_{p+q}(\delta\sigma))
=&
\varphi(\delta\sigma)-\delta\theta_{p+q}(\delta\sigma)
\\
=&
\varphi(\delta\sigma)+\theta_{p+q-1}(\delta\delta\sigma)-\varphi(\delta\sigma)
=0.
\end{align*}
By the inductive assumption of (b) and the assumption (1), 
we have
$\theta_{p+q}(\delta\sigma)\in L^{\sigma}_{q+1}$.
Since we have $\varphi(\sigma)\in L^{\sigma}$
by the assumption (2), 
$\varphi(\sigma)-\theta_{p+q}(\delta\sigma)$ is a closed element
in $L_{q+1}^{\sigma}$. By the assumption (3), there exists an element 
$t_{\sigma}\in L_{p+q+2}^{\sigma}$
such that $\delta t_{\sigma}=\varphi(\sigma)-\theta_{p+q}(\delta\sigma)$.
We define a morphism $\theta_{p+q+1}$ to be $\theta_{p+q+1}(\sigma)=t_{\sigma}$
and the map $\theta_{p+q+1}$ satisfies conditions (a) and (b) for $q+1$.
\end{proof}

\subsection{Independence of ordering}
\label{subsec orgering}
Let $K$ be a finite simplicial complex, $L$ a full subcomplexes of $K$,
and $\co$ be a good ordering with respect to $L$.
We set
$$
W=\underset{\si\cap L=\emptyset}\cup \si.
$$
Let $p$ be a positive even integer, and let $T$ be a $p$-cocycle in $C^p(K,W)$ and
$
\varphi:C_{r}(K;\QQ) \to C_{r-p}(K;\QQ)
$ 
be the map
defined by $\varphi(\alpha)=T\overset{\co}\cap\alpha$
(\ref{def:simplicial cap product}).
Then the map $\varphi$ is a homomorphism of complexes, and its
image is contained in
$C_{\bullet-p}(L;\QQ)$.
Thus we have a homomorphism of complexes:
$$
\varphi :
C_{\bullet}(K;\QQ) \to C_{\bullet-p}(L;\QQ).
$$
Let $T'$ be a $p$-cocycle in $C^p(K,W)$ and set
$\varphi'(\alpha)=T'\overset{\co}\cap \alpha$.
If $T-T'$ is the coboundary of $w\in C^{\bullet}(K,W)$,
i.e. $dw=T-T'$, 
then we have
$$
\delta(w\overset{\co}\cap \alpha)-w\overset{\co}\cap \delta\alpha=-(\varphi -\varphi')(\al)
$$
for each $\al\in C_\bullet(K)$.
Therefore the homomorphism of homologies
$[\varphi]:H_{p+q}(K;\QQ) \to H_{q}(L;\QQ)
$
induced by $\varphi$ depends only on the cohomology class $[T]$ of $T$.

Let $K^*$ be a subcomplex of $K$ and set $L^*=K^*\cap L$.
By restricting the homomorphism $\varphi$ to $C_\bullet(K^\ast)$, 
we have a homomorphism of subcomplexes
$
C_{p+q}(K^*;\QQ) \to C_{q}(L^*;\QQ)
$
and a homomorphism of relative homologies
\begin{equation}
\label{cap prod on homology}
[\varphi]:H_{p+q}(K,K^*;\QQ) \to H_{q}(L,L^*;\QQ).
\end{equation}
This homomorphism also depends only on the cohomology class $[T]$
of $T$ in $H^p(K,W)$.


\begin{proposition}
\label{app prop indep of good ordering}
The homomorphism 
(\ref{cap prod on homology})
is independent of the ordering $\co$.
\end{proposition}
\begin{proof} The assertion follows from the following lemma.

\begin{lemma}
\label{existence homotopy changing ordering}
Let $T$ be a cocycle in $C^{p}(K,W)$ and
$\co$ and $\co'$ be good orderings of $K$ with respect to $L$.
Then there exists a map 
$\theta_{p+q}:C_{p+q}(K)\to C_{q+1}(L)$ $(q\geq 0)$ 
satisfying the following conditions:
\begin{enumerate}
\item
$
\delta\theta_{p+q}(x)+\theta_{p+q-1}(\delta x)=
T\overset{\co}\cap x-
T\overset{\co'}\cap x
$. for $q\geq 0$. Here we set $\theta_{p-1}=0$.
\item
$\theta_{p+q}(\sigma)\in C_{q+1}(L\cap \si)$ for each simplex $\sigma \in K_{p+q}$.
\end{enumerate}
\end{lemma}
\begin{proof}
We apply Proposition \ref{existence restricted homotopy} 
to $\varphi(x)=T\overset{\co}\cap x-
T\overset{\co'}\cap x$, $D=C_{\bullet}(L)$ and $L^{\sigma}_{\bullet}=C_{\bullet}(L\cap \si)$.
Conditions (1), (2) are easily verified.
Since the complex $L$ is a full subcomplex of $K$,
the intersection $L\cap \sigma$ is a face of $\sigma$,
and condition (3) is satisfied.
We check condition (4). Let 
$\sigma=[v_0,\dots, v_p]=\pm [v_0',\dots,v_p']$ be a $p$-simplex.
Here
$v_0<\cdots <v_p$ for the ordering $\co$ and
$v_0'<\cdots <v_p'$ for the ordering $\co'$.
Then we have
$
T\overset{\co}\cap \sigma=T(\sigma)[v_p]$ and
$T\overset{\co'}\cap \sigma=T(\sigma)[v_p']$
Since $[v_p]$ and $[v_p']$ are in the same homology
class in 
$H_{0}(L\cap \si)$, and (4) is proved.
Thus we have a map satisfying conditions (1) and (2)
in the lemma.
\end{proof}

\end{proof}
Since the homomorphism (\ref{cap prod on homology})
depends only on the choice of cohomology class $[T]$ of
$T$, it is written as $[T]\cap$.

\begin{proof}[Proof of Proposition \ref{prop face map first properties} (2)] Let $\gamma$ be an element in $AC_k(K,\bold D^n;\QQ)$.
By the admissibility condition for $\delta\gamma$, we have
$ L_1\cap |\delta\gamma| \subset (L_1\cap |\gamma|)^{(k-3)}\cup \bold D^n$,
where $( L_1\cap|\delta\gamma| )^{(k-3)}$ is the $(k-3)$-skeleton of 
$ L_1 \cap |\ga|$.
We denote the set $\bold D^n\cap |\ga|$ by $\bold D^n_\ga$ and the set $  H_{1,0}\cap \bold D^n\cap |\ga|$
by $\bold D^{n-1}_\ga$. We have a homomorphism
\begin{align}
\label{extract coefficients} 
T\overset{\co}\cap :H_k(|\gamma|,|\delta\gamma|\cup \bold D^n_\ga ;\QQ) \to &
H_{k-2}( L_1\cap |\gamma|,(L_1\cap |\delta\gamma|)\cup  \bold D^{n-1}_\ga ;\QQ) 
\\
\nonumber
\to & 
H_{k-2}(L_1\cap |\gamma|,(L_1\cap |\gamma|)^{(k-3)}\cup  \bold D^{n-1}_\ga ;\QQ) 
\\
\nonumber
\simeq & 
\underset{\substack{\tau \in L_1\cap |\gamma|,\,\, \tau\not\subset \bold D^{n-1} \\ \dim \tau=k-2}}
\bigoplus
 \QQ\tau.
\end{align}
The chain $T\overset{\co}{\cap} \ga$
is equal to the image of the homology class $[\gamma]$  of $\ga$
under the homomorphism 
(\ref{extract coefficients}).
The map (\ref{extract coefficients}) 
is independent of the choice of Thom cocycle $T$ and the ordering $\co$
by Proposition \ref{app prop indep of good ordering}.
\end{proof}

\subsection{Compatibility with the subdivision map}
\label{subsec subdivision}
Let $K$ be a simplicial complex, $L$ a full subcomplex of $K$,
and $K'$ a subdivision of $K$. 
The subdivision of $L$ induced by
$K'$ is denoted by $L'$.  We assume that $L'$ is a full subcomplex of $K'$. 
Then we have the following subdivision operators:
\begin{align*}
&\la:C_{\bullet}(K;\QQ)\to C_{\bullet}(K';\QQ), \\
&\la:C_{\bullet}(L;\QQ)\to C_{\bullet}(L';\QQ).
\end{align*}
Let $W'$ be the subcomplex of $K'$ defined as $\underset{\si'\in K',\,\si\cap L'=\emptyset}\bigcup \si'$
and let $T$ be a closed element in $C^{p}(K', W')$,
i.e. $T(\sigma')=0$ if $\si'\cap L'=\emptyset$.
Then the pull back $\lambda^*T$ is contained in $C^{p}(K, W)$ where $W= \underset{\si\in K,\, \si\cap L=\emptyset}\bigcup \si$ .
We choose a good ordering $\cal O$ resp. $\cal O'$ of $K$ resp. $K'$ with respect to $L$ resp. $L'$.
Then we have the following (generally non-commutative) diagram.
\begin{equation}
\label{chain level diagram subdivision}
\begin{matrix}
C_{\bullet}(K;\QQ) &\xrightarrow{\lambda^*T\overset{\co}\cap}& C_{\bullet -p}(L;\QQ) \\
\text{\scriptsize{$\la$}}\downarrow\phantom{\text{\scriptsize{$\la$}}}
 & & 
\phantom{\text{\scriptsize{$\la$}}}\downarrow\text{\scriptsize{$\la$}}
\\
C_{\bullet}(K';\QQ) &\xrightarrow{T\overset{\cal O'}\cap}& C_{\bullet -p}(L';\QQ).
\end{matrix}
\end{equation}
For a subcomplex $K^*$ of $K$, we set $L^*=K^*\cap L$. The subdivisions of $K^*$ and $L^*$
induced by $K'$ are denoted by ${K^{*}}'$ and ${L^{*}}'$, respectively.
The diagram (\ref{chain level diagram subdivision})
induces the following diagram for relative homologies
\begin{equation}
\label{diagram for subdivision relative cohomology}
\begin{matrix}
H_{p+q}(K,K^*;\QQ) &\xrightarrow{[\lambda^*T]\cap}& H_{q}(L,L^*;\QQ) \\
\text{\scriptsize{$\la$}}\downarrow\phantom{\text{\scriptsize{$\la$}}}
 & & 
\phantom{\text{\scriptsize{$\la$}}}\downarrow\text{\scriptsize{$\la$}}
\\
H_{p+q}(K',{K^*}';\QQ) &\xrightarrow{[T]\cap}& H_{q}(L',{L^*}';\QQ).
\end{matrix}
\end{equation}
\begin{proposition}
\label{compat with subdivision}
The diagram (\ref{diagram for subdivision relative cohomology})
is commutative.
\end{proposition}
\begin{proof} The assertion follows from the following lemma.
\begin{lemma}
\label{lem compat subdiv}
Consider the following  two homomorphism of complexes
$$
\la\circ(\lambda^*T\overset{\co}\cap\ ) \quad \text{ and } \quad
(T\overset{\cal O'}\cap\ )\circ\la
:
C_{\bullet}(K) \xrightarrow{} C_{\bullet -p}(L')
$$
Then there exist maps $\theta_{p+q}:C_{p+q}(K) \to C_{q+1}(L')$
such that
\begin{enumerate}
\item
$
\delta\theta_{p+q}(x)+\theta_{p+q-1}(\delta x)=
\la(\lambda^*T\overset{\co}\cap\ x)-
T\overset{\cal O'}\cap \la(x)
$, $(q\geq 0$, $\theta_{p-1}=0)$ and
\item
$\theta_{p+q}(\sigma)\in C_{q+1}(L'\cap \si)$ for each simplex $\sigma \in K_{p+q}$.
\end{enumerate}
\end{lemma}
\begin{proof}
We apply Proposition \ref{existence restricted homotopy}
to the case where  $\varphi(x)=
\la(\lambda^*T\overset{\co}\cap\ x)-
T\overset{\cal O'}\cap \la(x)$, $D_\bullet=C_\bullet(L')$  and $L^{\sigma}_{\bullet}=
C_{\bullet}(L'\cap \si)$.
Conditions (1) and (2)  are easily verified. Since   $L$ is a full subcomplex of $K$,
for each $\si\in K$, the intersection $L\cap \sigma$ is a face of $\sigma$. The complex $L'\cap \si$
is a subdivision of $L\cap \si$, and so the condition (3) is satisfied.
We claim that the condition (4) of 
Proposition \ref{existence restricted homotopy}
holds for $\varphi$.
Let $\si=[v_0, \dots, v_p]\in K_p$ and set $\la\sigma=\sum_j\sigma_j
=\sum_j\pm [w_0^j, \dots, w_p^j]$. Here we assume that $v_0<\cdots <v_p$ resp. $w_0^j<\cdots <w_p^j$
under $\cal O$ resp. $\cal O'$ for each $j$. 
Then we have
$$
\la(\lambda^*T\overset{\co}\cap \sigma)=\sum_jT(\sigma_j)[v_p]
$$
and
$$
T\overset{\cal O'}\cap(\la\sigma)=\sum_jT(\sigma_j)[w^j_p].
$$
Since $ L\cap \si $ is
a simplex of $\sigma$, and $[v_p]$ and $[w^j_p]$
define the same homology class in $H_0(L^{\sigma}_{\bullet})$.
Thus condition (4) is satisfied.
\end{proof}

\end{proof}

\begin{proof}[Proof of Proposition \ref{prop face map first properties} (3)]  Let $\gamma$ be an element in 
$AC_{k}(K,\bold D^n;\QQ)$.
The homology class of $\ga$ in 
$H_k(|\gamma|,|\delta\gamma|;\QQ)$.
 is denoted by $[\gamma]$.
We set 
$|\gamma|'=K'\cap |\gamma|$ and
$|\delta\gamma|'=K'\cap |\delta\gamma|$.
The element $\lambda(\gamma)$ defines a class $[\lambda(\gamma)]$ in
$H_k(|\gamma|',|\delta\gamma|';\QQ)$. 
Let $T'\in C^2(K',W')$ be a Thom cocycle of $  H_{1,0}$.  Then $\la^*T'\in C^2(K,W)$ is also a Thom cocycle
of $  H_{1,0}$.  As in the proof of Proposition\ref{prop face map first properties} (2), for a chain $\ga\in AC_k(K,\bold D^n;\QQ)$, 
the chain $\la^*T\overset{\cal O}\cap \ga $ resp. $T\overset{\cal O'}\cap \la(\ga)$ is determined
by its class 
$$\la^*T\cap[\ga] \in H_{k-2}(L_1\cap |\gamma|,(L_1\cap |\gamma|)^{(k-3)}\cup  \bold D^{n-1}_\ga;\QQ)$$
resp. 
$$T\cap[\la(\ga)] \in H_{k-2}(L_1'\cap |\gamma|',(L'_1\cap |\gamma|')^{(k-3)}\cup  \bold D^{n-1}_\ga;\QQ)$$
Since these classes are independent of the orderings by Proposition \ref{prop face map first properties} (2), we forget them.  By
Proposition \ref{compat with subdivision}
applied to $K=|\ga|$, $K^*=|\delta\ga|\cup \bold D^n_\ga$ and $L=L_1$, 
we have the following commutative diagram.
\begin{equation}
\label{commutativity cohomologcal subdivision}
\begin{matrix}
H_k(|\gamma|,|\delta\gamma|\cup \bold D^n_\ga;\QQ) &\xrightarrow{\la^*T'\cap } &
H_{k-2}(L_1\cap |\gamma|,(L_1\cap |\gamma|)^{(k-3)}\cup  \bold D^{n-1}_\ga; \QQ) \\
\lambda \downarrow \phantom{\lambda}& & \phantom{\lambda}\downarrow\lambda \\
H_k(|\gamma|',|\delta\gamma|'\cup \bold D^n_\ga;\QQ) &\xrightarrow{T'\cap } &
H_{k-2}(L'_1\cap |\gamma|',(L'_1\cap |\gamma|')^{(k-3)}\cup  \bold D^{n-1}_\ga; \QQ)
\\
\end{matrix}
\end{equation}
The assertion follows from this.
\end{proof}

\subsection{Relations between the cap product and the cup product}
\label{subsec cup}
Let $K$ be a finite simplicial complex and
$L_1,L_2$ be subcomplexes in $K$.
Assume that $L_1, L_2$ and $L_1\cup L_2$ are full subcomplexes of $K$.
Let $\cal O$ be a good ordering with respect to $L_1$ and
$L_{12}=L_1\cap L_2$. We set
$$
W_i=\underset{\si\cap L_i=\emptyset}\cup \si.
$$
\begin{lemma}
\label{lem union}
Under  the above assumtions, we have  $\dis W_1\cup W_2=
\underset{\si\cap L_{12}=\emptyset}\cup \si$.
\end{lemma}
\begin{proof}
We will show that $W_1\cup W_2\supset \underset{\si\cap L_{12}=\emptyset}\cup \si$. Suppose that
$\si\cap L_{12}=\emptyset$. 
Since $L_1\cup L_2$ is a full subcomplex of $K$, $\si\cap (L_1\cup L_2)$ is a face of $\si$ which we denote by $\tau$.
If $\tau \in L_1$, then $\tau\cap L_2=\emptyset.$
\end{proof}

\begin{definition}[Cup product]
\label{cap and cup product}
For $T_1\in C^{p}(K,W_1)$ and
$T_2\in C^{q}(K,W_2)$, we define the cup product $T_1\overset{\cal O}\cup T_2\in C^{p+q}(K)$
by
$$
(T_1\overset{\co}\cup T_2)(\sigma)=T_1(v_0, \dots, v_p)T_2(v_p,\dots, v_{p+q})
$$
where $\sigma=[v_0, \dots, v_{p+q}]$
with $v_0<\cdots <v_{p+q}$. The cup product induces a homomorphism
of complexes:
$$
\overset{\co}\cup :C^{\bullet}(K,W_1)\otimes C^{\bullet}(K,W_2)\to C^{\bullet}(K).
$$
\end{definition}
Let $K^*$ be a subcomplex of $K$ and set $L_i^*=K^*\cap L_i$ and $L_{12}^*=K^*\cap L_{12}$.
The proof of the following proposition is obvious and is omitted. 
\begin{proposition}
\label{cup and cap projection formula for simplicial}
\begin{enumerate}
\item 
The restriction of the cup product $T_1\cup T_2$ to $W_1\cup W_2$ vanishes.
\item
Let $T_1$ and $T_2$ be closed elements in $C^p(K,W_1)$ and 
$C^q(K,W_2)$
and set $T_{12}=T_1\overset{\co}\cup T_2$. Then the composite of the homomorphisms
\begin{align*}
T_2\overset{\cal O}\cap (T_1\overset{\cal O}\cap *):\,C_{p+q+r}(K,K^*;\QQ)&\xrightarrow{T_1\overset{\co}{\cap}} 
C_{q+r}(L_{1}, L_{1}^*;\QQ) \\
&\xrightarrow{T_2\overset{\co}{\cap}} 
C_r(L_{12}, L_{12}^*;
\QQ)
\end{align*}
is equal to the homomorphism $T_{12}\overset{\co}\cap$.
\end{enumerate}
\end{proposition}
As a consequence the composite of the following morphisms coincides with
the cap product with $[T_{12}]$.
\begin{align*}
H_{p+q+r}(K,K^*;\QQ) 
&\xrightarrow{[T_1]\cap} H_{q+r}(L_1,L_1^*;\QQ)
\\
&\xrightarrow{[T_2]\cap} H_{r}(L_{12},L_{12}^*;\QQ)
\end{align*}

\begin{proof} [Proof of Proposition \ref{total face map is a differential}]
Let $K$ be a good triangulation of $P^n$, and let $\ga$ be an element of $AC_{p}(K, \bold D^n; \QQ)$.
We set $H_1=  H_{1,0}$, $H_2=H_{2,0}$, 
$H_{12}=H_1\cap H_2$, $L_1=K\cap  H_{1,0}$,
$L_2=K\cap H_2$
$L_{12}=K\cap H_{12}$, $W_i=\underset{\si\in K,\,\si\cap H_i=\emptyset}\cup \si$ ($i=1,2$) and
$ W_{12}=\underset{\si\in K,\,\si\cap H_{12}=\emptyset}\cup \si$. By Lemma \ref{lem union}
we have $W_{12}=W_1\cup W_2$. 
The face map $\partial_{i,0}$ $(i=1,2)$ is denoted by $\partial_i$.
Considering the symmetry on $H_{i,\alpha}$ $(1\leq i\leq n, \alpha=0,\infty)$,
it is enough to prove the commutativity of the following diagram
\begin{equation}
\label{face map is a cubic differential commutativity}
\begin{matrix}
AC_{p}(K,\bold D^n;\QQ) 
& \xrightarrow{\partial_1} &  
AC_{p-2}(L_1,\bold D^{n-1};\QQ) 
\\
\partial_2\downarrow \phantom{\partial_2}& & 
\phantom{\partial_2}\downarrow \partial_2
\\
AC_{p-2}(L_2,\bold D^{n-1};\QQ) 
& \xrightarrow{\partial_1} &  
AC_{p-4}(L_{12},\bold D^{n-2};\QQ). 
\end{matrix}
\qquad 
\qquad 
\end{equation}
We denote the complex $\bold D^n\cap |\ga|$ by $\bold D^n_\ga$. Let $T_1
\in C^2(K,W_1;\QQ)$ resp. $T_2\in C^2(K,W_2;\QQ)$ be a Thom cocycle
of the face $H_1$ resp. $H_2$. We have the equality 
$$[T_1]\cup [T_2]=[T_2]\cup [T_1]$$
in $H^4(K, W_{12};\QQ)$.
By Proposition \ref{cup and cap projection formula for simplicial}
applied to $K=|\ga|$ and $K^*=|\delta \ga|  \cup \bold D^n_\ga$, we have a commutative diagram
\[
\begin{matrix}
H_p(|\gamma|,|\delta\gamma| \cup \bold D^n_\ga ;\QQ)
& \xrightarrow{[T_1]\cap} &  
H_{p-2}(L_1\cap |\gamma|,L_1\cap (|\delta\gamma| \cup \bold D^n_\ga);\QQ) 
\\
[T_2]\cap\downarrow \phantom{[T_2]\cap}& & 
\phantom{[T_1]\cap}\downarrow [T_2]\cap
\\
H_{p-2}(L_2\cap |\gamma|,L_2\cap (|\delta\gamma|\cup \bold D^n_\ga) ;\QQ) 
& \xrightarrow{[T_1]\cap} &  
H_{p-4}(L_{12}\cap |\gamma|,L_{12}\cap (|\delta\gamma|\cup \bold D^n_\ga) ;\QQ) 
\end{matrix}
\]


The chain $\ga$ defines a class $[\ga]$ in $H_p(|\gamma|,|\delta\gamma| \cup \bold D^n_\ga ;\QQ)$.
By the admissibility condition,
the complex $L_{12}\cap ( |\delta\gamma| \cup \bold D^n_\ga)$
is contained in 
 $(L_{12}\cap |\delta\gamma| )^{(p-5)}\cup (L_{12}\cap \bold D^n_\ga)$
where $(L_{12}\cap |\delta\gamma| )^{(p-5)}$ denotes  the $(p-5)$-skeleton 
of $ L_{12}\cap |\delta\gamma|$.  The chain  $T_2\cap(T_1\cap \ga)$ resp.  $T_1\cap(T_2\cap \ga)$ is determined by its
class $[T_2]\cap([T_1]\cap [\ga])$ resp.  $[T_1]\cap([T_2]\cap [\ga])$ in $H_{p-4}(L_{12}\cap |\gamma|,  (L_{12}\cap |\delta\gamma|)^{(p-5)}\cup (L_{12}\cap \bold D^n_\ga);\QQ)$.
Hence we have the equality 
$\partial_2\partial_1(\ga)=\partial_1\partial_2(\ga)$.
\end{proof}

\vskip 0.5cm

\noindent Masaki Hanamura

\vskip 0.1cm 

\noindent Department of Mathematics,

\noindent Tohoku University,

\noindent 6-3, Aramaki Aza-Aoba, 

\noindent Aoba-ku, Sendai 980-8578,

\noindent Japan.

\vskip 0.5cm

\noindent Kenichiro Kimura

\vskip 0.1cm 

\noindent Department of Mathematics,

\noindent University of Tsukuba,

\noindent 1-1-1 Tennodai,

\noindent Tsukuba, Ibaraki,

\noindent 305-8571

\noindent Japan.

\noindent kimurak@math.tsukuba.ac.jp

\vskip 0.5cm

\noindent Tomohide Terasoma

\vskip 0.1cm 

\noindent Faculty of Engineering and Sciences,

\noindent Hosei University,

\noindent 3-7-2, Kajinocho, Koganeishi,

\noindent Tokyo, 184-8584, Japan.

\printindex

\end{document}